\documentclass{article}
\usepackage{amsfonts}
\usepackage{amsmath}
\usepackage{graphicx}

\newtheorem{theorem}{Theorem}

\newtheorem{definition}[theorem]{Definition}
\newtheorem{example}[theorem]{Example}

\newtheorem{proposition}[theorem]{Proposition}
\newtheorem{remark}[theorem]{Remark}

\newenvironment{proof}[1][Proof]{\noindent \textbf{#1.} }{\  \rule{0.5em}{0.5em}}
\input{tcilatex}
\begin{document}

\title{\textbf{Restricted distance-type Gaussian estimators based on density
power divergence and their applications in hypothesis testing}}
\date{}
\author{A. Felipe, M. Jaenada$,$ P. Miranda and L. Pardo \\
{\small Department of Statistics and O.R., Complutense University of Madrid,
Spain}}
\date{}
\maketitle

\begin{abstract}
Zhang (2019) presented a general estimation approach based on the Gaussian
distribution for general parametric models where the likelihood of the data
is difficult to obtain or unknown, but the mean and variance-covariance
matrix are known. Castilla and Zografos (2021) extended the method to
density power divergence-based estimators, which are more robust than the
likelihood-based Gaussian estimator against data contamination. In this
paper we introduce the restricted minimum density power divergence Gaussian
estimator (MDPDGE) and study its main asymptotic properties. Also, we examine it robustness through
its  influence function analysis. Restricted
estimators are required in many practical situations, in special in testing
composite null hypothesis, and provide here constrained estimators to
inherent restrictions of the underlying distribution. Further, we derive
robust Rao-type test statistics based on the MDPDGE for testing simple null hypothesis and we deduce explicit expressions for some main
important distributions. Finally, we empirically evaluate the efficiency and
robustness of the method through a simulation study.
\end{abstract}



\section{Introduction}

Let $\boldsymbol{Y}_{1},...,$ $\boldsymbol{Y}_{n}$ be independent and
identically distributed observations from a $m$-dimensional random vector $%
\boldsymbol{Y}$ with probability density function $f_{\mathbf{\theta }}(%
\boldsymbol{y}\mathbf{),}$ $\boldsymbol{\theta }\in \Theta \subset \mathbb{R}%
^{d}.$ We denote, 
\begin{equation}
E_{\boldsymbol{\theta }}\left[ \boldsymbol{Y}\right] =\boldsymbol{\mu }%
\left( \boldsymbol{\theta }\right) \text{ and }Cov_{\boldsymbol{\theta }}%
\left[ \boldsymbol{Y}\right] =\boldsymbol{\Sigma }\left( \boldsymbol{\theta }%
\right) .  \label{100}
\end{equation}%
The log-likelihood function of the assumed model is given by
\begin{equation*}
l\left( \boldsymbol{\theta }\right) =\tsum \limits_{i=1}^{n}\log f_{\mathbf{%
\theta }}(\boldsymbol{y}_{i}\mathbf{)}
\end{equation*}%
for $\boldsymbol{y}_{1},...,$ $\boldsymbol{y}_{n}$ observations of the $m$%
-dimensional random vectors $\mathbf{Y}_{1},...,$ $\mathbf{Y}_{n}.$ 
Then, the the maximum likelihood estimator (MLE) is computed as
\begin{equation} \label{eq:loglikelihood}
\widehat{\boldsymbol{\theta}}_{\text{MLE}} = \operatorname{max}_{\boldsymbol{\theta} \in \Theta} l\left( \boldsymbol{\theta }\right).
\end{equation}
 In many real life situations the underlying density function, $f_{\theta}(\cdot),$ is unknown or it computation is quite difficult but contrariwise, the mean vector and variance-covariance matrices of the underlying distribution of the data, namely $\boldsymbol{\mu}(\boldsymbol{\theta }) $ and $\boldsymbol{\Sigma }\left( \boldsymbol{\theta}\right),$ are known.
 
 In this case, Zhang (2019) proposed a general procedure based on the Gaussian distribution for estimating the model parameter vector $\boldsymbol{\theta}.$  
 Zhang (2019) assumed that the $m-$dimensional random vector $\boldsymbol{Y}$ came from a multidimensional normal distribution  with vector mean $\boldsymbol{\mu }\left( \boldsymbol{\theta }\right) $ and variance-covariance matrix $\boldsymbol{\Sigma }\left( \boldsymbol{\theta }\right).$ 
 From a statistical point of view this procedure can be justified on
the basis of the maximum-entropy principle, (see Kapur (1989)), as the multidimensional normal distribution  has maximum uncertainty in terms of Shannon entropy and is as well consistent with the given information, vector mean
and variance-covariance matrix.

 Then, an estimator of the model parameter $\boldsymbol{\theta}$ based on the Gaussian distribution can be obtained by maximizing the log-likelihood function as defined in (\ref{eq:loglikelihood}), but using $f_{\boldsymbol{\theta}}(\cdot)$  the probability density function of a normal distribution with known mean $\boldsymbol{\mu}(\boldsymbol{\theta }) $ and variance-covariance matrix $\boldsymbol{\Sigma }\left( \boldsymbol{\theta }\right),$ corresponding to the true mean and variance-covariance matrix of the underlying distribution.
That is, the Gaussian-based likelihood function of $\boldsymbol{\theta}$ is  given by
\begin{equation}
l_{G}\left( \boldsymbol{\theta }\right) =-\frac{nm}{2}\log 2\pi -\frac{n}{2}%
\log \left \vert \boldsymbol{\Sigma }\left( \boldsymbol{\theta }\right)
\right \vert -\frac{1}{2}\tsum \limits_{i=1}^{n}\left( \boldsymbol{y}_{i}-%
\boldsymbol{\mu }\left( \boldsymbol{\theta }\right) \right) ^{T}\boldsymbol{%
\Sigma }\left( \boldsymbol{\theta }\right) ^{-1}\left( \boldsymbol{y}_{i}-%
\boldsymbol{\mu }\left( \boldsymbol{\theta }\right) \right)  \label{101}
\end{equation}%
for any $\boldsymbol{y}_{1},...,\boldsymbol{y}_{n}$ independent observations of the population $\boldsymbol{Y},$
and the Gaussian MLE of $\boldsymbol{\theta }$ is defined by
\begin{equation*}
\widehat{\boldsymbol{\theta}}_{G}=\arg \max_{\boldsymbol{\theta }\in \Theta
}l_{G}\left( \boldsymbol{\theta}\right) .
\end{equation*}

The Gaussian estimator is a MLE and so inherit all good properties of the likelihood estimators. 
It works well in terms of the asymptotic efficiency but it has important robustness problems. That is, in the absence of contamination in data, the MLE consistently estimates the true value of the model parameter, but in counterpart it may get quite heavily affected by outlying observations in the data.
For this reason, Castilla and Zografos (2022) extended the concept of Gaussian estimator and defined a robust version of the estimator based on the density power divergence (DPD) introduced in Basu et al (1998). The DPD robustly quantifies the statistical difference between two distributions and it has been widely used for developing robust inferential methods in many different statistical models. 
Given a set of observations, the robust minimum DPD estimator (MDPDE) is computed as the minimizer of the DPD between the assumed model distribution and the empirical distribution of the data. The MDPDE enjoys good asymptotic properties and produces robust estimators under general statistical models, as discussed later.
The minimum density power divergence Gaussian estimator (MDPDGE) of the parameter $\boldsymbol{\theta},$ is defined for $\tau \geq 0$ as
\begin{equation}
\widehat{\boldsymbol{\theta }}_{G}^{\tau }=\arg \max_{\boldsymbol{\theta }%
\in \Theta \subset \mathbb{R}^{d}}H_{n}^{\tau }\left( \boldsymbol{\theta }%
\right)  \label{102}
\end{equation}%
where
\begin{equation} 
\label{eq:Hn}
\begin{aligned}
H_{n}^{\tau }\left( \boldsymbol{\theta }\right) =&\frac{\tau +1}{\tau
\left( 2\pi \right) ^{m\tau /2}\left \vert \boldsymbol{\Sigma }\left( 
\boldsymbol{\theta }\right) \right \vert ^{\tau /2}}\frac{1}{n} 
 \left[ \tsum \limits_{i=1}^{n}\exp \left \{ -\frac{\tau }{2}\left( 
\boldsymbol{y}_{i}-\boldsymbol{\mu }\left( \boldsymbol{\theta }\right)
\right) ^{T}\boldsymbol{\Sigma }\left( \boldsymbol{\theta }\right)
^{-1}\left( \boldsymbol{y}_{i}-\boldsymbol{\mu }\left( \boldsymbol{\theta }%
\right) \right) \right \} \right.  \\
& \hspace{0.3cm} \left. -\frac{\tau }{\left( 1+\tau \right) ^{(m/2)+1}}\right] -\frac{1}{%
\tau }  \\
=&a\left \vert \boldsymbol{\Sigma }\left( \boldsymbol{\theta }\right)
\right \vert ^{-\frac{\tau }{2}}\frac{1}{n}\left[ \tsum
\limits_{i=1}^{n}\exp \left \{ -\frac{\tau }{2}\left( \boldsymbol{y}_{i}-%
\boldsymbol{\mu }\left( \boldsymbol{\theta }\right) \right) ^{T}\boldsymbol{%
\Sigma }\left( \boldsymbol{\theta }\right) ^{-1}\left( \boldsymbol{y}_{i}-%
\boldsymbol{\mu }\left( \boldsymbol{\theta }\right) \right) \right \} -b%
\right] -\frac{1}{\tau }, 
\end{aligned}
\end{equation}
and
\begin{equation}
a=\frac{\tau +1}{\tau \left( 2\pi \right) ^{m\tau /2}}\text{ \ and }b=\frac{%
\tau }{\left( 1+\tau \right) ^{(m/2)+1}}.  \label{103b}
\end{equation}

The  MDPDGE family is indexed by a tuning parameter $\tau$ controlling the trade-off between robustness and efficiency; the greater value of $\tau$, the more robust the resulting estimator is but less efficiency. It has been shown in the literature that values of the tuning parameter above 1 do not provide sufficiently efficient estimators and so, the tuning parameter would be  chosen in the $[0,1]$ interval.
Further, at $\tau = 0$ the MDPDGE reduces to the Gaussian estimator of Zhang (2019),
\begin{equation*}
\widehat{\boldsymbol{\theta}}_{G}=\arg \max_{\boldsymbol{\theta }\in \Theta
\subset \mathbb{R}^{d}}H_{n}^{0}\left( \boldsymbol{\theta }\right)
\end{equation*}%
with 
\begin{equation} \label{eq:KL}
H_{n}^{0}\left( \boldsymbol{\theta }\right) =\lim_{\tau \rightarrow
0}H_{n}^{\tau }\left( \boldsymbol{\theta }\right) =-\frac{n}{2}\log \left
\vert \boldsymbol{\Sigma }\left( \boldsymbol{\theta }\right) \right \vert -%
\frac{1}{2}\tsum \limits_{i=1}^{n}\left( \boldsymbol{y}_{i}-\boldsymbol{\mu }%
\left( \boldsymbol{\theta }\right) \right) ^{T}\boldsymbol{\Sigma }\left( 
\boldsymbol{\theta }\right) ^{-1}\left( \boldsymbol{y}_{i}-\boldsymbol{\mu }%
\left( \boldsymbol{\theta }\right) \right) .
\end{equation}

Note the above objective function does not perfectly match with the likelihood function of the model stated in (\ref{eq:loglikelihood}) as it lacks the first term of the likelihood. However, this term does not depend on the parameter $\boldsymbol{\theta}$ and thus both loss functions will lead to the same maximizer.
Indeed, the loss in Equation (\ref{eq:KL}) corresponds to the Kullback-Leiber divergence between the assumed normal distribution and the empirical distribution of the data, which justifies the MLE from the information theory. The Kullback-Leiber divergence is the limiting divergence on the DPD family at $\tau=0$ and so the MDPDGE is a generalization of the classical Gaussian estimator with a tuning parameter controlling the compromise between efficiency and robustness.

Further, the MDPDGE is consistent and the asymptotically normal, that is,
given $\boldsymbol{Y}_{1},...,$ $\boldsymbol{Y}_{n}$  independent and identically
distributed vectors from the $m$-dimensional random vector $\boldsymbol{Y}$,
the MDPDGE, $\widehat{\boldsymbol{\theta }}_{G}^{\tau },$ defined in (\ref{102}) satisfies
\begin{equation}
\sqrt{n}\left( \widehat{\boldsymbol{\theta }}_{G}^{\tau }-\boldsymbol{\theta 
}\right) \underset{n\longrightarrow \infty }{\overset{\mathcal{L}}{%
\longrightarrow }}\mathcal{N}(\boldsymbol{0}_{d},\boldsymbol{J}_{\boldsymbol{%
\tau }}(\boldsymbol{\theta })^{-1}\boldsymbol{K}_{\boldsymbol{\tau }}(%
\boldsymbol{\theta })\boldsymbol{J}_{\boldsymbol{\tau }}(\boldsymbol{\theta }%
)^{-1})  \label{103a}
\end{equation}%
being 
\begin{equation*}
\boldsymbol{J}_{\boldsymbol{\tau }}(\boldsymbol{\theta })=\left( J_{\tau
}^{ij}\left( \boldsymbol{\theta }\right) \right) _{i,j=1,..,,d}\text{ and }%
\boldsymbol{K}_{\boldsymbol{\tau }}(\boldsymbol{\theta })=\left( K_{\tau
}^{ij}\left( \boldsymbol{\theta }\right) \right) _{i,j=1,..,,d}.
\end{equation*}%
and the elements $J_{\tau }^{ij}\left( \boldsymbol{\theta }\right) $ and $%
K_{\tau }^{ij}\left( \boldsymbol{\theta }\right) $ of the matrices $J_{\tau
}\left( \boldsymbol{\theta }\right) $ and $K_{\tau }\left( \boldsymbol{%
\theta }\right) $ are given by
\begin{eqnarray}
J_{\tau }^{ij}\left( \boldsymbol{\theta }\right) &=&\left( \frac{1}{\left(
2\pi \right) ^{m/2}\left \vert \boldsymbol{\Sigma }\left( \boldsymbol{\theta 
}\right) \right \vert ^{1/2}}\right) ^{\tau }\frac{1}{\left( 1+\tau \right)
^{(m/2)+2}}  \label{103aa} \\
&&\left[ \left( \tau +1\right) trace\left( \boldsymbol{\Sigma }\left( 
\boldsymbol{\theta }\right) ^{-1}\frac{\partial \boldsymbol{\mu }\left( 
\boldsymbol{\theta }\right) }{\partial \boldsymbol{\theta }}\left( \frac{%
\partial \boldsymbol{\mu }\left( \boldsymbol{\theta }\right) }{\partial 
\boldsymbol{\theta }}\right) ^{T}\right) \right.  \notag \\
&&\left. +\Delta _{\tau }^{i}\Delta _{\tau }^{j}+\frac{1}{2}trace\left( 
\boldsymbol{\Sigma }\left( \boldsymbol{\theta }\right) ^{-1}\dfrac{\partial 
\boldsymbol{\Sigma }\left( \boldsymbol{\theta }\right) }{\partial \theta _{j}%
}\boldsymbol{\Sigma }\left( \boldsymbol{\theta }\right) ^{-1}\dfrac{\partial 
\boldsymbol{\Sigma }\left( \boldsymbol{\theta }\right) }{\partial \theta _{i}%
}\right) \right]  \notag
\end{eqnarray}%
and 
\begin{eqnarray}
K_{\tau }^{ij}\left( \boldsymbol{\theta }\right) &=&\left( \frac{1}{\left(
2\pi \right) ^{m/2}\left \vert \boldsymbol{\Sigma }\left( \boldsymbol{\theta 
}\right) \right \vert ^{1/2}}\right) ^{2\tau }\frac{1}{\left( 1+2\tau
\right) ^{(m/2)+2}}\left[ \Delta _{2\tau }^{i}\Delta _{2\tau }^{j}\right.
\label{103aaa} \\
&&\left. +\left( 1+2\tau \right) trace\left( \boldsymbol{\Sigma }\left( 
\boldsymbol{\theta }\right) ^{-1}\frac{\partial \boldsymbol{\mu }\left( 
\boldsymbol{\theta }\right) }{\partial \boldsymbol{\theta }}\left( \frac{%
\partial \boldsymbol{\mu }\left( \boldsymbol{\theta }\right) }{\partial 
\boldsymbol{\theta }}\right) ^{T}\right) \right.  \notag \\
&&\left. +\frac{1}{2}trace\left( \boldsymbol{\Sigma }\left( \boldsymbol{%
\theta }\right) ^{-1}\dfrac{\partial \boldsymbol{\Sigma }\left( \boldsymbol{%
\theta }\right) }{\partial \theta _{j}}\boldsymbol{\Sigma }\left( 
\boldsymbol{\theta }\right) ^{-1}\dfrac{\partial \boldsymbol{\Sigma }\left( 
\boldsymbol{\theta }\right) }{\partial \theta _{i}}\right) \right]  \notag \\
&&-\left( \frac{1}{\left( 2\pi \right) ^{m/2}\left \vert \boldsymbol{\Sigma }%
\left( \boldsymbol{\theta }\right) \right \vert ^{1/2}}\right) ^{2\tau }%
\frac{1}{\left( 1+\tau \right) ^{m+2}}\Delta _{\tau }^{i}\Delta _{\tau }^{j},
\notag
\end{eqnarray}%
with $\Delta _{\tau }^{i}=\frac{\tau }{2}trace\left( \boldsymbol{\Sigma }%
\left( \boldsymbol{\theta }\right) ^{-1}\dfrac{\partial \boldsymbol{\Sigma }%
\left( \boldsymbol{\theta }\right) }{\partial \theta _{i}}\right) .$

The above asymptotic distribution follows from Theorem 2 in Basu et al. (2018b), after explicitly compute the matrices 
$$\boldsymbol{J}_{\tau }\left( \boldsymbol{\theta }\right)$$
and $\boldsymbol{K}_{\tau }\left( \boldsymbol{\theta }\right)$ there defined for general statistical models.

Additionally, in some situations we may have  additional knowledge about the true
parameter space. For example, the parameter of the exponential or Poisson models must be always positive.
Then, the  parameter space $\Theta$ should be restricted constraints. Here, we will consider restricted parameter spaces of the form
\begin{equation}
\Theta_0 = \left \{ \boldsymbol{\theta \in }\Theta /\text{ }\boldsymbol{g}(\boldsymbol{%
\theta )=0}_{r}\right \} ,  \label{104}
\end{equation}%
where $\boldsymbol{g}:\mathbb{R}^{d}\rightarrow \mathbb{R}^{r}$ is a
vector-valued function mapping such that the $d\times r$ matrix 
\begin{equation}
\mathbf{G}\left( \boldsymbol{\theta }\right) =\frac{\partial \boldsymbol{g}%
^{T}(\boldsymbol{\theta )}}{\partial \boldsymbol{\theta }}  \label{105}
\end{equation}%
exists and is continuous in $\boldsymbol{\theta }$ and rank$\left( \mathbf{G}%
\left( \boldsymbol{\theta }\right) \right) =r$, and $\boldsymbol{0}_{r}$
denotes the null vector of dimension $r$. 
The notation  $\Theta _{0}$ clues the use of the present restricted estimator for defining test statistics under composite null hypothesis.

The most popular estimator of $\boldsymbol{\theta }$ satisfying the constraints
in (\ref{104}) is the restricted MLE (RMLE), naturally defined as the
maximizer of the loglikelihood function of the model, but subject to the parameter space restrictions $\boldsymbol{g}(\boldsymbol{\theta )}=\boldsymbol{0}_{r}$ (see Silvey, 1975). Unlikely, the RMLE has the same robustness problems than the MLE, and so robust alternatives should be adopted in the presence of contamination in data.
Several robust restricted estimators have been considered in the statistical literature to overcome the robustness drawback of the RMLE. For example,
Pardo et al. (2002) introduced the restricted minimum Phi-divergence estimator  and studied its properties. Basu et al (2018b) presented the restricted minimum
density power divergence estimators (RMDPDE) and studied some applications
of them in testing hypothesis. In Ghosh (2015) the theoretical robustness
properties of the RMDPDE were studied.  Jaenada et al (2022a,
2022b) considered the restricted R\'{e}nyi pseudodistance estimator and from it, they defined Rao-type tests. More recently, Martin (2021) studied the RMDPD under normal distributions and develop independence test under the normal assumption, and later Martin (2023) used the RMDPDE in the context of independent but not identically distributed variables under heterocedastic linear regression models.
In this paper, we introduce and study the restricted minimum density power
divergence Gaussian estimator (RMDPDGE).

The rest of the paper is organized as follows: In Section 2 we introduce the RMDPDGE and we obtain its asymptotic distribution.  Section 3 presents the influence function of the RMDPDGE and theoretically proves the robustness of the proposed estimators. 
Some statistical applications for testing are presented in Section 4 and explicit expression of the Rao-type test statistics based on the RMDPDGEs under exponential and Poisson models are given.
Section 5  empirically demonstrates the robustness of the method trough a  simulation study, and the  advantages and disadvantage of the Gaussian assumption are discussed there.
Section 6 presents some conclusions.
The proofs of the main results stated in the paper are included in an Appendix

\section{Restricted minimum density power divergence Gaussian estimators}

In this section we present the RMDPDGE under general equality non-linear constraints and we study it asymptotic distribution, showing the consistency of the estimator.

\begin{definition}
Let $\boldsymbol{Y}_{1},...,$ $\boldsymbol{Y}_{n}$ be independent and
identically distributed observations from a $m$-dimensional random vector $%
\boldsymbol{Y}$ with $E_{\boldsymbol{\theta }}\left[ \boldsymbol{Y}\right] =%
\boldsymbol{\mu }\left( \boldsymbol{\theta }\right) $ and $Cov_{\boldsymbol{%
\theta }}\left[ \boldsymbol{Y}\right] =\boldsymbol{\Sigma }\left( 
\boldsymbol{\theta }\right) $, $\boldsymbol{\theta }\in \Theta \subset 
\mathbb{R}^{d}$. The RMDPDGE, $\widetilde{\boldsymbol{\theta }}_{G}^{\tau },$
is defined by%
\begin{equation*}
\widetilde{\boldsymbol{\theta }}_{G}^{\tau }=\arg \max_{\Theta_0}H_{n}^{\tau
}\left( \boldsymbol{\theta }\right) ,
\end{equation*}%
where $H_{n}^{\tau }\left( \boldsymbol{\theta }\right) $ is as given in Equation (\ref{eq:Hn}) and $\Theta_0 = \{ \boldsymbol{\theta \in }\Theta /\text{ }\boldsymbol{g}(\boldsymbol{\theta}_{r}) = \boldsymbol{0}_r \}$ is the restricted parameter space defined in (\ref{104}).
\end{definition}

 Before presenting the asymptotic distribution of the MDPDGE we present some previous results whose proofs are included in the Appendix.

\begin{proposition}
\label{Proposition1}Let $\boldsymbol{Y}_{1},...,$ $\boldsymbol{Y}_{n}$ be
independent and identically distributed observations from a $m$-dimensional
random vector $\boldsymbol{Y}$ with $E_{\boldsymbol{\theta }}\left[ 
\boldsymbol{Y}\right] =\boldsymbol{\mu }\left( \boldsymbol{\theta }\right) $
and $Cov_{\boldsymbol{\theta }}\left[ \boldsymbol{Y}\right] =\boldsymbol{%
\Sigma }\left( \boldsymbol{\theta }\right) $, $\boldsymbol{\theta }\in
\Theta \subset \mathbb{R}^{d}.$ Then, 
\begin{equation*}
\sqrt{n}\left( \frac{1}{\tau +1}\frac{\partial H_{n}^{\tau }\left( 
\boldsymbol{\theta }\right) }{\partial \boldsymbol{\theta }}\right) \underset%
{n\longrightarrow \infty }{\overset{\mathcal{L}}{\longrightarrow }}\mathcal{N%
}(\boldsymbol{0}_{d},\boldsymbol{K}_{\boldsymbol{\tau }}(\boldsymbol{\theta }%
)),
\end{equation*}%
where $\boldsymbol{K}_{\boldsymbol{\tau }}(\boldsymbol{\theta })$ was
defined in (\ref{103aaa}).
\end{proposition}

\begin{proof}
See Appendix A.
\end{proof}

\begin{proposition}
\label{Proposition2}Let $\boldsymbol{Y}_{1},...,$ $\boldsymbol{Y}_{n}$ be
independent and identically distributed observations from a $m$-dimensional
random vector $\boldsymbol{Y}$ with $E_{\boldsymbol{\theta }}\left[ 
\boldsymbol{Y}\right] =\boldsymbol{\mu }\left( \boldsymbol{\theta }\right) $
and $Cov_{\boldsymbol{\theta }}\left[ \boldsymbol{Y}\right] =\boldsymbol{%
\Sigma }\left( \boldsymbol{\theta }\right) $, $\boldsymbol{\theta }\in
\Theta \subset \mathbb{R}^{d}.$ Then,%
\begin{equation*}
\frac{\partial ^{2}H_{n}^{\tau }\left( \boldsymbol{\theta }\right) }{%
\partial \boldsymbol{\theta \partial \theta }^{T}}\underset{n\longrightarrow
\infty }{\overset{\mathcal{P}}{\longrightarrow }}-\left( \tau +1\right) 
\boldsymbol{J}_{\boldsymbol{\tau }}(\boldsymbol{\theta }),
\end{equation*}%
where $\boldsymbol{J}_{\boldsymbol{\tau }}(\boldsymbol{\theta })$ was
defined in (\ref{103aa}).
\end{proposition}

\begin{proof}
See Appendix B.
\end{proof}

Next, we present the asymptotic distribution of \ $\widetilde{\boldsymbol{\theta }}_{G}^{\tau }.$ 

\begin{theorem}
Let $\boldsymbol{Y}_{1},...,$ $\boldsymbol{Y}_{n}$ be independent and
identically distributed observations from a $m$-dimensional random vector $%
\boldsymbol{Y}$ with $E_{\boldsymbol{\theta }}\left[ \boldsymbol{Y}\right] =%
\boldsymbol{\mu }\left( \boldsymbol{\theta }\right) $ and $Cov_{\boldsymbol{%
\theta }}\left[ \boldsymbol{Y}\right] =\boldsymbol{\Sigma }\left( 
\boldsymbol{\theta }\right) $, $\boldsymbol{\theta }\in \Theta \subset 
\mathbb{R}^{d}.$ Suppose the true distribution of $\boldsymbol{Y}$ belongs
to the model and we consider $\boldsymbol{\theta }\in \boldsymbol{\Theta }%
_{0}$. Then the RMDPDGE $\widetilde{\boldsymbol{\theta }}_{G}^{\tau }$ of $%
\boldsymbol{\theta }$ obtained under the constraints $\boldsymbol{g}(%
\boldsymbol{\theta })=\boldsymbol{0}_{r}$ has the asymptotic distribution, 
\begin{equation}
n^{1/2}(\widetilde{\boldsymbol{\theta }}_{G}^{\tau }-\boldsymbol{\theta })%
\underset{n\longrightarrow \infty }{\overset{\mathcal{L}}{\longrightarrow }}%
\mathcal{N}(\boldsymbol{0}_{d},\boldsymbol{M}_{\boldsymbol{\tau }}(%
\boldsymbol{\theta }))  \label{105a}
\end{equation}%
where 
\begin{equation*}
\boldsymbol{M}_{\boldsymbol{\tau }}(\boldsymbol{\theta })=\boldsymbol{P}_{%
\boldsymbol{\tau }}^{\ast }(\boldsymbol{\theta })\boldsymbol{K}_{\boldsymbol{%
\tau }}\left( \boldsymbol{\theta }\right) \boldsymbol{P}_{\boldsymbol{\tau }%
}^{\ast }(\boldsymbol{\theta })^{T},
\end{equation*}%
\begin{equation}
\boldsymbol{P}_{\boldsymbol{\tau }}^{\ast }(\boldsymbol{\theta })=%
\boldsymbol{J}_{\boldsymbol{\tau }}(\boldsymbol{\theta })^{-1}-\boldsymbol{Q}%
_{\boldsymbol{\tau }}(\boldsymbol{\theta })\boldsymbol{G}(\boldsymbol{\theta 
})^{T}\boldsymbol{J}_{\boldsymbol{\tau }}(\boldsymbol{\theta })^{-1},
\label{106}
\end{equation}%
\begin{equation}
\boldsymbol{Q}_{\boldsymbol{\tau }}(\boldsymbol{\theta })=\boldsymbol{J}_{%
\boldsymbol{\tau }}^{-1}(\boldsymbol{\theta })\boldsymbol{G}(\boldsymbol{%
\theta })\left[ \boldsymbol{G}(\boldsymbol{\theta })^{T}\boldsymbol{J}_{%
\boldsymbol{\tau }}(\boldsymbol{\theta })^{-1}\boldsymbol{G}(\boldsymbol{%
\theta })\right] ^{-1},  \label{107}
\end{equation}%
and $\boldsymbol{J}_{\boldsymbol{\tau }}(\boldsymbol{\theta })$ and $%
\boldsymbol{K}_{\boldsymbol{\tau }}\left( \boldsymbol{\theta }\right) $ were
defined in (\ref{103aa}) and (\ref{103aaa}), respectively.
\end{theorem}

\begin{proof}
The estimating equations for the RMDPDGE are given by 
\begin{equation}
\left \{ 
\begin{array}{r}
\tfrac{\partial }{\partial \boldsymbol{\theta }}H_{n}^{\tau }(\boldsymbol{%
\theta })+\boldsymbol{G}(\boldsymbol{\theta })\boldsymbol{\lambda }_{n}=%
\boldsymbol{0}_{d}, \\ 
\boldsymbol{g}(\widetilde{\boldsymbol{\theta }}_{G}^{\tau })=\boldsymbol{0}%
_{r},%
\end{array}%
\right.  \label{108}
\end{equation}%
where $\boldsymbol{\lambda }_{n}$ is a vector of Lagrangian multipliers. Now
we consider $\boldsymbol{\theta }_{n}=\boldsymbol{\theta }+\boldsymbol{m}%
n^{-1/2}$, where $||\boldsymbol{m}||<k$, for $0<k<\infty $. We have, 
\begin{equation*}
\frac{\partial }{\partial \boldsymbol{\theta }}\left. H_{n}^{\tau }(%
\boldsymbol{\theta })\right \vert _{\boldsymbol{\theta =\theta }_{n}}=\frac{%
\partial }{\partial \boldsymbol{\theta }}H_{n}^{\tau }(\boldsymbol{\theta })+%
\frac{\partial ^{2}}{\partial \boldsymbol{\theta }^{T}\partial \boldsymbol{%
\theta }}\left. H_{n}^{\tau }(\boldsymbol{\theta })\right \vert _{%
\boldsymbol{\theta =\theta }_{\ast }}(\boldsymbol{\theta }_{n}^{\ast }-%
\boldsymbol{\theta })
\end{equation*}%
and 
\begin{equation}
n^{1/2}\left. \frac{\partial }{\partial \boldsymbol{\theta }}H_{n}^{\tau }(%
\boldsymbol{\theta })\right \vert _{\boldsymbol{\theta =\theta }_{n}}=n^{1/2}%
\frac{\partial }{\partial \boldsymbol{\theta }}H_{n}^{\tau }(\boldsymbol{%
\theta })+\frac{\partial ^{2}}{\partial \boldsymbol{\theta }^{T}\partial 
\boldsymbol{\theta }}\left. H_{n}^{\tau }(\boldsymbol{\theta })\right \vert
_{\boldsymbol{\theta =\theta }_{\ast }}n^{1/2}(\boldsymbol{\theta }_{n}-%
\boldsymbol{\theta })  \label{109}
\end{equation}%
where $\boldsymbol{\theta }^{\ast }$ belongs to the segment joining $%
\boldsymbol{\theta }$ and $\boldsymbol{\theta }_{n}$

Since 
\begin{equation*}
\lim_{n\rightarrow \infty }\frac{\partial ^{2}}{\partial \boldsymbol{\theta }%
^{T}\partial \boldsymbol{\theta }}H_{n}^{\tau }(\boldsymbol{\theta }%
)=-\left( \tau +1\right) \boldsymbol{J}_{\boldsymbol{\tau }}(\boldsymbol{%
\theta })
\end{equation*}%
we obtain 
\begin{equation}
n^{1/2}\left. \frac{\partial }{\partial \boldsymbol{\theta }}H_{n}^{\tau }(%
\boldsymbol{\theta })\right \vert _{\boldsymbol{\theta =\theta }_{n}}=n^{1/2}%
\frac{\partial }{\partial \boldsymbol{\theta }}H_{n}^{\tau }(\boldsymbol{%
\theta })-\left( \tau +1\right) n^{1/2}\boldsymbol{J}_{\boldsymbol{\tau }}(%
\boldsymbol{\theta })(\boldsymbol{\theta }_{n}-\boldsymbol{\theta }%
)+o_{p}(1).  \label{110}
\end{equation}%
Taking into account that $\boldsymbol{G}(\boldsymbol{\theta })$ is
continuous in $\boldsymbol{\theta }$ 
\begin{equation}
n^{1/2}\boldsymbol{g}(\boldsymbol{\theta }_{n})=\boldsymbol{G}(\boldsymbol{%
\theta })^{T}n^{1/2}(\boldsymbol{\theta }_{n}-\boldsymbol{\theta })+o_{p}(1).
\label{111}
\end{equation}

The RMDPDGE $\widetilde{\boldsymbol{\theta }}_{G}^{\tau }$ must satisfy the
conditions in (\ref{108}), and in view of (\ref{110}) we have 
\begin{equation}
n^{1/2}\frac{\partial }{\partial \boldsymbol{\theta }}H_{n}^{\tau }(%
\boldsymbol{\theta })-\left( \tau +1\right) \boldsymbol{J}_{\boldsymbol{\tau 
}}(\boldsymbol{\theta })n^{1/2}(\widetilde{\boldsymbol{\theta }}_{G}^{\tau }-%
\boldsymbol{\theta })+\boldsymbol{G}(\boldsymbol{\theta })n^{1/2}\boldsymbol{%
\lambda }_{n}+o_{p}(1)=\boldsymbol{0}_{p}.  \label{113}
\end{equation}%
From (\ref{111}) it follows that 
\begin{equation}
\boldsymbol{G}^{T}(\boldsymbol{\theta })n^{1/2}(\widetilde{\boldsymbol{%
\theta }}_{G}^{\tau }-\boldsymbol{\theta })+o_{p}(1)=\boldsymbol{0}_{r}.
\label{114}
\end{equation}%
Now we can express equations (\ref{113}) and (\ref{114}) in the matrix form
as 
\begin{equation*}
\left( 
\begin{array}{cc}
\left( \tau +1\right) \boldsymbol{J}_{\boldsymbol{\tau }}(\boldsymbol{\theta 
}) & -\boldsymbol{G}(\boldsymbol{\theta }) \\ 
-\boldsymbol{G}^{T}(\boldsymbol{\theta }) & \boldsymbol{0}%
\end{array}%
\right) \left( 
\begin{array}{c}
n^{1/2}(\widetilde{\boldsymbol{\theta }}_{G}^{\tau }-\boldsymbol{\theta })
\\ 
n^{1/2}\boldsymbol{\lambda }_{n}%
\end{array}%
\right) =\left( 
\begin{array}{c}
n^{1/2}\frac{\partial }{\partial \boldsymbol{\theta }}H_{n}^{\tau }(%
\boldsymbol{\theta }) \\ 
\boldsymbol{0}_{{}}%
\end{array}%
\right) +o_{p}(1).
\end{equation*}%
Therefore 
\begin{equation*}
\left( 
\begin{array}{c}
n^{1/2}(\widetilde{\boldsymbol{\theta }}_{G}^{\tau }-\boldsymbol{\theta })
\\ 
n^{1/2}\boldsymbol{\lambda }_{n}%
\end{array}%
\right) =\left( 
\begin{array}{cc}
\left( \tau +1\right) \boldsymbol{J}_{\boldsymbol{\tau }}(\boldsymbol{\theta 
}) & -\boldsymbol{G}(\boldsymbol{\theta }) \\ 
-\boldsymbol{G}^{T}(\boldsymbol{\theta }) & \boldsymbol{0}%
\end{array}%
\right) ^{-1}\left( 
\begin{array}{c}
n^{1/2}\frac{\partial }{\partial \boldsymbol{\theta }}H_{n}^{\tau }(%
\boldsymbol{\theta }) \\ 
\boldsymbol{0}_{r}%
\end{array}%
\right) +o_{p}(1).
\end{equation*}%
But 
\begin{equation*}
\left( 
\begin{array}{cc}
\left( \tau +1\right) \boldsymbol{J}_{\boldsymbol{\tau }}(\boldsymbol{\theta 
}) & -\boldsymbol{G}(\boldsymbol{\theta }) \\ 
-\boldsymbol{G}^{T}(\boldsymbol{\theta }) & \boldsymbol{0}%
\end{array}%
\right) ^{-1}={\left( 
\begin{array}{cc}
\boldsymbol{L}_{\tau }^{\ast }(\boldsymbol{\theta }) & \boldsymbol{Q}_{\tau
}(\boldsymbol{\theta }) \\ 
\boldsymbol{Q}_{\tau }(\boldsymbol{\theta }_{0})^{T} & \boldsymbol{R}_{\tau
}(\boldsymbol{\theta })%
\end{array}%
\right) },
\end{equation*}%
where 
\begin{eqnarray*}
\boldsymbol{L}_{\tau }^{\ast }(\boldsymbol{\theta }) &=&\frac{1}{\tau +1}%
\left( \boldsymbol{J}_{\boldsymbol{\tau }}(\boldsymbol{\theta })^{-1}-%
\boldsymbol{Q}_{\boldsymbol{\tau }}(\boldsymbol{\theta })\boldsymbol{G}(%
\boldsymbol{\theta })^{T}\boldsymbol{J}_{\boldsymbol{\tau }}(\boldsymbol{%
\theta })^{-1}\right) \\
&=&\frac{1}{\tau +1}\boldsymbol{P}_{\tau }^{\ast }(\boldsymbol{\theta }) \\
\boldsymbol{Q}_{\tau }(\boldsymbol{\theta }) &=&\boldsymbol{J}_{\boldsymbol{%
\tau }}^{-1}(\boldsymbol{\theta })\boldsymbol{G}(\boldsymbol{\theta })\left[ 
\boldsymbol{G}(\boldsymbol{\theta })^{T}\boldsymbol{J}_{\boldsymbol{\tau }}(%
\boldsymbol{\theta })^{-1}\boldsymbol{G}(\boldsymbol{\theta })\right] ^{-1}
\\
\boldsymbol{R}_{\tau }(\boldsymbol{\theta }) &=&\boldsymbol{G}(\boldsymbol{%
\theta })^{T}\boldsymbol{J}_{\boldsymbol{\tau }}(\boldsymbol{\theta })^{-1}%
\boldsymbol{G}(\boldsymbol{\theta })
\end{eqnarray*}%
and $\boldsymbol{P}_{\tau }^{\ast }(\boldsymbol{\theta }_{0})$ and $%
\boldsymbol{Q}_{\tau }(\boldsymbol{\theta }_{0})$ are as given in (\ref{106}%
) and (\ref{107}) respectively. Then, 
\begin{equation}
n^{1/2}(\widetilde{\boldsymbol{\theta }}_{G}^{\tau }-\boldsymbol{\theta }%
)=\left( \tau +1\right) ^{-1}\boldsymbol{P}_{\tau }^{\ast }(\boldsymbol{%
\theta })n^{1/2}\frac{\partial }{\partial \boldsymbol{\theta }}H_{n}^{\tau }(%
\boldsymbol{\theta })+o_{p}(1),  \label{115}
\end{equation}%
and we know by Proposition \ref{Proposition1} that 
\begin{equation}
n^{1/2}\left( \tau +1\right) ^{-1}\frac{\partial }{\partial \boldsymbol{%
\theta }}H_{n}^{\tau }(\boldsymbol{\theta })\underset{n\longrightarrow
\infty }{\overset{\mathcal{L}}{\longrightarrow }}\mathcal{N}\left( 
\boldsymbol{0},\boldsymbol{K}_{\boldsymbol{\tau }}\left( \boldsymbol{\theta }%
\right) \right) .  \label{116}
\end{equation}%
Now by (\ref{115}) and (\ref{116}) we have the desired result presented in (%
\ref{105a}).
\end{proof}

\begin{remark}
Notice that the result in (\ref{103a}) is a special case of the previous
theorem when there is no restriction on the parametric space, in the sense
that $\boldsymbol{G}$, defined in (\ref{105}), is the null matrix. In this
case the matrix $\boldsymbol{P}_{\boldsymbol{\tau }}^{\ast }(\boldsymbol{%
\theta })$ given in (\ref{106}) becomes $\boldsymbol{P}_{\boldsymbol{\tau }%
}^{\ast }(\boldsymbol{\theta })=\boldsymbol{J}_{\boldsymbol{\tau }}(%
\boldsymbol{\theta })^{-1}.$ Therefore, the asymptotic variance-covariance
matrix of the unrestricted estimator, i.e., the MDPDGE, may be reconstructed
from the previous theorem.
\end{remark}

The MDPDGE is an optimum of a differentiable function $H_n^\tau$, so it must annul  it first derivatives.
Using that 
\begin{itemize}
	\item $\dfrac{\partial \left \vert \boldsymbol{\Sigma }\left( \boldsymbol{%
			\theta }\right) \right \vert ^{-\tau /2}}{\partial \boldsymbol{\theta }}=-%
	\frac{\tau }{2}\left \vert \boldsymbol{\Sigma }\left( \boldsymbol{\theta }%
	\right) \right \vert ^{-\tau /2}trace\left( \boldsymbol{\Sigma }\left( 
	\boldsymbol{\theta }\right) ^{-1}\dfrac{\partial \boldsymbol{\Sigma }\left( 
		\boldsymbol{\theta }\right) }{\partial \boldsymbol{\theta }}\right) $
	
	\item $\frac{\partial }{\partial \boldsymbol{\theta }}\left( \boldsymbol{y}%
	_{i}-\boldsymbol{\mu }\left( \boldsymbol{\theta }\right) \right) ^{T}%
	\boldsymbol{\Sigma }\left( \boldsymbol{\theta }\right) ^{-1}\left( 
	\boldsymbol{y}_{i}-\boldsymbol{\mu }\left( \boldsymbol{\theta }\right)
	\right) =-2\left( \frac{\partial \boldsymbol{\mu }\left( \boldsymbol{\theta }%
		\right) }{\partial \boldsymbol{\theta }}\right) ^{T}\boldsymbol{\Sigma }%
	\left( \boldsymbol{\theta }\right) ^{-1}\left( \boldsymbol{y}_{i}-%
	\boldsymbol{\mu }\left( \boldsymbol{\theta }\right) \right) -\left( 
	\boldsymbol{y}_{i}-\boldsymbol{\mu }\left( \boldsymbol{\theta }\right)
	\right) ^{T}\left( \boldsymbol{\Sigma }\left( \boldsymbol{\theta }\right)
	^{-1}\dfrac{\partial \boldsymbol{\Sigma }\left( \boldsymbol{\theta }\right) 
	}{\partial \boldsymbol{\theta }}\boldsymbol{\Sigma }\left( \boldsymbol{%
		\theta }\right) ^{-1}\right) \left( \boldsymbol{y}_{i}-\boldsymbol{\mu }%
	\left( \boldsymbol{\theta }\right) \right) ,$
\end{itemize}
and taking derivatives in (\ref{eq:Hn}), for a certain fixed $\tau,$ we have that 
\begin{eqnarray*}
	\frac{\partial }{\partial \boldsymbol{\theta }}H_{n}^{\tau }(\boldsymbol{%
		\theta }) &=&\frac{1}{n}\tsum \limits_{i=1}^{n}\left \{ -a\text{ }\frac{\tau 
	}{2}\left \vert \boldsymbol{\Sigma }\left( \boldsymbol{\theta }\right)
	\right \vert ^{-\tau /2}trace\left( \boldsymbol{\Sigma }\left( \boldsymbol{%
		\theta }\right) ^{-1}\dfrac{\partial \boldsymbol{\Sigma }\left( \boldsymbol{%
			\theta }\right) }{\partial \boldsymbol{\theta }}\right) \right. \\
	&&\left. \exp \left( -\frac{\tau }{2}\left( \boldsymbol{y}_{i}-\boldsymbol{%
		\mu }\left( \boldsymbol{\theta }\right) \right) ^{T}\boldsymbol{\Sigma }%
	\left( \boldsymbol{\theta }\right) ^{-1}\left( \boldsymbol{y}_{i}-%
	\boldsymbol{\mu }\left( \boldsymbol{\theta }\right) \right) \right) +ba\text{
	}\frac{\tau }{2}\left \vert \boldsymbol{\Sigma }\left( \boldsymbol{\theta }%
	\right) \right \vert ^{-\tau /2}\text{ }\right. \\
	&&\left. trace\left( \boldsymbol{\Sigma }\left( \boldsymbol{\theta }\right)
	^{-1}\dfrac{\partial \boldsymbol{\Sigma }\left( \boldsymbol{\theta }\right) 
	}{\partial \boldsymbol{\theta }}\right) +a\text{ }\frac{\tau }{2}\left \vert 
	\boldsymbol{\Sigma }\left( \boldsymbol{\theta }\right) \right \vert ^{-\tau
		/2}\exp \left( -\frac{\tau }{2}\left( \boldsymbol{y}_{i}-\boldsymbol{\mu }%
	\left( \boldsymbol{\theta }\right) \right) ^{T}\boldsymbol{\Sigma }\left( 
	\boldsymbol{\theta }\right) ^{-1}\left( \boldsymbol{y}_{i}-\boldsymbol{\mu }%
	\left( \boldsymbol{\theta }\right) \right) \right) \right. \\
	&&\left[ 2\left( \frac{\partial \boldsymbol{\mu }\left( \boldsymbol{\theta }%
		\right) }{\partial \boldsymbol{\theta }}\right) ^{T}\boldsymbol{\Sigma }%
	\left( \boldsymbol{\theta }\right) ^{-1}\left( \boldsymbol{y}_{i}-%
	\boldsymbol{\mu }\left( \boldsymbol{\theta }\right) \right) \right. \\
	&&\left. \left. +\left( \boldsymbol{y}_{i}-\boldsymbol{\mu }\left( 
	\boldsymbol{\theta }\right) \right) ^{T}\left( \boldsymbol{\Sigma }\left( 
	\boldsymbol{\theta }\right) ^{-1}\dfrac{\partial \boldsymbol{\Sigma }\left( 
		\boldsymbol{\theta }\right) }{\partial \boldsymbol{\theta }}\boldsymbol{%
		\Sigma }\left( \boldsymbol{\theta }\right) ^{-1}\right) \left( \boldsymbol{y}%
	_{i}-\boldsymbol{\mu }\left( \boldsymbol{\theta }\right) \right) \right]
	\right \} \\
	&:=&\frac{1}{n}\tsum \limits_{i=1}^{n}\boldsymbol{\Psi }_{\boldsymbol{\tau }}(%
	\boldsymbol{y}_{i};\boldsymbol{\theta }),
\end{eqnarray*}%
with 
\begin{eqnarray}
	\boldsymbol{\Psi }_{\boldsymbol{\tau }}(\boldsymbol{y}_{i};\boldsymbol{%
		\theta )} &=&a\text{ }\frac{\tau }{2}\left \vert \boldsymbol{\Sigma }\left( 
	\boldsymbol{\theta }\right) \right \vert ^{-\tau /2}\left \{ \left[
	-trace\left( \boldsymbol{\Sigma }\left( \boldsymbol{\theta }\right) ^{-1}%
	\dfrac{\partial \boldsymbol{\Sigma }\left( \boldsymbol{\theta }\right) }{%
		\partial \boldsymbol{\theta }}\right) \right. \right.  \label{116b} \\
	&&\left. +\left( 2\left( \frac{\partial \boldsymbol{\mu }\left( \boldsymbol{%
			\theta }\right) }{\partial \boldsymbol{\theta }}\right) ^{T}\boldsymbol{%
		\Sigma }\left( \boldsymbol{\theta }\right) ^{-1}\left( \boldsymbol{y}_{i}-%
	\boldsymbol{\mu }\left( \boldsymbol{\theta }\right) \right) \right. \right. 
	\notag \\
	&&\left. \left. \left. +\left( \boldsymbol{y}_{i}-\boldsymbol{\mu }\left( 
	\boldsymbol{\theta }\right) \right) ^{T}\left( \boldsymbol{\Sigma }\left( 
	\boldsymbol{\theta }\right) ^{-1}\dfrac{\partial \boldsymbol{\Sigma }\left( 
		\boldsymbol{\theta }\right) }{\partial \boldsymbol{\theta }}\boldsymbol{%
		\Sigma }\left( \boldsymbol{\theta }\right) ^{-1}\right) \boldsymbol{y}_{i}-%
	\boldsymbol{\mu }\left( \boldsymbol{\theta }\right) \right. \right. \right) .
	\notag \\
	&&\left. \left. \left. \exp \left( -\frac{\tau }{2}\left( \boldsymbol{y}_{i}-%
	\boldsymbol{\mu }\left( \boldsymbol{\theta }\right) \right) ^{T}\boldsymbol{%
		\Sigma }\left( \boldsymbol{\theta }\right) ^{-1}\left( \boldsymbol{y}_{i}-%
	\boldsymbol{\mu }\left( \boldsymbol{\theta }\right) \right) \right) \right.
	\right. \right]  \notag \\
	&&\left. \left. +b\text{ }trace\left( \boldsymbol{\Sigma }\left( \boldsymbol{%
		\theta }\right) ^{-1}\dfrac{\partial \boldsymbol{\Sigma }\left( \boldsymbol{%
			\theta }\right) }{\partial \boldsymbol{\theta }}\right) \right. \right \} . 
	\notag
\end{eqnarray}

Therefore the estimating equations of the MDPDGE for a fixed parameter $\tau$ are given by
\begin{equation}
	\tsum \limits_{i=1}^{n}\boldsymbol{\Psi }_{\boldsymbol{\tau }}(\boldsymbol{y}%
	_{i};\boldsymbol{\theta })=\boldsymbol{0}_{d}.  \label{116a}
\end{equation}%

The previous estimating equations characterizes the the MDPDGE as an M-estimator and so it asymptotic distribution could have been also derived from the general theory of M-estimators. In particular, the MDPDGE, $\widehat{\boldsymbol{\theta }}_{G}^{\tau},$ satisfies for any $\tau \geq 0$
\begin{equation}\label{eq:asymptM}
	\sqrt{n}\left( \widehat{\boldsymbol{\theta }}_{G}^{\tau }-\boldsymbol{\theta 
	}\right) \underset{n\longrightarrow \infty }{\overset{\mathcal{L}}{%
			\longrightarrow }}\mathcal{N}\left( \boldsymbol{0}_{d},\boldsymbol{S}^{-1}%
	\boldsymbol{MS}^{-1}\right)
\end{equation}%
with
\begin{equation}\label{eq:SM}
	\boldsymbol{S}=-E\left[ \frac{\partial ^{2}H_{n}^{\tau }\left( \boldsymbol{%
			\theta }\right) }{\partial \boldsymbol{\theta \partial \theta }^{T}}\right] 
	\text{ and }\boldsymbol{M}=Cov\left[ \sqrt{n}\frac{\partial }{\partial 
		\boldsymbol{\theta }}H_{n}^{\tau }(\boldsymbol{\theta })\right] .
\end{equation}%
Based on Propositions \ref{Proposition1} and \ref{Proposition2} we can express the previous matrices as
\begin{equation*}
	\boldsymbol{S=}\left( \tau +1\right) \boldsymbol{J}_{\boldsymbol{\tau }%
	}\left( \boldsymbol{\theta }\right) \text{ and }\boldsymbol{M=}\left( \tau
	+1\right) ^{2}\boldsymbol{K}_{\boldsymbol{\tau }}\left( \boldsymbol{\theta }%
	\right)
\end{equation*}%
and we get back the expressions stated in (\ref{103a}). The asymptotic convergence in (\ref{eq:asymptM}) offers an alternative proof of the 
 asymptotic distribution of MDPDGE stated in Castilla and Zografos (2019), in terms of the transformed matrices $\boldsymbol{S}$ and $\boldsymbol{M}$ in Equation (\ref{eq:SM}).

\section{Influence function for the RMDPDGE}

To analyze the robustness of an estimator, Hampel et al. (1986) introduced
the concept of Influence Function (IF). Since then, the IF have been widely
used in the statistical literature to measure robustness in different
statistical contexts. Intuitively, the IF describes the effect of an
infinitesimal contamination of the model on the estimation. Robust estimators should be less affected by contamination and thus IFs associated to locally robust (B-robust) estimators should be bounded.

The IF of an estimator, $\widetilde{\boldsymbol{\theta}}_G^\tau,$ is defined in terms of its statistical functional $\widetilde{T}_{\tau }$ satisfying  $\widetilde{T}_{\tau }(g)= \widetilde{\boldsymbol{\theta }}_{G}^{\tau },$ where $g$ is the true density function underlying the data.
 Given the density function $g$ we define it contaminated version at the point perturbation $\boldsymbol{y}_0$ as,
\begin{equation} \label{eq:contdensity}
	g = (1-\varepsilon)g+\varepsilon\Delta_{\boldsymbol{y}_0},
\end{equation}
where $\varepsilon$ is  fraction of contamination and $\Delta_{\boldsymbol{y}_0}$ denotes the indicator function at $\boldsymbol{y}_0.$
Then, the IF of $\widetilde{\theta}_G^\tau$ is defined as the derivative of the functional at $\varepsilon = 0$
$$IF (\boldsymbol{y}_0, \widetilde{T}_\tau) = \frac{\partial \widetilde{T}_\tau (g_\varepsilon)}{\varepsilon}|_{\varepsilon = 0} $$
The above derivative quantifies the rate of change of the sample estimator as the fraction of contamination changes, i.e. how much the estimate is affected by contamination.

 Let us now obtain the IF of RMDPDE. We consider the contaminated model 
 $$ g_{\varepsilon }(\boldsymbol{y})=(1-\varepsilon )f_{\boldsymbol{\theta }}(%
\boldsymbol{y})+\varepsilon \Delta _{\boldsymbol{y}_0},$$ with $\Delta _{%
\boldsymbol{y}_0}$ the indicator function degenerated at the mass point $\boldsymbol{y}_0,$ and  $f_{\boldsymbol{\theta }}$ is the assumed probability density function of a normal population.
The MDPDGE for the contaminated model is then given by $\widetilde{\boldsymbol{\theta }}_{G,\varepsilon }^{\tau }=\widetilde{T}%
_{\tau }(g_{\varepsilon }).$

 By definition $\widetilde{\boldsymbol{\theta }}_{G,\varepsilon }^{\tau }$ is the minimizer of the loss function $H_{n}^{\tau }(\boldsymbol{\theta })$ in (\ref{eq:Hn}), subject to the constraints $\boldsymbol{g}(\widetilde{\boldsymbol{\theta }}_{G,\varepsilon }^{\tau })\boldsymbol{=0.}$ 
 Using the characterization of the MDPDGE as an M-estimator we have that the influence function of the MDPDGE is given by
\begin{equation}
IF(\boldsymbol{y},\widetilde{T}_{\tau },\boldsymbol{\theta })=\boldsymbol{J}%
_{\tau }(\boldsymbol{\theta })^{-1}\boldsymbol{\Psi }_{\boldsymbol{\tau }}(%
\boldsymbol{y};\boldsymbol{\theta )},  \label{IF1}
\end{equation}%
where $\boldsymbol{J}_{\tau }(\boldsymbol{\theta })$ was defined in (\ref%
{103aa}) and $\boldsymbol{\Psi }_{\boldsymbol{\tau }}(\boldsymbol{y};%
\boldsymbol{\theta )}$ in (\ref{116b}). The influence function of the
RMDPDGE will be obtained with the additional condition $\boldsymbol{g}(%
\widetilde{\boldsymbol{\theta }}_{G,\varepsilon }^{\tau })\boldsymbol{=0.}$
Differentiating this last equation gives, at $\varepsilon =0,$%
\begin{equation}
\mathbf{G}\left( \boldsymbol{\theta }\right) ^{T}IF(\boldsymbol{y},%
\widetilde{T}_{\tau },\boldsymbol{\theta })=\boldsymbol{0}.  \label{IF2}
\end{equation}%
Based on (\ref{IF1}) and (\ref{IF2}) we have

\begin{equation*}
\left( 
\begin{array}{c}
\boldsymbol{J}_{\tau }(\boldsymbol{\theta }) \\ 
\mathbf{G}\left( \boldsymbol{\theta }\right) ^{T}%
\end{array}%
\right) IF(\boldsymbol{y},\widetilde{T}_{\tau },\boldsymbol{\theta })=\left( 
\begin{array}{c}
\boldsymbol{\Psi }_{\boldsymbol{\tau }}(\boldsymbol{y};\boldsymbol{\theta )}
\\ 
\mathbf{0}%
\end{array}%
\right) .
\end{equation*}%
Therefore, 
\begin{equation*}
\left( \boldsymbol{J}_{\tau }(\boldsymbol{\theta })^{T},\mathbf{G}\left( 
\boldsymbol{\theta }\right) \right) \left( 
\begin{array}{c}
\boldsymbol{J}_{\tau }(\boldsymbol{\theta }) \\ 
\mathbf{G}\left( \boldsymbol{\theta }\right) ^{T}%
\end{array}%
\right) IF(\boldsymbol{y},\widetilde{T}_{\tau },\boldsymbol{\theta })=%
\boldsymbol{J}_{\tau }(\boldsymbol{\theta })^{T}\boldsymbol{\Psi }_{%
\boldsymbol{\tau }}(\boldsymbol{y};\boldsymbol{\theta )}
\end{equation*}%
and the influence function of the RMDPDGE, $\widetilde{\boldsymbol{\theta }}%
_{G}^{\tau }\boldsymbol{,}$ is given by 
\begin{equation}
\left( \boldsymbol{J}_{\tau }(\boldsymbol{\theta })^{T}\boldsymbol{J}_{\tau
}(\boldsymbol{\theta }))+\mathbf{G}\left( \boldsymbol{\theta }\right) 
\mathbf{G}\left( \boldsymbol{\theta }\right) ^{T}\right) ^{-1}\boldsymbol{J}%
_{\tau }(\boldsymbol{\theta })^{T}\boldsymbol{\Psi }_{\boldsymbol{\tau }}(%
\boldsymbol{y};\boldsymbol{\theta )}.  \label{eq:IF}
\end{equation}

We can observe that the influence function of $\widetilde{\boldsymbol{\theta 
}}_{G}^{\tau },$ obtained in (\ref{eq:IF}), will be bounded if the influence
function of the MDPDGE, $\widehat{\boldsymbol{\theta }}_{G}^{\tau },$ given
in (\ref{IF1}) is bounded. In general it is not easy to see if it is bounded
or not but in particular situations it is not difficult. On the other hand
if there are not restrictions, $\mathbf{G}\left( \boldsymbol{\theta }\right)
=\boldsymbol{0},$ and therefore (\ref{eq:IF}) coincides with (\ref{IF1}).

In Section 4.1 we shall present the expression of $J_{\tau }(\theta )$ and $%
\psi _{\tau }(y,\theta )$ for the exponential and Poisson models. Based on
that results we present in Figure 1 the influence function of the MDPDGE, $%
\widehat{\boldsymbol{\theta }}_{G}^{\tau }$, for $\theta =4$ and $\tau =0,$ $%
0,2$ and $0.8$ for the exponential model. We can see that for $\tau =0,$ \
the influence function is not bounded and for $\tau =0,2$ and $0.8$ is
bounded. This fact points out the robustness of the MDPDGE, $\widehat{%
\boldsymbol{\theta }}_{G}^{\tau },$ for $\tau >0.$

\begin{eqnarray*}
&&\FRAME{itbpFX}{4.4996in}{3in}{0in}{}{}{Plot}{\special{language "Scientific
Word";type "MAPLEPLOT";width 4.4996in;height 3in;depth 0in;display
"USEDEF";plot_snapshots TRUE;mustRecompute FALSE;lastEngine "MuPAD";xmin
"0";xmax "15";xviewmin "0";xviewmax "15";yviewmin "-2";yviewmax
"12";viewset"XY";rangeset"X";plottype 4;labeloverrides 3;x-label "y";y-label
"IF(y)";axesFont "Times New Roman,12,0000000000,useDefault,normal";numpoints
100;plotstyle "patch";axesstyle "normal";axestips FALSE;xis
\TEXUX{x};var1name \TEXUX{$x$};function \TEXUX{$v$};linecolor
"green";linestyle 1;discont FALSE;pointstyle "point";linethickness
2;lineAttributes "Solid";var1range "0,15";num-x-gridlines 100;curveColor
"[flat::RGB:0x00008000]";curveStyle "Line";discont FALSE;function
\TEXUX{$m$};linecolor "blue";linestyle 1;pointstyle "point";linethickness
1;lineAttributes "Solid";var1range "0,15";num-x-gridlines 100;curveColor
"[flat::RGB:0x000000ff]";curveStyle "Line";function \TEXUX{$n$};linecolor
"black";linestyle 1;pointstyle "point";linethickness 1;lineAttributes
"Solid";var1range "0,15";num-x-gridlines 100;curveColor
"[flat::RGB:0000000000]";curveStyle "Line";VCamFile
'RPGRWB02.xvz';valid_file "T";tempfilename
'RPGRWB01.wmf';tempfile-properties "XPR";}} \\
&&%
\begin{array}{c}
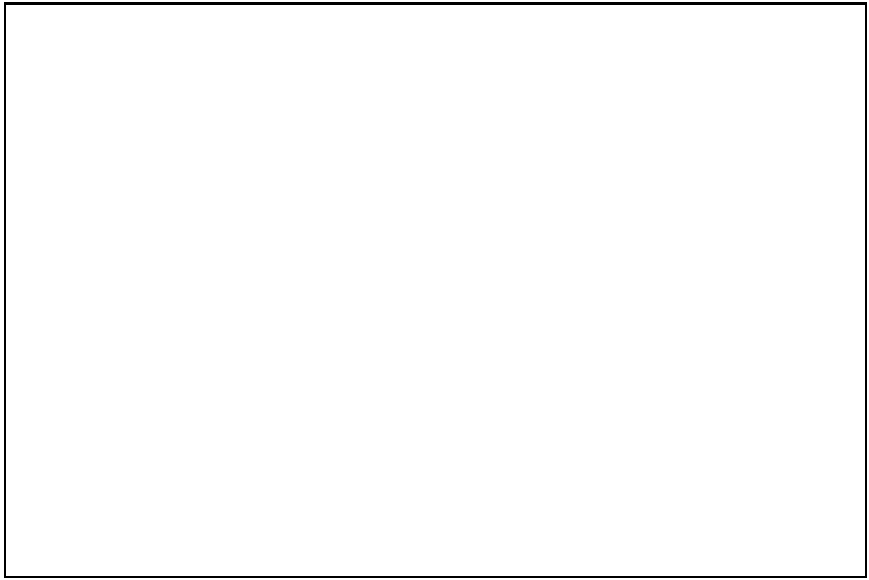
\text{Figure 1. {\small Influence function for} }{\small \tau =0}\text{ 
{\small (read)}, }{\small 0.2}\text{ {\small (black) and} }{\small 0.8}\text{
{\small (Green)}} \\ 
\text{{\small (Exponential model)}}%
\end{array}%
\end{eqnarray*}

\section{Rao-type tests based on RMDPDGE}

Recently many robust test statistics based on minimum distance estimators have been introduced in the statistical literature for testing under different statistical models. Among them, density power divergence and Rényi's pseudodistance based test statistics have shown very competitive performance with respect to classical tests in many different problems.
Distance-based test statistics are essentially of two types: Wald-type tests and Rao-type tests. Some applications of these tests are the following: Basu et al. (2013, 2016, 2017, 2018a, 2018b, 2021, 2022a, 2022b), Castilla et al (2020, 2022a, 2022b), Ghosh et al (2016, 2021), Jaenada et al (2022a, 2022b), Martín (2021), Men\'{e}ndez et al
(1995) and all references therein.
In this section we introduce  the Rao-type tests based on RMDPDGE and we study their asymptotic properties, proving the consistency of the tests.

We analyze here simple null hypothesis of the form
\begin{equation}
	H_{0}:\boldsymbol{\theta }=\boldsymbol{\theta }_{0}\text{ versus }H_{1}:%
	\boldsymbol{\theta }\neq \boldsymbol{\theta }_{0}.  \label{117}
\end{equation}%
 
 \begin{definition}
 	Let $\boldsymbol{Y}_{1},...,\boldsymbol{Y}_{n}$ be independent and
 	identically distributed observations from a $m$-dimensional random vector $%
 	\boldsymbol{Y}$ with $E_{\boldsymbol{\theta }}\left[ \boldsymbol{Y}\right] =%
 	\boldsymbol{\mu }\left( \boldsymbol{\theta }\right) $ and $Cov_{\boldsymbol{%
 			\theta }}\left[ \boldsymbol{Y}\right] =\boldsymbol{\Sigma }\left( 
 	\boldsymbol{\theta }\right) $, $\boldsymbol{\theta }\in \Theta \subset 
 	\mathbb{R}^{d},$ and consider the testing problem defined in (\ref{117}).
 	The Rao-type test statistic based on RMDPDGE, is defined by
 	\begin{equation}
 		R_{\boldsymbol{\tau }}\left( \boldsymbol{\theta }_{0}\right) =\frac{1}{n}%
 		\boldsymbol{U}_{n}^{\boldsymbol{\tau }}(\boldsymbol{\theta }_{0}\boldsymbol{)%
 		}^{T}\boldsymbol{K}_{\boldsymbol{\tau }}\left( \boldsymbol{\theta }%
 		_{0}\right) ^{-1}\boldsymbol{U}_{n}^{\boldsymbol{\tau }}(\boldsymbol{\theta }%
 		_{0}),  \label{117AAA}
 	\end{equation}
 	where 
 	\begin{equation*}
 		\boldsymbol{U}_{n}^{\boldsymbol{\tau }}(\boldsymbol{\theta )=}\left( \frac{1%
 		}{\tau +1}\tsum \limits_{i=1}^{n}\Psi _{\boldsymbol{\tau }}^{1}(\boldsymbol{y%
 		}_{i};\boldsymbol{\theta }),...,\frac{1}{\tau +1}\tsum
 		\nolimits_{i=1}^{n}\Psi _{\boldsymbol{\tau }}^{d}(\boldsymbol{y}_{i};%
 		\boldsymbol{\theta })\right) ^{T},
 	\end{equation*}
 	is the score function defining the estimating equations of the MDPDGE and $\boldsymbol{\Psi }_{\boldsymbol{\tau }}(\boldsymbol{y}_{i};\boldsymbol{%
 		\theta )=}\left( \Psi _{\boldsymbol{\tau }}^{1}(\boldsymbol{y}_{i};%
 	\boldsymbol{\theta }),....,\Psi _{\boldsymbol{\tau }}^{d}(\boldsymbol{y}_{i};%
 	\boldsymbol{\theta })\right) .$
 \end{definition}
  
Before deriving the asymptotic distribution of the Rao-type test based on the MDPDGE, it is interesting to note that we rewrite
\begin{equation*}
\frac{1}{\sqrt{n}}\tsum \limits_{i=1}^{n}\frac{1}{\tau +1}\boldsymbol{\Psi }%
_{\boldsymbol{\tau }}(\boldsymbol{y}_{i};\boldsymbol{\theta )=}\sqrt{n}\frac{%
1}{\tau +1}\frac{\partial }{\partial \boldsymbol{\theta }}H_{n}^{\tau }(%
\boldsymbol{\theta }),
\end{equation*}%
and hence by Proposition \ref{Proposition1} we can stablish the asymptotic distribution of the score function $\boldsymbol{U}_n^\tau(\boldsymbol{\theta})$
\begin{equation*}
\frac{1}{\sqrt{n}}\tsum \limits_{i=1}^{n}\frac{1}{\tau +1}\boldsymbol{\Psi }%
_{\boldsymbol{\tau }}(\boldsymbol{y}_{i};\boldsymbol{\theta )}\underset{%
n\rightarrow \infty }{\overset{L}{\rightarrow }}\mathcal{N}\left( 
\boldsymbol{0}_{p},\boldsymbol{K}_{\boldsymbol{\tau }}\left( \boldsymbol{%
\theta }\right) \right) .
\end{equation*}

The next result establishes the asymptotic behaviour of the proposed Rao-type test statistic.
\begin{theorem} \label{thm:asymptrao}
Let $\boldsymbol{Y}_{1},...,$ $\boldsymbol{Y}_{n}$ be independent and
identically distributed observations from a $m$- dimensional random vector $%
\boldsymbol{Y}$ with $E_{\boldsymbol{\theta }}\left[ \boldsymbol{Y}\right] =%
\boldsymbol{\mu }\left( \boldsymbol{\theta }\right) $ and $Cov_{\boldsymbol{%
\theta }}\left[ \boldsymbol{Y}\right] =\boldsymbol{\Sigma }\left( 
\boldsymbol{\theta }\right) $, $\boldsymbol{\theta }\in \Theta \subset 
\mathbb{R}^{d}.$ Under the null hypothesis given in (\ref{117}) it holds
\begin{equation*}
R_{\boldsymbol{\tau }}\left( \boldsymbol{\theta }_{0}\right) \underset{%
n\rightarrow \infty }{\overset{L}{\rightarrow }}\chi _{d}^{2}.
\end{equation*}
\end{theorem}

\begin{proof}
As remarked before, the score function is asymptotically normal,
\begin{equation*}
\frac{1}{\sqrt{n}}\boldsymbol{U}_{n}^{\boldsymbol{\tau }}(\boldsymbol{\theta
)=}\frac{1}{\sqrt{n}}\tsum \limits_{i=1}^{n}\frac{1}{\tau +1}\boldsymbol{%
\Psi }_{\boldsymbol{\tau }}(\boldsymbol{y}_{i};\boldsymbol{\theta )=}\sqrt{n}%
\frac{1}{\tau +1}\frac{\partial }{\partial \boldsymbol{\theta }}H_{n}^{\tau
}(\boldsymbol{\theta })\underset{n\longrightarrow \infty }{\overset{\mathcal{%
L}}{\longrightarrow }}\mathcal{N}\left( \boldsymbol{0},\boldsymbol{K}_{%
\boldsymbol{\tau }}\left( \boldsymbol{\theta }\right) \right) .
\end{equation*}%
Then, applying a suitable transformation, the result follows.
\end{proof}

\begin{remark}
Based on Theorem \ref{thm:asymptrao}, for large enough sample sizes, one can use
the $100(1-\alpha )$ percentile, $\chi _{d,\alpha}^{2},$ of
the chi-square with $d$ degrees of freedom satisfying, $$\Pr
\left( \chi _{d}^{2}>\chi _{d,\boldsymbol{\alpha }}^{2}\right) =\alpha, $$ to
define the reject region of the test with null hypothesis in (\ref{117})
$$ RC= \{R_{\boldsymbol{\tau }}\left( \boldsymbol{\theta }_{0}\right) >\chi_{d,\alpha}^{2}\}.$$
\end{remark}

For illustrative purposes, we present here the application of the proposed method in elliptical distributions.
 
\begin{example}
(Elliptical distributions). The $m$-dimensional random vector $\boldsymbol{Y}
$ follows an elliptical distribution if it characteristic function has the form%
\begin{equation*}
\varphi _{\boldsymbol{Y}}\left( \boldsymbol{t}\right) =\exp \left( i%
\boldsymbol{t}^{2}\boldsymbol{\mu }\right) \psi \left( \frac{1}{2}%
\boldsymbol{t}^{2}\boldsymbol{\Sigma t}\right)
\end{equation*}%
where $\boldsymbol{\mu }$ is a $m$-dimensional vector, $\boldsymbol{%
\Sigma }$ is a positive definite matrix and $\psi (t)$ denotes the so-called
characteristic generator function. The function $\psi $ may depend on the
dimension of random vector $\boldsymbol{Y}$. In general, it does not hold
that $\boldsymbol{Y}$ has a joint density function, $f_{\boldsymbol{Y}}(%
\boldsymbol{y}),$ but if this density exists, it is given by 
\begin{equation*}
f_{\boldsymbol{Y}}(\boldsymbol{y})=c_{m}\left \vert \boldsymbol{\Sigma }%
\right \vert ^{-\frac{1}{2}}g_{m}\left( \frac{1}{2}\left( \boldsymbol{y}-%
\boldsymbol{\mu }\right) ^{T}\boldsymbol{\Sigma }^{-1}\left( \boldsymbol{y}-%
\boldsymbol{\mu }\right) \right)
\end{equation*}%
for some density generator function $g_{m}$ which could depend on the
dimension of the random vector. The elliptical distribution family is in the following denoted by $%
E_{m}\left( \boldsymbol{\mu ,\Sigma ,}g_{m}\right) .$ Moreover, if  the
density exists, the parameter $c_{m}$ is given explicitly by 
\begin{equation*}
c_{m}=\left( 2\pi \right) ^{-\frac{m}{2}}\Gamma \left( \frac{m}{2}\right)
\left( \int x^{\frac{m}{2}-1}g_{m}\left( x\right) dx\right) ^{-1}.
\end{equation*}%
For more details about the elliptical family $E_{m}\left( \boldsymbol{\mu ,\Sigma ,}%
g_{m}\right) $ see  Fang et al (1987), Gupta and Varga (1993),
Cambanis et al (1981), Fang and Zhang (1990) and references therein. 
In Fang et al (1987), for instance, it can be seen that the mean vector and variance covariance matrix can obtained as
\begin{equation*}
E\left[ \boldsymbol{Y}\right] =\boldsymbol{\mu }\text{ and }Cov\left[ 
\boldsymbol{Y}\right] =c_{\boldsymbol{Y}}\boldsymbol{\Sigma }
\end{equation*}%
where $c_{\boldsymbol{Y}}=-2\psi ^{\prime }(0).$

For the elliptical model, the parameter to be estimated is $\boldsymbol{\theta }=\left( 
\boldsymbol{\mu }^{T},\boldsymbol{\Sigma }\right) $ whose dimension is $s=m+%
\frac{m\left( m+1\right) }{2}.$ In the following we denote $%
\boldsymbol{\mu (\theta )}$ instead of $\boldsymbol{\mu }$ and $\boldsymbol{%
\Sigma }\left( \boldsymbol{\theta }\right) $ instead of $\boldsymbol{\Sigma }
$, so as  to be consistent with the paper notation.

Let us consider the testing problem 
\begin{equation}
H_{0}:\left( \boldsymbol{\mu }(\boldsymbol{\theta }),\boldsymbol{\Sigma }%
\left( \boldsymbol{\theta }\right) \right) =\left( \boldsymbol{\mu }_{0},%
\boldsymbol{\Sigma }_{0}\right) \text{ versus }H_{1}:\left( \boldsymbol{\mu }%
(\boldsymbol{\theta }),\boldsymbol{\Sigma }\left( \boldsymbol{\theta }%
\right) \right) \neq \left( \boldsymbol{\mu }_{0},\boldsymbol{\Sigma }%
_{0}\right)  \label{117b}
\end{equation}%
where $\boldsymbol{\mu }_{0}$ and $\boldsymbol{\Sigma }_{0}$ are 
known.  The Rao-type test statistic based on the MDPDGE for the elliptical model is given by
(\ref{117b}) if 
\begin{equation*}
R_{\boldsymbol{\tau }}\left( \boldsymbol{\mu }_{0},\boldsymbol{\Sigma }%
_{0}\right) =\frac{1}{n}\boldsymbol{U}_{n}^{\boldsymbol{\tau }}(\boldsymbol{%
\mu }_{0},\boldsymbol{\Sigma }_{0})^{T}\boldsymbol{K}_{\boldsymbol{\tau }%
}\left( \boldsymbol{\mu }_{0},\boldsymbol{\Sigma }_{0}\right) ^{-1}%
\boldsymbol{U}_{n}^{\boldsymbol{\tau }}(\boldsymbol{\mu }_{0},\boldsymbol{%
\Sigma }_{0})
\end{equation*}%
where 
\begin{equation*}
\boldsymbol{U}_{n}^{\boldsymbol{\tau }}(\boldsymbol{\mu }_{0},\boldsymbol{%
\Sigma }_{0}\boldsymbol{)=}\tsum \limits_{i=1}^{n}\frac{1}{\tau +1}%
\boldsymbol{\Psi }_{\boldsymbol{\tau }}(\boldsymbol{y}_{i};\boldsymbol{\mu }%
_{0},\boldsymbol{\Sigma }_{0})
\end{equation*}
with $\boldsymbol{\Psi }_{\boldsymbol{\tau }}(\boldsymbol{y}_{i};\boldsymbol{%
\mu }_{0},\boldsymbol{\Sigma }_{0}\boldsymbol{)}$ are as defined in  (\ref{116b})
and (\ref{103aaa}), respectively, but replacing $\boldsymbol{\Sigma }\left( 
\boldsymbol{\theta }\right) $ by $c_{\boldsymbol{Y}}\boldsymbol{\Sigma }$
and $\boldsymbol{\mu }\left( \boldsymbol{\theta }\right) $ by $\boldsymbol{%
\mu }$.
Then, the null hypothesis in (\ref{117b}) should be rejected if
\begin{equation*}
	R_{\boldsymbol{\tau }}\left( \boldsymbol{\mu }_{0},\boldsymbol{\Sigma }%
	_{0}\right) >\chi _{m+\frac{m(m+1)}{2},\alpha}^{2},
\end{equation*}
with $\chi _{m+\frac{m(m+1)}{2},\alpha}^{2}$ is the $1-\alpha$ upper quantile  of a chi-square with $m+\frac{m(m+1)}{2}$ degrees of freedom.
\end{example}

We finally prove the consistency of the Rao-type test based on
RMDPDGE. To simplify the statement of the next result, we first define the vector
\begin{eqnarray}
\boldsymbol{Y}_{\tau }(\boldsymbol{\theta }) &=&a\text{ }\frac{\tau }{2}%
\left \vert \boldsymbol{\Sigma }\left( \boldsymbol{\theta }\right) \right
\vert ^{-\tau /2}\left \{ -trace\left( \boldsymbol{\Sigma }\left( 
\boldsymbol{\theta }\right) ^{-1}\dfrac{\partial \boldsymbol{\Sigma }\left( 
\boldsymbol{\theta }\right) }{\partial \boldsymbol{\theta }}\right) \right.
\label{117a} \\
&&\left. \exp \left( -\frac{\tau }{2}\left( \boldsymbol{Y}-\boldsymbol{\mu }%
\left( \boldsymbol{\theta }\right) \right) ^{T}\boldsymbol{\Sigma }\left( 
\boldsymbol{\theta }\right) ^{-1}\left( \boldsymbol{Y}-\boldsymbol{\mu }%
\left( \boldsymbol{\theta }\right) \right) \right) +b\text{ }trace\left( 
\boldsymbol{\Sigma }\left( \boldsymbol{\theta }\right) ^{-1}\dfrac{\partial 
\boldsymbol{\Sigma }\left( \boldsymbol{\theta }\right) }{\partial 
\boldsymbol{\theta }}\right) \right.  \notag \\
&&\left. +\text{ }\exp \left( -\frac{\tau }{2}\left( \boldsymbol{Y}-%
\boldsymbol{\mu }\left( \boldsymbol{\theta }\right) \right) ^{T}\boldsymbol{%
\Sigma }\left( \boldsymbol{\theta }\right) ^{-1}\left( \boldsymbol{Y}-%
\boldsymbol{\mu }\left( \boldsymbol{\theta }\right) \right) \right) \right. 
\notag \\
&&\left[ 2\left( \frac{\partial \boldsymbol{\mu }\left( \boldsymbol{\theta }%
\right) }{\partial \boldsymbol{\theta }}\right) ^{T}\boldsymbol{\Sigma }%
\left( \boldsymbol{\theta }\right) ^{-1}\left( \boldsymbol{Y}-\boldsymbol{%
\mu }\left( \boldsymbol{\theta }\right) \right) \right.  \notag \\
&&\left. \left. +\left( \boldsymbol{Y}-\boldsymbol{\mu }\left( \boldsymbol{%
\theta }\right) \right) ^{T}\left( \boldsymbol{\Sigma }\left( \boldsymbol{%
\theta }\right) ^{-1}\dfrac{\partial \boldsymbol{\Sigma }\left( \boldsymbol{%
\theta }\right) }{\partial \boldsymbol{\theta }}\boldsymbol{\Sigma }\left( 
\boldsymbol{\theta }\right) ^{-1}\right) \left( \boldsymbol{Y}-\boldsymbol{%
\mu }\left( \boldsymbol{\theta }\right) \right) \right] \right \} ,  \notag
\end{eqnarray}%
where $\ a$ and $b$ were defined in (\ref{103b}). We can observe that $\frac{%
\partial }{\partial \boldsymbol{\theta }}H_{n}(\boldsymbol{\theta })$ is the
sample mean of a random sample of size $n$ from the $m$-dimensional
population $\boldsymbol{Y}_{\tau }(\boldsymbol{\theta }).$

\begin{theorem}
Let $\boldsymbol{Y}_{1},...,$ $\boldsymbol{Y}_{n}$ be independent and
identically distributed observations from a $m$-dimensional random vector $%
\boldsymbol{Y}$ with $E_{\boldsymbol{\theta }}\left[ \boldsymbol{Y}\right] =%
\boldsymbol{\mu }\left( \boldsymbol{\theta }\right) $ and $Cov_{\boldsymbol{%
\theta }}\left[ \boldsymbol{Y}\right] =\boldsymbol{\Sigma }\left( 
\boldsymbol{\theta }\right) $, $\boldsymbol{\theta }\in \Theta \subset 
\mathbb{R}^{d}.$ 
Let $\boldsymbol{\theta }\in \boldsymbol{\Theta }$ with $%
\boldsymbol{\theta }\neq \boldsymbol{\theta }_{0},$ with $\boldsymbol{\theta 
}_{0}$ defined in (\ref{117}), and let us assume that $E_{\boldsymbol{\theta 
}}\left[ \boldsymbol{Y}_{\tau }(\boldsymbol{\theta }_{0}\boldsymbol{)}\right]
\neq \boldsymbol{0}_{d}$. Then, 
\begin{equation*}
\lim_{n\rightarrow \infty }P_{\boldsymbol{\theta }}\left( R_{\boldsymbol{%
\tau }}\left( \boldsymbol{\theta }_{0}\right) >\chi _{d,\alpha }^{2}\right)
=1.
\end{equation*}
\end{theorem}

\begin{proof}
From the previous results, it holds that
\begin{equation*}
\frac{1}{n}\boldsymbol{U}_{n}^{\boldsymbol{\tau }}(\boldsymbol{\theta }_{0})=%
\frac{1}{n}\tsum \limits_{i=1}^{n}\frac{1}{\tau +1}\boldsymbol{\Psi }_{%
\boldsymbol{\tau }}(\boldsymbol{Y}_{i};\boldsymbol{\theta }_{0}\boldsymbol{)=%
}\frac{1}{\tau +1}\frac{\partial }{\partial \boldsymbol{\theta }}H_{n}^{\tau
}(\boldsymbol{\theta }_{0})\underset{n\rightarrow \infty }{\overset{P}{%
\rightarrow }}\frac{1}{\tau +1}E_{\boldsymbol{\theta }}\left[ \boldsymbol{Y}%
_{\tau }(\boldsymbol{\theta }_{0}\boldsymbol{)}\right] ,
\end{equation*}%
where $\boldsymbol{Y}_{\tau }(\boldsymbol{\theta }_{0}\boldsymbol{)}$ is as
defined in (\ref{117a}). Therefore,%
\begin{align*}
P_{\theta }\left( R_{\boldsymbol{\tau }}\left( \boldsymbol{\theta }%
_{0}\right) >\chi _{d,\alpha }^{2}\right) & =P_{\theta }\left( \tfrac{1}{n}%
R_{\boldsymbol{\tau }}\left( \boldsymbol{\theta }_{0}\right) >\tfrac{1}{n}%
\chi _{d,\alpha }^{2}\right) \\
& \underset{n\rightarrow \infty }{\longrightarrow }\mathrm{I}\left( \frac{1}{%
\left( \tau +1\right) ^{2}}E_{\boldsymbol{\theta }}\left[ \boldsymbol{Y}%
_{\tau }(\boldsymbol{\theta }_{0}\boldsymbol{)}\right] \boldsymbol{K}_{\tau
}^{-1}\left( \boldsymbol{\theta }\right) E_{\boldsymbol{\theta }}^{T}\left[ 
\boldsymbol{Y}_{\tau }(\boldsymbol{\theta }_{0}\boldsymbol{)}\right]
>0\right) =1,
\end{align*}%
where $\mathrm{I}(\cdot )$ is the indicator function.
\end{proof}

A natural question that arises here is how the asymptotic power of
different test statistics considered for testing the hypothesis in (\ref{107}) could be compared. Lehmann (1959) stated that contiguous alternative hypotheses are of great interest in practical use for such purposes, as their associated power functions do not converge to 1. In this regard, we
next derive the asymptotic distribution of $R_{\boldsymbol{\tau }%
}\left( \boldsymbol{\theta }_{0}\right) $ under local Pitman-type
alternative hypotheses of the form $$H_{1,n}:\boldsymbol{\theta }=\boldsymbol{\theta }_{n}:= \boldsymbol{\theta }_{0}+n^{-1/2}\boldsymbol{l},$$ where $\boldsymbol{l}$ is a $d$-dimensional normal vector and $\boldsymbol{\theta }_{0}$ is the closest element to the null hypothesis.
The next result determines the asymptotic  power of the Rao-type test based on RMDPDGE under contiguous alternative hypothesis.

\begin{theorem}
Let $\boldsymbol{Y}_{1},...,\boldsymbol{Y}_{n}$ be independent and
identically distributed observations from a $m$-dimensional random vector $%
\boldsymbol{Y}$ with $E_{\boldsymbol{\theta }}\left[ \boldsymbol{Y}\right] =%
\boldsymbol{\mu }\left( \boldsymbol{\theta }\right) $ and $Cov_{\boldsymbol{%
\theta }}\left[ \boldsymbol{Y}\right] =\boldsymbol{\Sigma }\left( 
\boldsymbol{\theta }\right) $, $\boldsymbol{\theta }\in \Theta \subset 
\mathbb{R}^{d}.$ Under the contiguous alternative hypothesis of the form $$H_{1,n}:%
\boldsymbol{\theta }_{n}=\boldsymbol{\theta }_{0}+n^{-1/2}\boldsymbol{l},$$
the asymptotic distribution of the Rao-type test based on RMDPDGE, $R_{\boldsymbol{\tau }}\left( \boldsymbol{%
\theta }_{0}\right) ,$ is a non-central chi-square distribution with $d$
degrees of freedom and non-centrality parameter given by 
\begin{equation*}
\delta _{\tau }(\boldsymbol{\theta }_{0},\boldsymbol{l})=\boldsymbol{l}^{T}%
\boldsymbol{J}_{\tau }\left( \boldsymbol{\theta }_{0}\right) \boldsymbol{K}%
_{\tau }^{-1}\left( \boldsymbol{\theta }_{0}\right) \boldsymbol{J}_{\tau
}\left( \boldsymbol{\theta }_{0}\right) \mathbf{l}\boldsymbol{.}
\end{equation*}
\end{theorem}

\begin{proof}
Consider the Taylor series expansion%
\begin{equation*}
\frac{1}{\sqrt{n}}\boldsymbol{U}_{n}^{\boldsymbol{\tau }}(\boldsymbol{\theta 
}_{n}\boldsymbol{)}=\frac{1}{\sqrt{n}}\boldsymbol{U}_{n}^{\tau }(\boldsymbol{%
\theta }_{0})+\frac{1}{n}\left. \frac{\partial \boldsymbol{U}_{n}^{\tau }(%
\boldsymbol{\theta })}{\partial \boldsymbol{\theta }^{T}}\right \vert _{%
\boldsymbol{\theta =\theta }_{n}^{\ast }}\mathbf{l}\boldsymbol{,}
\end{equation*}%
where $\boldsymbol{\theta }_{n}^{\ast }$ belongs to the line segment joining 
$\boldsymbol{\theta }_{0}$ and $\boldsymbol{\theta }_{0}+\tfrac{1}{\sqrt{n}}%
\boldsymbol{l}$. Now, by proposition \ref{Proposition2}%
\begin{equation*}
\frac{1}{n}\frac{\partial \boldsymbol{U}_{n}^{\tau }(\boldsymbol{\theta })}{%
\partial \boldsymbol{\theta }^{T}}=\frac{1}{\tau +1}\frac{\partial
^{2}H_{n}^{\tau }\left( \boldsymbol{\theta }\right) }{\partial \boldsymbol{%
\theta }\text{ }\boldsymbol{\partial \theta }^{T}}\underset{n\longrightarrow
\infty }{\overset{\mathcal{P}}{\longrightarrow }}-\boldsymbol{J}_{%
\boldsymbol{\tau }}(\boldsymbol{\theta })
\end{equation*}%
Therefore, 
\begin{equation*}
\left. \frac{1}{\sqrt{n}}\boldsymbol{U}_{n}^{\tau }(\boldsymbol{\theta }%
)\right \vert _{\boldsymbol{\theta =\theta }_{0}\boldsymbol{+}n^{-1/2}%
\boldsymbol{d}}\underset{n\rightarrow \infty }{\overset{\mathcal{L}}{%
\longrightarrow }}\mathcal{N}\left( -\boldsymbol{J}_{\tau }\left( 
\boldsymbol{\theta }_{0}\right) \boldsymbol{l},\text{ }\boldsymbol{K}_{\tau
}\left( \boldsymbol{\theta }_{0}\right) \right) ,
\end{equation*}%
and 
\begin{equation*}
R_{\boldsymbol{\tau }}\left( \boldsymbol{\theta }_{0}\right) \underset{%
n\rightarrow \infty }{\overset{\mathcal{L}}{\longrightarrow }}\chi
_{p}^{2}\left( \delta _{\tau }(\boldsymbol{\theta }_{0},\boldsymbol{l}%
)\right) ,
\end{equation*}%
with $\delta _{\tau }(\boldsymbol{\theta }_{0},\boldsymbol{d})$ $\ $given by%
\begin{equation*}
\delta _{\tau }(\boldsymbol{\theta }_{0},\boldsymbol{l})=\boldsymbol{l}^{T}%
\boldsymbol{J}_{\tau }\left( \boldsymbol{\theta }_{0}\right) \boldsymbol{K}%
_{\tau }^{-1}\left( \boldsymbol{\theta }_{0}\right) \boldsymbol{J}_{\tau
}\left( \boldsymbol{\theta }_{0}\right) \boldsymbol{l}.
\end{equation*}%
\begin{equation*}
\delta _{\tau }(\boldsymbol{\theta }_{0},\boldsymbol{l})=\boldsymbol{l}^{T}%
\boldsymbol{J}_{\tau }\left( \boldsymbol{\theta }_{0}\right) \boldsymbol{K}%
_{\tau }^{-1}\left( \boldsymbol{\theta }_{0}\right) \boldsymbol{J}_{\tau
}\left( \boldsymbol{\theta }_{0}\right) \boldsymbol{l}.
\end{equation*}
\end{proof}

\begin{remark}
The previous result can be used for defining an approximation to the power
function under any alternative hypothesis, $\boldsymbol{\theta} \in \Theta \setminus \Theta_0,$ given as
\begin{equation*}
\boldsymbol{\theta }=\boldsymbol{\theta }-\boldsymbol{\theta }_{0}+%
\boldsymbol{\theta }_{0}=\sqrt{n}\frac{1}{\sqrt{n}}\left( \boldsymbol{\theta 
}-\boldsymbol{\theta }_{0}\right) +\boldsymbol{\theta }_{0}=\boldsymbol{%
\theta }_{0}+n^{-1/2}\boldsymbol{l}
\end{equation*}%
with $\boldsymbol{l=}\sqrt{n}\left( \boldsymbol{\theta }-\boldsymbol{\theta }%
_{0}\right) .$
\end{remark}

\begin{remark}
The family of Rao-type tests, $R_{\boldsymbol{\tau }}\left( \boldsymbol{%
\theta }_{0}\right) ,$ presented in this Section for simple null hypothesis
can be extended to composite null hypothesis. If we are interested in
testing $H_{0}:\boldsymbol{\theta }\in \boldsymbol{\Theta }_{0}=\left \{ 
\boldsymbol{\theta \in }\Theta /\text{ }\boldsymbol{g}(\boldsymbol{\theta )}=%
\boldsymbol{0}_{r}\right \} $ we can consider the family of Rao-type tests
given by 
\begin{equation}
R_{\boldsymbol{\tau }}\left( \widetilde{\boldsymbol{\theta }}_{G}^{\tau
}\right) =\frac{1}{n}\boldsymbol{U}_{n}^{\tau }(\widetilde{\boldsymbol{%
\theta }}_{G}^{\tau })^{T}\boldsymbol{Q}_{\boldsymbol{\tau }}(\widetilde{%
\boldsymbol{\theta }}_{G}^{\tau })\left[ \boldsymbol{Q}_{\boldsymbol{\tau }}(%
\widetilde{\boldsymbol{\theta }}_{G}^{\tau })\boldsymbol{K}_{\tau }\left( 
\widetilde{\boldsymbol{\theta }}_{G}^{\tau }\right) \boldsymbol{Q}_{%
\boldsymbol{\tau }}(\widetilde{\boldsymbol{\theta }}_{G}^{\tau })\right]
^{-1}\boldsymbol{Q}_{\boldsymbol{\tau }}(\widetilde{\boldsymbol{\theta }}%
_{G}^{\tau })^{T}\boldsymbol{U}_{n}^{\tau }(\widetilde{\boldsymbol{\theta }}%
_{G}^{\tau }).  \label{117aaaa}
\end{equation}%
However, the extension of the presented results for the family of robust test statistics defined in (\ref{117aaaa}) is not trivial, and it will be established in future research.

In partircular, the simple null hypothesis in (\ref{117})
can be written as a composite null hypothesis with $\boldsymbol{g}(%
\boldsymbol{\theta )=\theta }-\boldsymbol{\theta }_{0}.$ In this case, $%
\boldsymbol{G}(\boldsymbol{\theta })$ reduces to the identity matrix of dimension $p,$ 
the restricted estimator $\widetilde{\boldsymbol{\theta }}_{G}^{\tau }$ coincides with $\boldsymbol{\theta }_{0}$ and the Rao-type test statistic in (\ref{117aaaa}) $%
R_{\boldsymbol{\tau }}\left( \widetilde{\boldsymbol{\theta }}_{G}^{\tau
}\right), $  coincides with the proposed $R_{\boldsymbol{\tau }%
}\left( \boldsymbol{\theta }_{0}\right) $ given in (\ref{117AAA}).For multidimensional normal populations Martin (2021) developed Rao-type test statistics based on the RMDPDE.

\end{remark}

\subsection{Rao-type tests based on MDPDGE for univariate distributions} 

Let $Y_{1},....,Y_{n}$ a random sample from the population $Y,$ with 
\begin{equation*}
E\left[ Y\right] =\mu \left( \theta \right) \text{ and }Var\left[ Y\right]
=\sigma ^{2}\left( \theta \right) .
\end{equation*}

Based on (\ref{116a}) the estimating equation is given by%
\begin{equation*}
\tsum \limits_{i=1}^{n}\Psi _{\tau }\left( y_{i},\theta \right) =0
\end{equation*}%
with

\begin{eqnarray}
\Psi _{\tau }\left( y_{i},\theta \right) &=&\frac{\left( \tau +1\right)
\left( \sigma ^{2}\left( \theta \right) \right) ^{-\tau /2}}{2\left( 2\pi
\right) ^{\tau /2}}\left \{ \left[ -\frac{\partial \log \sigma ^{2}\left(
\theta \right) }{\partial \theta }+\frac{\partial \log \sigma ^{2}\left(
\theta \right) }{\partial \theta }\right. \right.  \label{118} \\
&&\left. (\ref{117aaaa})+2\frac{\partial \mu \left( \theta \right) }{%
\partial \theta }\left( y_{i}-\mu \left( \theta \right) \right) \frac{1}{%
\sigma ^{2}\left( \theta \right) }\right]  \notag \\
&&\left. \exp \left( -\frac{\tau }{2\sigma ^{2}\left( \theta \right) }\left(
y_{i}-\mu \left( \theta \right) \right) ^{2}\right) +\frac{\tau }{\left(
1+\tau \right) ^{3/2}}\frac{\partial \log \sigma ^{2}\left( \theta \right) }{%
\partial \theta }\right \} .  \notag
\end{eqnarray}%
Moreover the expressions of $J_{\tau }\left( \theta \right) $ and $K_{\tau }\left(
\theta \right) $ are, respectively, given by%
\begin{eqnarray*}
J_{\tau }\left( \theta \right) &=&\frac{1}{\left( 2\pi \sigma \left( \theta
\right) ^{2}\right) ^{\frac{\tau }{2}}}\frac{1}{\left( 1+\tau \right) ^{5/2}}%
\left[ \left( \tau +1\right) \sigma ^{-2}\left( \theta \right) \left( \frac{%
\partial \mu \left( \theta \right) }{\partial \theta }\right) ^{2}+\frac{%
\tau ^{2}}{4}\left( \frac{\partial \log \sigma ^{2}\left( \theta \right) }{%
\partial \theta }\right) ^{2}\right. \\
&&\left. +\frac{1}{2}\left( \frac{\partial \log \sigma ^{2}\left( \theta
\right) }{\partial \theta }\right) ^{2}\right]
\end{eqnarray*}%
and 
\begin{eqnarray}
K_{\tau }\left( \theta \right) &=&\left( \frac{1}{\left( 2\pi \right)
^{1/2}\sigma \left( \theta \right) }\right) ^{2\tau }\left \{ \frac{1}{%
\left( 1+2\tau \right) ^{5/2}}\left[ \tau ^{2}\left( \frac{\partial \log
\sigma ^{2}\left( \theta \right) }{\partial \theta }\right) ^{2}\right.
\right.  \label{119} \\
&&\left. +\left( 1+2\tau \right) \sigma ^{-2}\left( \theta \right) \left( 
\frac{\partial \mu \left( \theta \right) }{\partial \theta }\right) ^{2}+%
\frac{1}{2}\left( \frac{\partial \log \sigma ^{2}\left( \theta \right) }{%
\partial \theta }\right) ^{2}\right]  \notag \\
&&\left. -\frac{\tau ^{2}}{4\left( 1+\tau \right) ^{3}}\left( \frac{\partial
\log \sigma ^{2}\left( \theta \right) }{\partial \theta }\right) ^{2}\right
\} .  \notag
\end{eqnarray}%
Therefore, if we are interesting in testing 
\begin{equation*}
H_{0}:\theta =\theta _{0}\text{ versus }H_{1}:\theta \neq \theta _{0},
\end{equation*}%
 the Rao-type tests based on RMDPDGE is given by, $$R_{\boldsymbol{\tau }}\left( \theta _{0}\right) =\frac{1}{n}U_{n}^{\boldsymbol{\tau }}(\theta _{0})^{2}K_{\boldsymbol{\tau }}\left( \theta _{0}\right) ^{-1},$$
 where 
 \begin{equation*}
 	U_{n}^{\boldsymbol{\tau }}(\theta _{0})=\frac{1}{\tau +1}\tsum%
 	\limits_{i=1}^{n}\Psi _{\tau }\left( y_{i},\theta \right)
 \end{equation*}
 and $\Psi _{\tau }\left( y_{i},\theta \right) $ and  $K_{\tau }\left( \theta \right) $ are given as in (\ref{118}) and (\ref{119}).
The null hypothesis is rejected if 
\begin{equation*}
R_{\boldsymbol{\tau }}\left( \theta _{0}\right) >\chi _{1,\alpha }^{2},
\end{equation*}%
where $\chi _{1,\alpha }^{2}$ is the upper $1-\alpha$ quantile of a chi-square distribution with 1 degree of freedom.

We finally derive explicit expressions of the Rao-type test statistics under Poisson and exponential models.

\subsubsection{Poisson Model}

We us assume that the random variable $Y$ is Poisson with parameter $%
\theta .$ In this case it is well known that $E\left[ Y\right] =Var\left[ Y\right] =\theta,$ and so the RMDPDGE, for $\tau >0,$ is given by, 
\begin{equation*}
\widehat{\theta }_{G}^{\tau }=\arg \max_{\theta }\left \{ \frac{\tau +1}{%
\tau \left( 2\pi \theta \right) ^{\frac{\tau }{2}}}\left( \frac{1}{n}\tsum
\limits_{i=1}^{n}\exp \left( -\frac{\tau }{2\theta }\left( y_{i}-\theta
\right) ^{2}\right) -\frac{\tau }{\left( 1+\tau \right) ^{3/2}}\right) -%
\frac{1}{\tau }\right \}.
\end{equation*}%
At $\tau = 0,$ the RMDPDGE reduces to the restricted MLE,
\begin{equation*}
\widehat{\theta }_{G}=\arg \max_{\theta }\left \{ -\frac{1}{2}\log 2\pi -%
\frac{1}{2}\log \theta -\frac{1}{n}\tsum \limits_{i=1}^{n}\frac{1}{2\theta }%
\left( y_{i}-\theta \right) ^{2}\right \} .
\end{equation*}

On the other hand the score function $\Psi _{\tau }\left( \cdot \right) $ is
\begin{equation*}
\Psi _{\tau }\left( y_{i},\theta \right) =\frac{\tau +1}{2\left( 2\pi \theta
\right) ^{\frac{\tau }{2}}\theta ^{2}}\left \{ \left( -2\theta
^{2}+y_{i}^{2}\right) \exp \left( -\frac{\tau }{2\theta }\left( y_{i}-\theta
\right) ^{2}\right) +\frac{\tau \theta }{\left( 1+\tau \right) ^{\frac{3}{2}}%
}\right \}
\end{equation*}%
and naturally at $\tau =0,$ we obtain the score function of the MLE presented in Zhang (2019)
\begin{equation*}
\Psi _{0}\left( y_{i},\theta \right) =\frac{1}{2\theta ^{2}}\left( -2\theta
^{2}+y_{i}^{2}\right) .
\end{equation*}

On the other hand, the matrix $K_{\tau }\left( \theta \right)$ under the Poisson model has the explicit expression
\begin{equation*}
K_{\tau }\left( \theta \right) =\left( \frac{1}{2\pi }\right) ^{\tau }\frac{1%
}{2\theta ^{2+\tau }}\left \{ \frac{1}{\left( 1+2\tau \right) ^{5/2}}\left(
\left( 2\tau ^{2}+2\theta +4\theta \tau +1\right) -\frac{\tau ^{2}}{2\left(
1+\tau \right) ^{3}}\right) \right \}
\end{equation*}%
and hence the  Rao-type tests based on RMDPDGE,
$R_{\boldsymbol{\tau }}\left( \theta _{0}\right) ,$ for testing simple null hypothesis 
is given,  for $\tau >0,$  by 
\begin{eqnarray*}
R_{\boldsymbol{\tau }}\left( \theta _{0}\right) &=&\frac{1}{n}\frac{1}{%
\left( 2\left( 2\pi \theta \right) ^{\frac{\tau }{2}}\theta ^{2}\right) ^{2}}%
\left( \tsum \limits_{i=1}^{n}\left( \left( -2\theta
_{0}^{2}+y_{i}^{2}\right) \exp \left( -\frac{\tau }{2\theta _{0}}\left(
y_{i}-\theta _{0}\right) ^{2}\right) +\frac{\tau \theta }{\left( 1+\tau
\right) ^{\frac{3}{2}}}\right) \right) ^{2} \\
&&\times \left( 2\pi \right)^{\tau } (2\theta ^{2+\tau}) \left( 1+2\tau \right) ^{5/2} \left \{\left(
\left( 2\tau ^{2}+2\theta +4\theta \tau +1\right) -\frac{\tau ^{2}}{2\left(
	1+\tau \right) ^{3}}\right) \right \} ^{-1}.
\end{eqnarray*}%
Again for $\tau =0$ we get the expression of the classical Rao test based on the MLE,
\begin{equation*}
R_{0}\left( \theta _{0}\right) =\frac{1}{4n}\left( \tsum
\limits_{i=1}^{n}\left( \frac{-2\theta _{0}^{2}+y_{i}^{2}}{\theta _{0}^{2}}%
\right) \right) ^{2} \frac{2\theta _{0}^{2}}{2\theta _{0}+1}.
\end{equation*}%

In practical use, the  null hypothesis is rejected if 
\begin{equation*}
R_{\boldsymbol{\tau }}\left( \theta _{0}\right) >\chi _{1,\alpha }^{2}.
\end{equation*}

\subsubsection{Exponential model}

Let assume now that the random variable $Y$ comes from an exponential distribution with probability density function,
\begin{equation}
f_{\theta }(x)=\frac{1}{\theta }\exp \left( -\frac{x}{\theta }\right) ,\text{
}x>0.  \label{119a}
\end{equation}

In this case the true mean and variance are given $E\left[ Y\right] =\theta $ and $Var\left[ Y\right] =\theta^{2}.$ The RMDPDGE under the exponential model, for $\tau >0,$ is given by 
\begin{equation*}
\widehat{\theta }_{G}^{\tau }=\arg \max_{\theta }\left \{ \frac{\tau +1}{%
\tau }\left( \frac{1}{\theta \sqrt{2\pi }}\right) ^{\tau }\left( \frac{1}{n}%
\tsum \limits_{i=1}^{n}\exp \left( -\frac{\tau }{2}\left( \frac{y_{i}-\theta 
}{\theta }\right) ^{2}\right) -\frac{\tau }{\left( 1+\tau \right) ^{3/2}}%
\right) -\frac{1}{\tau }\right \} ,
\end{equation*}%
and for $\tau \rightarrow 0,$ we have
\begin{equation*}
\widehat{\theta }_{G}=\arg \max_{\theta }\left \{ -\frac{1}{2}\log 2\pi
-\log \theta -\frac{1}{n}\tsum \limits_{i=1}^{n}\frac{1}{2}\left( \frac{%
y_{i}-\theta }{\theta }\right) ^{2}\right \} .
\end{equation*}

On the other hand, the score function is
\begin{equation*}
\Psi _{\tau }\left( y_{i},\theta \right) =\frac{\left( \tau +1\right) }{%
\theta ^{\tau +3}\left( \sqrt{2\pi }\right) ^{\tau }}\left \{ \left(
y_{i}^{2}-y_{i}\theta -\theta ^{2}\right) \exp \left( -\frac{\tau }{2}\left( 
\frac{y_{i}-\theta }{\theta }\right) ^{2}\right) +\frac{\tau \theta ^{2}}{%
\left( 1+\tau \right) ^{\frac{3}{2}}}\right \} ,
\end{equation*}%
and for $\tau =0,$ we recover the score function of the Gaussian MLE,
\begin{equation*}
\Psi _{0}\left( y_{i},\theta \right) =\frac{1}{\theta ^{3}}\left(
y_{i}^{2}-y_{i}\theta +\theta ^{2}\right) .
\end{equation*}%
The matrix $K_{\tau }\left( \theta \right) $ has the expression 
\begin{equation*}
K_{\tau }\left( \theta \right) =\frac{1}{\left( 2\pi \right) ^{\tau }\theta
^{2\left( \tau +1\right) }}\left \{ \frac{1}{\left( 1+2\tau \right) ^{5/2}}%
\left( 4\tau ^{2}+2\tau +3\right) -\frac{\tau ^{2}}{\left( 1+\tau \right)
^{3}}\right \}
\end{equation*}%
and at $\tau=0$
\begin{equation*}
K_{0}\left( \theta \right) =\frac{2}{\theta ^{2}}.
\end{equation*}

Correspondingly, the Rao-type tests based on RMDPDGE for testing
\begin{equation*}
H_{0}:\theta =\theta _{0}\text{ versus }H_{1}:\theta \neq \theta _{0},
\end{equation*}%
is given,  for $\tau >0,$  by
\begin{eqnarray}
R_{\boldsymbol{\tau }}\left( \theta _{0}\right) &=&\frac{1}{n}\frac{1}{%
\theta _{0}^{2\tau +6}\left( 2\pi \right) ^{\tau }}\left( \tsum
\limits_{i=1}^{n}\left \{ \left( y_{i}^{2}-y_{i}\theta _{0}-\theta
_{0}^{2}\right) \exp \left( -\frac{\tau }{2}\left( \frac{y_{i}-\theta _{0}}{%
\theta _{0}}\right) ^{2}\right) \right. \right.  \label{120} \\
&&\left. \left. +\frac{\tau \theta _{0}^{2}}{\left( 1+\tau \right) ^{\frac{3%
}{2}}}\right \} \right) ^{2}\times 2n\theta_{0}^{4}  \left( \tsum \limits_{i=1}^{n}\left \{ \left( y_{i}^{2}-y_{i}\theta_{0}-\theta _{0}^{2}\right) \right \} \right) ^{-2}.
\notag
\end{eqnarray}%

\section{Simulation study}

We analyze here the performance of the Rao-type tests based on the MDPDGE, $%
R_{\boldsymbol{\tau }}\left( \theta_{0}\right) ,$ in terms of robustness and
efficiency. We compare the proposed general method assuming Gaussian
distribution with Rao-type test statistics based on the true parametric
distribution underlying the data. 

We consider the exponential model with density function $f_{\theta _{0}}(x)$
given in (\ref{119a}). For the exponential model, the Rao-type test
statistics based on MDPDGE is, for $\tau >0,$ as given in (\ref{120}) and
for $\tau =0$ as given in (\ref{120a}). To evaluate the robustness of the
tests we generate samples from an exponential mixture, 
\begin{equation*}
f_{\theta _{0}}^{\varepsilon }(x)=(1-\varepsilon )f_{\theta
_{0}}(x)+\varepsilon f_{2\theta _{0}}(x),
\end{equation*}%
where $\theta _{0}$ denotes the parameter of the exponential distribution
and $\varepsilon $ is the contamination proportion. The uncontaminated model
is thus obtained by setting $\varepsilon =0.$

For comparison purposes we have also considered the robust Rao-type tests
based on the restricted MDPDE, introduced and studied in Basu et al (2022b).
The efficiency loss caused by the Gaussian assumption should be advertised
by the poorer performance of the Rao-type tests based on the restricted
MDPDGE with respect to their analogous based on the restricted MDPDE. For
the exponential model, the family Rao-type test statistics based on the
restricted MDPDE is given, for $\beta >0,$ as 
\begin{equation*}
S_{n}^{\beta }(\theta _{0})=\left( \frac{4\beta ^{2}+1}{(2\beta +1)^{3}}-%
\frac{\beta ^{2}}{(\beta +1)^{4}}\right) ^{-1}\frac{1}{n}\left( \frac{1}{%
\theta _{0}}\sum_{i=1}^{n}\left( y_{i}-\theta _{0}\right) \exp \left( -\frac{%
\beta y_{i}}{\theta _{0}}\right) +\frac{n\beta }{(\beta +1)^{2}}\right) ^{2}.
\end{equation*}%
For $\beta =0$, the above test reduces to the classical Rao test given by 
\begin{equation*}
S_{n}\left( \theta _{0}\right) =S_{\beta =0,n}\left( \theta _{0}\right)
=\left( \sqrt{n}\frac{\bar{X}_{n}-\theta _{0}}{\theta _{0}}\right) ^{2}.
\end{equation*}

We consider the testing problem 
\begin{equation*}
H_{0}:\theta _{0}=2\text{ vs }H_{1}:\theta \neq 2.
\end{equation*}%
and we empirically examine the level and power of both Rao-type test
statistics, the usual test based on the parametric model and the
Gaussian-based test by setting the true value of the parameter $\theta
_{0}=2 $ and $\theta _{0}=1,$ respectively. Different sample sizes were
considered, namely $n=10,$ $20,$ $30,$ $40,$ $50,$ $60,$ $70,$ $80,$ $90,$ $%
100$ and $200,$ but simulation results were quite similar and so, for
brevity, we only report here results for $n=20$ and $n=40.$

The empirical level of the test is computed 
\begin{equation*}
\widehat{\alpha }_{n}\left( \varepsilon \right) =\frac{\text{Number of times}%
\left \{ R_{n}^{\tau }\left( \theta _{0}\right) \text{ (or }S_{n}^{\beta
}\left( \theta _{0}\right) )>\chi _{1,0.05}^{2}=3.84146\right \} }{\text{%
Number of simulated samples}}.
\end{equation*}%
We set $\varepsilon =0\%,5\%,$ $10\%$ and $20\%$ of contamination
proportions and perform the Monte-Carlo study over $R=10000$ replications.
The tuning parameters $\tau $ and $\beta $ are fixed from a grid of values,
namely $\{0,0.1,...,0.7\}.$

Simulation results are presented in Tables 1 and 2 for $n=20$ and $n=40,$
respectively. The empirical powers are denoted by $\widehat{\pi }_{n}\left(
\varepsilon \right) $ and we have considered $\varepsilon =0\%,$ $10\%,$ $%
10\%$ and $20\%.$ The robustness advantage in terms of level of both
Rao-type tests considered, $R_{\boldsymbol{\tau }}\left( \theta _{0}\right) $%
{\small \ and }$S_{n}^{\beta }(\theta _{0})$ with positive values of the
turning parameter with respect to the test statistics with $\tau =0$ and $%
\beta =0$ is clearly shown, as their simulated levels are closer to the
nominal in the presence of contamination.

\begin{equation*}
\begin{tabular}{|l||llllllll|}
\hline
$\tau $ & $\widehat{\alpha }_{20}\left( 0\right) $ & \multicolumn{1}{|l}{$%
\widehat{\alpha }_{20}\left( 0.05\right) $} & \multicolumn{1}{|l}{$\widehat{%
\alpha }_{20}\left( 0.10\right) $} & \multicolumn{1}{|l}{$\widehat{\alpha }%
_{20}\left( 0.20\right) $} & \multicolumn{1}{|l}{$\widehat{\pi }_{20}\left(
0\right) $} & \multicolumn{1}{|l}{$\widehat{\pi }_{20}\left( 0.1\right) $} & 
\multicolumn{1}{|l}{$\widehat{\pi }_{20}\left( 0.15\right) $} & 
\multicolumn{1}{|l|}{$\widehat{\pi }_{20}\left( 0.20\right) $} \\ \hline
0.0 & 0.2601 & \multicolumn{1}{|l}{0.3093} & \multicolumn{1}{|l}{0.3453} & 
\multicolumn{1}{|l}{0.4661} & \multicolumn{1}{|l}{0.9278} & 
\multicolumn{1}{|l}{0.6791} & \multicolumn{1}{|l}{0.6887} & 
\multicolumn{1}{|l|}{0.5088} \\ 
0.1 & 0.1895 & \multicolumn{1}{|l}{0.1748} & \multicolumn{1}{|l}{0.1561} & 
\multicolumn{1}{|l}{0.1989} & \multicolumn{1}{|l}{0.9544} & 
\multicolumn{1}{|l}{0.7213} & \multicolumn{1}{|l}{0.7301} & 
\multicolumn{1}{|l|}{0.0595} \\ 
0.2 & 0.2120 & \multicolumn{1}{|l}{0.1776} & \multicolumn{1}{|l}{0.1417} & 
\multicolumn{1}{|l}{0.1174} & \multicolumn{1}{|l}{0.9747} & 
\multicolumn{1}{|l}{0.8398} & \multicolumn{1}{|l}{0.8430} & 
\multicolumn{1}{|l|}{0.5095} \\ 
0.3 & 0.2532 & \multicolumn{1}{|l}{0.2113} & \multicolumn{1}{|l}{0.1660} & 
\multicolumn{1}{|l}{0.1275} & \multicolumn{1}{|l}{0.9826} & 
\multicolumn{1}{|l}{0.8963} & \multicolumn{1}{|l}{0.8961} & 
\multicolumn{1}{|l|}{0.7301} \\ 
0.4 & 0.2963 & \multicolumn{1}{|l}{0.2447} & \multicolumn{1}{|l}{0.1986} & 
\multicolumn{1}{|l}{0.1471} & \multicolumn{1}{|l}{0.9863} & 
\multicolumn{1}{|l}{0.9228} & \multicolumn{1}{|l}{0.9257} & 
\multicolumn{1}{|l|}{0.7893} \\ 
0.5 & 0.3243 & \multicolumn{1}{|l}{0.2773} & \multicolumn{1}{|l}{0.2307} & 
\multicolumn{1}{|l}{0.1695} & \multicolumn{1}{|l}{0.9875} & 
\multicolumn{1}{|l}{0.9363} & \multicolumn{1}{|l}{0.9386} & 
\multicolumn{1}{|l|}{0.8254} \\ 
0.6 & 0.3512 & \multicolumn{1}{|l}{0.3055} & \multicolumn{1}{|l}{0.2599} & 
\multicolumn{1}{|l}{0.1899} & \multicolumn{1}{|l}{0.9885} & 
\multicolumn{1}{|l}{0.9441} & \multicolumn{1}{|l}{0.9437} & 
\multicolumn{1}{|l|}{0.8434} \\ 
0.7 & 0.3751 & \multicolumn{1}{|l}{0.3258} & \multicolumn{1}{|l}{0.2762} & 
\multicolumn{1}{|l}{0.2060} & \multicolumn{1}{|l}{0.9884} & 
\multicolumn{1}{|l}{0.9466} & \multicolumn{1}{|l}{0.9469} & 
\multicolumn{1}{|l|}{0.8541} \\ \hline\hline
$\beta $ &  &  &  &  &  &  &  &  \\ \hline
0.0 & 0.0453 & \multicolumn{1}{|l}{0.0682} & \multicolumn{1}{|l}{0.1048} & 
\multicolumn{1}{|l}{0.1909} & \multicolumn{1}{|l}{0.7200} & 
\multicolumn{1}{|l}{0.4365} & \multicolumn{1}{|l}{0.4384} & 
\multicolumn{1}{|l|}{0.2323} \\ 
0.1 & 0.0476 & \multicolumn{1}{|l}{0.0602} & \multicolumn{1}{|l}{0.0780} & 
\multicolumn{1}{|l}{0.1417} & \multicolumn{1}{|l}{0.7799} & 
\multicolumn{1}{|l}{0.5223} & \multicolumn{1}{|l}{0.5267} & 
\multicolumn{1}{|l|}{0.3029} \\ 
0.2 & 0.0498 & \multicolumn{1}{|l}{0.0552} & \multicolumn{1}{|l}{0.0667} & 
\multicolumn{1}{|l}{0.1103} & \multicolumn{1}{|l}{0.7922} & 
\multicolumn{1}{|l}{0.5751} & \multicolumn{1}{|l}{0.5780} & 
\multicolumn{1}{|l|}{0.3558} \\ 
0.3 & 0.0494 & \multicolumn{1}{|l}{0.0517} & \multicolumn{1}{|l}{0.0584} & 
\multicolumn{1}{|l}{0.0897} & \multicolumn{1}{|l}{0.7882} & 
\multicolumn{1}{|l}{0.5997} & \multicolumn{1}{|l}{0.6024} & 
\multicolumn{1}{|l|}{0.3878} \\ 
0.4 & 0.0489 & \multicolumn{1}{|l}{0.0505} & \multicolumn{1}{|l}{0.0535} & 
\multicolumn{1}{|l}{0.0773} & \multicolumn{1}{|l}{0.7779} & 
\multicolumn{1}{|l}{0.6067} & \multicolumn{1}{|l}{0.6058} & 
\multicolumn{1}{|l|}{0.4106} \\ 
0.5 & 0.0494 & \multicolumn{1}{|l}{0.0498} & \multicolumn{1}{|l}{0.0504} & 
\multicolumn{1}{|l}{0.0692} & \multicolumn{1}{|l}{0.7634} & 
\multicolumn{1}{|l}{0.6048} & \multicolumn{1}{|l}{0.6037} & 
\multicolumn{1}{|l|}{0.4221} \\ 
0.6 & 0.0491 & \multicolumn{1}{|l}{0.0504} & \multicolumn{1}{|l}{0.0497} & 
\multicolumn{1}{|l}{0.0647} & \multicolumn{1}{|l}{0.7492} & 
\multicolumn{1}{|l}{0.6008} & \multicolumn{1}{|l}{0.5986} & 
\multicolumn{1}{|l|}{0.4265} \\ 
0.7 & 0.0502 & \multicolumn{1}{|l}{0.0495} & \multicolumn{1}{|l}{0.0494} & 
\multicolumn{1}{|l}{0.0613} & \multicolumn{1}{|l}{0.7348} & 
\multicolumn{1}{|l}{0.5932} & \multicolumn{1}{|l}{0.5919} & 
\multicolumn{1}{|l|}{0.4259} \\ \hline
\multicolumn{9}{|l|}{\small Table 1. Simulated sizes and powers for
different contamination proportions and different tuning} \\ \hline
\multicolumn{9}{|l|}{{\small parameters }$\tau ,\beta =0,0.1,...,0.7${\small %
\ for the Rao-type tests }$R_{\boldsymbol{\tau }}\left( \theta _{0}\right) $%
{\small \ and }$S_{20}^{\beta }(\theta _{0})${\small \ for }$n=20.$} \\ 
\hline
\end{tabular}%
\end{equation*}

\begin{equation*}
\begin{tabular}{|l||llllllll|}
\hline
$\tau $ & $\widehat{\alpha }_{40}\left( 0\right) $ & \multicolumn{1}{|l}{$%
\widehat{\alpha }_{40}\left( 0.05\right) $} & \multicolumn{1}{|l}{$\widehat{%
\alpha }_{40}\left( 0.10\right) $} & \multicolumn{1}{|l}{$\widehat{\alpha }%
_{40}\left( 0.20\right) $} & \multicolumn{1}{|l}{$\widehat{\pi }_{40}\left(
0\right) $} & \multicolumn{1}{|l}{$\widehat{\pi }_{40}\left( 0.1\right) $} & 
\multicolumn{1}{|l}{$\widehat{\pi }_{40}\left( 0.15\right) $} & 
\multicolumn{1}{|l|}{$\widehat{\pi }_{40}\left( 0.20\right) $} \\ \hline
0.0 & 0.3014 & \multicolumn{1}{|l}{0.3588} & \multicolumn{1}{|l}{0.4407} & 
\multicolumn{1}{|l}{0.5919} & \multicolumn{1}{|l}{0.9948} & 
\multicolumn{1}{|l}{0.8064} & \multicolumn{1}{|l}{0.7591} & 
\multicolumn{1}{|l|}{0.5957} \\ 
0.1 & 0.2393 & \multicolumn{1}{|l}{0.1934} & \multicolumn{1}{|l}{0.1757} & 
\multicolumn{1}{|l}{0.2032} & \multicolumn{1}{|l}{0.9991} & 
\multicolumn{1}{|l}{0.9540} & \multicolumn{1}{|l}{0.9229} & 
\multicolumn{1}{|l|}{0.7712} \\ 
0.2 & 0.4257 & \multicolumn{1}{|l}{0.2559} & \multicolumn{1}{|l}{0.1970} & 
\multicolumn{1}{|l}{0.1317} & \multicolumn{1}{|l}{0.9995} & 
\multicolumn{1}{|l}{0.9916} & \multicolumn{1}{|l}{0.9846} & 
\multicolumn{1}{|l|}{0.9204} \\ 
0.3 & 0.4257 & \multicolumn{1}{|l}{0.3485} & \multicolumn{1}{|l}{0.2782} & 
\multicolumn{1}{|l}{0.1753} & \multicolumn{1}{|l}{0.9997} & 
\multicolumn{1}{|l}{0.9997} & \multicolumn{1}{|l}{0.9953} & 
\multicolumn{1}{|l|}{0.9694} \\ 
0.4 & 0.5021 & \multicolumn{1}{|l}{0.4294} & \multicolumn{1}{|l}{0.3572} & 
\multicolumn{1}{|l}{0.2388} & \multicolumn{1}{|l}{0.9999} & 
\multicolumn{1}{|l}{0.9989} & \multicolumn{1}{|l}{0.9978} & 
\multicolumn{1}{|l|}{0.9851} \\ 
0.5 & 0.5642 & \multicolumn{1}{|l}{0.4920} & \multicolumn{1}{|l}{0.4253} & 
\multicolumn{1}{|l}{0.2993} & \multicolumn{1}{|l}{0.9999} & 
\multicolumn{1}{|l}{0.9992} & \multicolumn{1}{|l}{0.9986} & 
\multicolumn{1}{|l|}{0.9908} \\ 
0.6 & 0.6084 & \multicolumn{1}{|l}{0.5415} & \multicolumn{1}{|l}{0.4742} & 
\multicolumn{1}{|l}{0.3491} & \multicolumn{1}{|l}{1.0000} & 
\multicolumn{1}{|l}{0.9992} & \multicolumn{1}{|l}{0.9994} & 
\multicolumn{1}{|l|}{0.9935} \\ 
0.7 & 0.6416 & \multicolumn{1}{|l}{0.5755} & \multicolumn{1}{|l}{0.5081} & 
\multicolumn{1}{|l}{0.3831} & \multicolumn{1}{|l}{1.0000} & 
\multicolumn{1}{|l}{0.9994} & \multicolumn{1}{|l}{0.9994} & 
\multicolumn{1}{|l|}{0.9948} \\ \hline\hline
$\beta $ &  &  &  &  &  &  &  &  \\ \hline
0.0 & 0.0467 & \multicolumn{1}{|l}{0.0758} & \multicolumn{1}{|l}{0.1309} & 
\multicolumn{1}{|l}{0.2728} & \multicolumn{1}{|l}{0.9838} & 
\multicolumn{1}{|l}{0.8093} & \multicolumn{1}{|l}{0.7483} & 0.4905 \\ 
0.1 & 0.0469 & \multicolumn{1}{|l}{0.0623} & \multicolumn{1}{|l}{0.0959} & 
\multicolumn{1}{|l}{0.1987} & \multicolumn{1}{|l}{0.9870} & 
\multicolumn{1}{|l}{0.8770} & \multicolumn{1}{|l}{0.8317} & 
\multicolumn{1}{|l|}{0.6072} \\ 
0.2 & 0.0464 & \multicolumn{1}{|l}{0.0554} & \multicolumn{1}{|l}{0.0800} & 
\multicolumn{1}{|l}{0.1526} & \multicolumn{1}{|l}{0.9862} & 
\multicolumn{1}{|l}{0.9010} & \multicolumn{1}{|l}{0.8687} & 
\multicolumn{1}{|l|}{0.6778} \\ 
0.3 & 0.0481 & \multicolumn{1}{|l}{0.0529} & \multicolumn{1}{|l}{0.0704} & 
\multicolumn{1}{|l}{0.1220} & \multicolumn{1}{|l}{0.9846} & 
\multicolumn{1}{|l}{0.9084} & \multicolumn{1}{|l}{0.8804} & 
\multicolumn{1}{|l|}{0.7169} \\ 
0.4 & 0.0483 & \multicolumn{1}{|l}{0.0518} & \multicolumn{1}{|l}{0.0649} & 
\multicolumn{1}{|l}{0.1036} & \multicolumn{1}{|l}{0.9808} & 
\multicolumn{1}{|l}{0.9059} & \multicolumn{1}{|l}{0.8809} & 
\multicolumn{1}{|l|}{0.7316} \\ 
0.5 & 0.0500 & \multicolumn{1}{|l}{0.0519} & \multicolumn{1}{|l}{0.0618} & 
\multicolumn{1}{|l}{0.0929} & \multicolumn{1}{|l}{0.9756} & 
\multicolumn{1}{|l}{0.9008} & \multicolumn{1}{|l}{0.8742} & 
\multicolumn{1}{|l|}{0.7338} \\ 
0.6 & 0.0500 & \multicolumn{1}{|l}{0.0501} & \multicolumn{1}{|l}{0.0577} & 
\multicolumn{1}{|l}{0.0858} & \multicolumn{1}{|l}{0.9689} & 
\multicolumn{1}{|l}{0.8914} & \multicolumn{1}{|l}{0.8662} & 
\multicolumn{1}{|l|}{0.7317} \\ 
0.7 & 0.0504 & \multicolumn{1}{|l}{0.0519} & \multicolumn{1}{|l}{0.0562} & 
\multicolumn{1}{|l}{0.0801} & \multicolumn{1}{|l}{0.9634} & 
\multicolumn{1}{|l}{0.8813} & \multicolumn{1}{|l}{0.8562} & 
\multicolumn{1}{|l|}{0.7258} \\ \hline
\multicolumn{9}{|l|}{{\small Table 2.} {\small Simulated sizes and powers
for different contamination proportions and different tuning }} \\ \hline
\multicolumn{9}{|l|}{{\small parameters }$\tau ,\beta =0,0.1,...,0.7${\small %
\ for the Rao-type tests }$R_{\boldsymbol{\tau }}\left( \theta _{0}\right) $%
{\small \ and }$S_{20}^{\beta }(\theta _{0})${\small \ for }$n=20.$} \\ 
\hline
\end{tabular}%
\end{equation*}

Regarding the power of the tests, uncontaminated scenarios there are values
at least so good than the corresponding to $\tau =0$ and $\beta =0$ and for
contaminated data the power corresponding to $\tau >0$ and $\beta >0$ are
higher.

The loss of efficiency caused by the Guassian assumption can be measured by
the discrepancy of the estimated levels and powers between the family of
Rao-type tests based on the restricted MDPDGE and the MDPDE. As expected,
empirical levels of the test statistics based on the MDPDGE are quite higher
than the corresponding levels of the test based in the MDPDE. However, the
test statistics based on the parametric model, $S_{n}^{\beta }(\theta _{0}),$
is quite conservative and so the corresponding powers are higher than those
of the proposed tests, $R_{\boldsymbol{\tau }}\left( \theta _{0}\right) .$
Based on the presented results, it seems that the proposed Rao-type tests, $%
R_{\boldsymbol{\tau }}\left( \theta _{0}\right) ,$ performs reasonably well
and offers an appealing alternative for situations where the probability
density function of the true model is unknown or it is very complicated to
work with it.

\section{Conclusions}

In this paper we have considered inferential techniques for situations where we do not know the parametric form of the density but the only available information is the mean vector and the variance-covariance matrix, expressed in terms of the a parameter vector $\boldsymbol{\theta}.$ To deal with this problem, Zhang (2019)
proposed a procedure based on the Gaussian distribution. However, the therein proposed estimator lacks of robustness and thus the procedure was
extended to robust estimators based on the DPD.
We have focused on the case in which additional constraints must be imposed to the estimated parameters, thus leading to the RMDPGE. We have derived the asymptotic distribution of the proposed estimator and we have studied its robustness properties in terms of the corresponding IF. Further, we have developed robust Rao-type test statistics under  null hypothesis, which requires the restricted version of the estimator.
Finally, a simulation study have been carried out to examine the performance of the proposed test statistics. From the results, we empirically showed that the Rao-type tests considered have a good performance in terms of efficiency and enjoys more robustness than the Zhang (2019) approach based on Gaussian estimators.

\section*{Acknowledgements}

This research is supported by the Spanish Grant:\ PID2021-124933NB-I00. The
authors are members of the Interdisciplinary Mathematics Institute (IMI).

\section{Appendix}

In the different Sections of the Appendix will be important the following
results:

\begin{enumerate}
\item Results in relation to the derivatives

\begin{enumerate}
\item $\dfrac{\partial \boldsymbol{\Sigma }\left( \boldsymbol{\theta }%
\right) }{\partial \theta _{i}}=\left \vert \boldsymbol{\Sigma }\left( 
\boldsymbol{\theta }\right) \right \vert $ $trace\left( \boldsymbol{\Sigma }%
\left( \boldsymbol{\theta }\right) ^{-1}\dfrac{\partial \boldsymbol{\Sigma }%
\left( \boldsymbol{\theta }\right) }{\partial \theta _{i}}\right) .$

\item $\frac{\partial trace\left( \boldsymbol{\Sigma }\left( \boldsymbol{%
\theta }\right) \right) }{\partial \theta _{i}}=trace\left( \dfrac{\partial 
\boldsymbol{\Sigma }\left( \boldsymbol{\theta }\right) }{\partial \theta _{i}%
}\right) .$

\item $\dfrac{\partial \boldsymbol{\Sigma }\left( \boldsymbol{\theta }%
\right) }{\partial \theta _{i}}^{-1}=-\boldsymbol{\Sigma }\left( \boldsymbol{%
\theta }\right) ^{-1}\dfrac{\partial \boldsymbol{\Sigma }\left( \boldsymbol{%
\theta }\right) }{\partial \theta _{i}}\boldsymbol{\Sigma }\left( 
\boldsymbol{\theta }\right) ^{-1}.$
\end{enumerate}

\item Let $\boldsymbol{Y}$ be a normal population with vector mean $%
\boldsymbol{\mu }$ and variance-covariance $\boldsymbol{\Sigma }$ we have,

\begin{enumerate}
\item $E\left[ \left( \boldsymbol{Y}-\boldsymbol{\mu }\right) ^{T}%
\boldsymbol{A}\left( \boldsymbol{Y}-\boldsymbol{\mu }\right) \right]
=Trace\left( \boldsymbol{A\Sigma }\right) .$

\item $E\left[ \left( \boldsymbol{Y}-\boldsymbol{\mu }\right) ^{T}%
\boldsymbol{A}\left( \boldsymbol{Y}-\boldsymbol{\mu }\right) \left( 
\boldsymbol{Y}-\boldsymbol{\mu }\right) ^{T}\boldsymbol{B}\left( \boldsymbol{%
Y}-\boldsymbol{\mu }\right) \right] =Trace\left( \boldsymbol{A\Sigma }\left( 
\boldsymbol{B+B}^{T}\right) \boldsymbol{\Sigma }\right) +Trace\left( 
\boldsymbol{A\Sigma }\right) Trace\left( \boldsymbol{B\Sigma }\right) .$

\item $E\left[ \left( \boldsymbol{Y}-\boldsymbol{\mu }\right) ^{T}%
\boldsymbol{A}\left( \boldsymbol{Y}-\boldsymbol{\mu }\right) \left( 
\boldsymbol{Y}-\boldsymbol{\mu }\right) \right] =\boldsymbol{0}.$
\end{enumerate}
\end{enumerate}

For more details about these results see for instance Harville (1997).

\subsection{Appendix A (Proof of Proposition \protect\ref{Proposition1})}

The expresion of $H_{n}^{\tau }(\boldsymbol{\theta }),$ introduced in (\ref%
{103}), is given by, 
\begin{equation*}
H_{n}^{\tau }(\boldsymbol{\theta })=a\left \vert \boldsymbol{\Sigma }\left( 
\boldsymbol{\theta }\right) \right \vert ^{-\tau /2}\left( \frac{1}{n}\tsum
\limits_{i=1}^{n}\exp \left \{ -\frac{\tau }{2}\left( \boldsymbol{y}_{i}-%
\boldsymbol{\mu }\left( \boldsymbol{\theta }\right) \right) ^{T}\boldsymbol{%
\Sigma }\left( \boldsymbol{\theta }\right) ^{-1}\left( \boldsymbol{y}_{i}-%
\boldsymbol{\mu }\left( \boldsymbol{\theta }\right) \right) \right \}
-b\right) -\frac{1}{\tau }
\end{equation*}%
and we consider the $d$-dimensional random vector $\boldsymbol{Y}_{\tau }(%
\boldsymbol{\theta })$ defined in (\ref{117a}). Applying Central Limit
Theorem we have, 
\begin{equation*}
\sqrt{n}\frac{\partial }{\partial \boldsymbol{\theta }}H_{n}^{\tau }(%
\boldsymbol{\theta })=\frac{1}{\sqrt{n}}\tsum \limits_{i=1}^{n}\boldsymbol{%
\Psi }_{\boldsymbol{\tau }}(\boldsymbol{y}_{i};\boldsymbol{\theta })\underset%
{n\longrightarrow \infty }{\overset{\mathcal{L}}{\longrightarrow }}\mathcal{N%
}(\boldsymbol{0}_{m},\boldsymbol{S}_{\boldsymbol{\tau }}(\boldsymbol{\theta }%
_{0})\boldsymbol{)}
\end{equation*}%
with 
\begin{equation*}
\boldsymbol{S}_{\boldsymbol{\tau }}(\boldsymbol{\theta }_{0})\boldsymbol{=}%
Cov\left[ \boldsymbol{Y}_{\tau }(\boldsymbol{\theta })\right] =E\left[ 
\boldsymbol{Y}_{\tau }(\boldsymbol{\theta })^{T}\boldsymbol{Y}_{\tau }(%
\boldsymbol{\theta })\right]
\end{equation*}%
because 
\begin{equation*}
E\left[ \boldsymbol{Y}_{\tau }(\boldsymbol{\theta })\right] =\boldsymbol{0}%
_{d}.
\end{equation*}%
To see that $E\left[ \boldsymbol{Y}_{\tau }(\boldsymbol{\theta })\right] =%
\boldsymbol{0}_{d}$ consider%
\begin{eqnarray*}
E\left[ \boldsymbol{Y}_{\tau }(\boldsymbol{\theta })\right] &=&a\text{ }%
\frac{\tau }{2}\left \vert \boldsymbol{\Sigma }\left( \boldsymbol{\theta }%
\right) \right \vert ^{-\tau /2}E\left[ -trace\left( \boldsymbol{\Sigma }%
\left( \boldsymbol{\theta }\right) ^{-1}\dfrac{\partial \boldsymbol{\Sigma }%
\left( \boldsymbol{\theta }\right) }{\partial \boldsymbol{\theta }}\right)
\right. \\
&&\left. \left. \exp \left \{ -\frac{\tau }{2}\left( \boldsymbol{Y}-%
\boldsymbol{\mu }\left( \boldsymbol{\theta }\right) \right) ^{T}\boldsymbol{%
\Sigma }\left( \boldsymbol{\theta }\right) ^{-1}\left( \boldsymbol{Y}-%
\boldsymbol{\mu }\left( \boldsymbol{\theta }\right) \right) \right \} +b%
\text{ }trace\left( \boldsymbol{\Sigma }\left( \boldsymbol{\theta }\right)
^{-1}\dfrac{\partial \boldsymbol{\Sigma }\left( \boldsymbol{\theta }\right) 
}{\partial \boldsymbol{\theta }}\right) \right. \right. \\
&&\left. \left. +\text{ }\exp \left \{ -\frac{\tau }{2}\left( \boldsymbol{Y}-%
\boldsymbol{\mu }\left( \boldsymbol{\theta }\right) \right) ^{T}\boldsymbol{%
\Sigma }\left( \boldsymbol{\theta }\right) ^{-1}\left( \boldsymbol{Y}-%
\boldsymbol{\mu }\left( \boldsymbol{\theta }\right) \right) \right \}
\right. \right. \\
&&\left. \left[ -2\left( \frac{\partial \boldsymbol{\mu }\left( \boldsymbol{%
\theta }\right) }{\partial \boldsymbol{\theta }}\right) ^{T}\boldsymbol{%
\Sigma }\left( \boldsymbol{\theta }\right) ^{-1}\left( \boldsymbol{Y}-%
\boldsymbol{\mu }\left( \boldsymbol{\theta }\right) \right) \right. \right.
\\
&&\left. \left. +\left( \boldsymbol{Y}-\boldsymbol{\mu }\left( \boldsymbol{%
\theta }\right) \right) ^{T}\left( \boldsymbol{\Sigma }\left( \boldsymbol{%
\theta }\right) ^{-1}\dfrac{\partial \boldsymbol{\Sigma }\left( \boldsymbol{%
\theta }\right) }{\partial \boldsymbol{\theta }}\boldsymbol{\Sigma }\left( 
\boldsymbol{\theta }\right) ^{-1}\right) \left( \boldsymbol{Y}-\boldsymbol{%
\mu }\left( \boldsymbol{\theta }\right) \right) \right] \right] \\
&=&a\text{ }\frac{\tau }{2}\left \vert \boldsymbol{\Sigma }\left( 
\boldsymbol{\theta }\right) \right \vert ^{-\tau /2}\left \{ -trace\left( 
\boldsymbol{\Sigma }\left( \boldsymbol{\theta }\right) ^{-1}\dfrac{\partial 
\boldsymbol{\Sigma }\left( \boldsymbol{\theta }\right) }{\partial 
\boldsymbol{\theta }}\right) \frac{1}{\left( \tau +1\right) ^{m/2}}\right. \\
&&\left. +\frac{\tau }{\left( \tau +1\right) ^{\frac{m}{2}+1}}trace\left( 
\boldsymbol{\Sigma }\left( \boldsymbol{\theta }\right) ^{-1}\dfrac{\partial 
\boldsymbol{\Sigma }\left( \boldsymbol{\theta }\right) }{\partial 
\boldsymbol{\theta }}\right) \right. \\
&&\left. +\frac{1}{\left( \tau +1\right) ^{\frac{m}{2}+1}}trace\left( 
\boldsymbol{\Sigma }\left( \boldsymbol{\theta }\right) ^{-1}\dfrac{\partial 
\boldsymbol{\Sigma }\left( \boldsymbol{\theta }\right) }{\partial 
\boldsymbol{\theta }}\right) \right \} \\
&=&\boldsymbol{0}_{d}.
\end{eqnarray*}%
We can observe that $\boldsymbol{Y}_{\tau }(\boldsymbol{\theta })$ is a $d$%
-dimensional vector whose j-th component is%
\begin{eqnarray*}
Y_{\tau }^{j}(\boldsymbol{\theta }) &=&a\text{ }\frac{\tau }{2}\left \vert 
\boldsymbol{\Sigma }\left( \boldsymbol{\theta }\right) \right \vert ^{-\tau
/2}\left \{ -trace\left( \boldsymbol{\Sigma }\left( \boldsymbol{\theta }%
\right) ^{-1}\dfrac{\partial \boldsymbol{\Sigma }\left( \boldsymbol{\theta }%
\right) }{\partial \theta _{j}}\right) \right. \\
&&\left. \exp \left \{ -\frac{\tau }{2}\left( \boldsymbol{y}-\boldsymbol{\mu 
}\left( \boldsymbol{\theta }\right) \right) ^{T}\boldsymbol{\Sigma }\left( 
\boldsymbol{\theta }\right) ^{-1}\left( \boldsymbol{y}-\boldsymbol{\mu }%
\left( \boldsymbol{\theta }\right) \right) \right \} +b\text{ }trace\left( 
\boldsymbol{\Sigma }\left( \boldsymbol{\theta }\right) ^{-1}\dfrac{\partial 
\boldsymbol{\Sigma }\left( \boldsymbol{\theta }\right) }{\partial \theta _{j}%
}\right) \right. \\
&&\left. +\text{ }\exp \left \{ -\frac{\tau }{2}\left( \boldsymbol{y}-%
\boldsymbol{\mu }\left( \boldsymbol{\theta }\right) \right) ^{T}\boldsymbol{%
\Sigma }\left( \boldsymbol{\theta }\right) ^{-1}\left( \boldsymbol{y}-%
\boldsymbol{\mu }\left( \boldsymbol{\theta }\right) \right) \right \} \right.
\\
&&\left[ -2\left( \frac{\partial \boldsymbol{\mu }\left( \boldsymbol{\theta }%
\right) }{\partial \theta _{j}}\right) ^{T}\boldsymbol{\Sigma }\left( 
\boldsymbol{\theta }\right) ^{-1}\left( \boldsymbol{y}-\boldsymbol{\mu }%
\left( \boldsymbol{\theta }\right) \right) \right. \\
&&\left. \left. +\left( \boldsymbol{y}-\boldsymbol{\mu }\left( \boldsymbol{%
\theta }\right) \right) ^{T}\left( \boldsymbol{\Sigma }\left( \boldsymbol{%
\theta }\right) ^{-1}\dfrac{\partial \boldsymbol{\Sigma }\left( \boldsymbol{%
\theta }\right) }{\partial \theta _{j}}\boldsymbol{\Sigma }\left( 
\boldsymbol{\theta }\right) ^{-1}\right) \left( \boldsymbol{y}-\boldsymbol{%
\mu }\left( \boldsymbol{\theta }\right) \right) \right] \right \} ,\text{ }%
j=1,...,d.
\end{eqnarray*}%
Therefore \ the element $(i,j)$ of the matrix $\boldsymbol{S}_{\boldsymbol{%
\tau }}(\boldsymbol{\theta }_{0})$ is given by 
\begin{equation*}
E\left[ Y_{\tau }^{i}(\boldsymbol{\theta })Y_{\tau }^{j}(\boldsymbol{\theta }%
)\right] .
\end{equation*}%
We are going to get $Y_{\tau }^{i}(\boldsymbol{\theta })Y_{\tau }^{j}(%
\boldsymbol{\theta }).$ We have,%
\begin{eqnarray*}
Y_{\tau }^{i}(\boldsymbol{\theta })Y_{\tau }^{j}(\boldsymbol{\theta })
&=&\left \{ a\text{ }\frac{\tau }{2}\left \vert \boldsymbol{\Sigma }\left( 
\boldsymbol{\theta }\right) \right \vert ^{-\tau /2}\left[ -trace\left( 
\boldsymbol{\Sigma }\left( \boldsymbol{\theta }\right) ^{-1}\dfrac{\partial 
\boldsymbol{\Sigma }\left( \boldsymbol{\theta }\right) }{\partial \theta _{i}%
}\right) \right. \right. \\
&&\left. \left. \exp \left \{ -\frac{\tau }{2}\left( \boldsymbol{y}_{i}-%
\boldsymbol{\mu }\left( \boldsymbol{\theta }\right) \right) ^{T}\boldsymbol{%
\Sigma }\left( \boldsymbol{\theta }\right) ^{-1}\left( \boldsymbol{y}_{i}-%
\boldsymbol{\mu }\left( \boldsymbol{\theta }\right) \right) \right \} +b%
\text{ }trace\left( \boldsymbol{\Sigma }\left( \boldsymbol{\theta }\right)
^{-1}\dfrac{\partial \boldsymbol{\Sigma }\left( \boldsymbol{\theta }\right) 
}{\partial \theta _{i}}\right) \right. \right. \\
&&\left. \left. +\text{ }\exp \left \{ -\frac{\tau }{2}\left( \boldsymbol{y}%
_{i}-\boldsymbol{\mu }\left( \boldsymbol{\theta }\right) \right) ^{T}%
\boldsymbol{\Sigma }\left( \boldsymbol{\theta }\right) ^{-1}\left( 
\boldsymbol{y}_{i}-\boldsymbol{\mu }\left( \boldsymbol{\theta }\right)
\right) \right \} \right. \right. \\
&&\left. \left[ 2\left( \frac{\partial \boldsymbol{\mu }\left( \boldsymbol{%
\theta }\right) }{\partial \theta _{i}}\right) ^{T}\boldsymbol{\Sigma }%
\left( \boldsymbol{\theta }\right) ^{-1}\left( \boldsymbol{y}_{i}-%
\boldsymbol{\mu }\left( \boldsymbol{\theta }\right) \right) \right. \right.
\\
&&\left. \left. \left. +\left( \boldsymbol{y}_{i}-\boldsymbol{\mu }\left( 
\boldsymbol{\theta }\right) \right) ^{T}\left( \boldsymbol{\Sigma }\left( 
\boldsymbol{\theta }\right) ^{-1}\dfrac{\partial \boldsymbol{\Sigma }\left( 
\boldsymbol{\theta }\right) }{\partial \theta _{j}}\boldsymbol{\Sigma }%
\left( \boldsymbol{\theta }\right) ^{-1}\right) \left( \boldsymbol{y}_{i}-%
\boldsymbol{\mu }\left( \boldsymbol{\theta }\right) \right) \right] \right]
\right \} \\
&&.\left \{ a\text{ }\frac{\tau }{2}\left \vert \boldsymbol{\Sigma }\left( 
\boldsymbol{\theta }\right) \right \vert ^{-\tau /2}\left[ -trace\left( 
\boldsymbol{\Sigma }\left( \boldsymbol{\theta }\right) ^{-1}\dfrac{\partial 
\boldsymbol{\Sigma }\left( \boldsymbol{\theta }\right) }{\partial \theta _{j}%
}\right) \right. \right. \\
&&\left. \left. \exp \left \{ -\frac{\tau }{2}\left( \boldsymbol{y}-%
\boldsymbol{\mu }\left( \boldsymbol{\theta }\right) \right) ^{T}\boldsymbol{%
\Sigma }\left( \boldsymbol{\theta }\right) ^{-1}\left( \boldsymbol{y}-%
\boldsymbol{\mu }\left( \boldsymbol{\theta }\right) \right) \right \} +b%
\text{ }trace\left( \boldsymbol{\Sigma }\left( \boldsymbol{\theta }\right)
^{-1}\dfrac{\partial \boldsymbol{\Sigma }\left( \boldsymbol{\theta }\right) 
}{\partial \theta _{j}}\right) \right. \right. \\
&&\left. \left. +\text{ }\exp \left \{ -\frac{\tau }{2}\left( \boldsymbol{y}-%
\boldsymbol{\mu }\left( \boldsymbol{\theta }\right) \right) ^{T}\boldsymbol{%
\Sigma }\left( \boldsymbol{\theta }\right) ^{-1}\left( \boldsymbol{y}-%
\boldsymbol{\mu }\left( \boldsymbol{\theta }\right) \right) \right \}
\right. \right. \\
&&\left. \left[ 2\left( \frac{\partial \boldsymbol{\mu }\left( \boldsymbol{%
\theta }\right) }{\partial \theta _{j}}\right) ^{T}\boldsymbol{\Sigma }%
\left( \boldsymbol{\theta }\right) ^{-1}\left( \boldsymbol{y}-\boldsymbol{%
\mu }\left( \boldsymbol{\theta }\right) \right) \right. \right. \\
&&\left. \left. \left. +\left( \boldsymbol{y}-\boldsymbol{\mu }\left( 
\boldsymbol{\theta }\right) \right) ^{T}\left( \boldsymbol{\Sigma }\left( 
\boldsymbol{\theta }\right) ^{-1}\dfrac{\partial \boldsymbol{\Sigma }\left( 
\boldsymbol{\theta }\right) }{\partial \theta _{j}}\boldsymbol{\Sigma }%
\left( \boldsymbol{\theta }\right) ^{-1}\right) \left( \boldsymbol{y}-%
\boldsymbol{\mu }\left( \boldsymbol{\theta }\right) \right) \right] \right]
\right \} .
\end{eqnarray*}%
Therefore, $Y_{\tau }^{i}(\boldsymbol{\theta })Y_{\tau }^{j}(\boldsymbol{%
\theta })$ is given by 
\begin{eqnarray*}
&&a^{2}\text{ }\left( \frac{\tau }{2}\right) ^{2}\left \vert \boldsymbol{%
\Sigma }\left( \boldsymbol{\theta }\right) \right \vert ^{-\tau }\left \{
trace\left( \boldsymbol{\Sigma }\left( \boldsymbol{\theta }\right) ^{-1}%
\dfrac{\partial \boldsymbol{\Sigma }\left( \boldsymbol{\theta }\right) }{%
\partial \theta _{i}}\right) trace\left( \boldsymbol{\Sigma }\left( 
\boldsymbol{\theta }\right) ^{-1}\dfrac{\partial \boldsymbol{\Sigma }\left( 
\boldsymbol{\theta }\right) }{\partial \theta _{j}}\right) \right. \\
&&\left. \exp \left \{ -\frac{2\tau }{2}\left( \boldsymbol{y}_{i}-%
\boldsymbol{\mu }\left( \boldsymbol{\theta }\right) \right) ^{T}\boldsymbol{%
\Sigma }\left( \boldsymbol{\theta }\right) ^{-1}\left( \boldsymbol{y}_{i}-%
\boldsymbol{\mu }\left( \boldsymbol{\theta }\right) \right) \right \} \right.
\\
&&\left. -trace\left( \boldsymbol{\Sigma }\left( \boldsymbol{\theta }\right)
^{-1}\dfrac{\partial \boldsymbol{\Sigma }\left( \boldsymbol{\theta }\right) 
}{\partial \theta _{i}}\right) trace\left( \boldsymbol{\Sigma }\left( 
\boldsymbol{\theta }\right) ^{-1}\dfrac{\partial \boldsymbol{\Sigma }\left( 
\boldsymbol{\theta }\right) }{\partial \theta _{j}}\right) b\right. \\
&&\left. \exp \left \{ -\frac{2\tau }{2}\left( \boldsymbol{y}_{i}-%
\boldsymbol{\mu }\left( \boldsymbol{\theta }\right) \right) ^{T}\boldsymbol{%
\Sigma }\left( \boldsymbol{\theta }\right) ^{-1}\left( \boldsymbol{y}_{i}-%
\boldsymbol{\mu }\left( \boldsymbol{\theta }\right) \right) \right \} \right.
\\
&&\left. -trace\left( \boldsymbol{\Sigma }\left( \boldsymbol{\theta }\right)
^{-1}\dfrac{\partial \boldsymbol{\Sigma }\left( \boldsymbol{\theta }\right) 
}{\partial \theta _{i}}\right) \exp \left \{ -\frac{2\tau }{2}\left( 
\boldsymbol{y}_{i}-\boldsymbol{\mu }\left( \boldsymbol{\theta }\right)
\right) ^{T}\boldsymbol{\Sigma }\left( \boldsymbol{\theta }\right)
^{-1}\left( \boldsymbol{y}_{i}-\boldsymbol{\mu }\left( \boldsymbol{\theta }%
\right) \right) \right \} \right. \\
&&\left. \left[ 2\left( \frac{\partial \boldsymbol{\mu }\left( \boldsymbol{%
\theta }\right) }{\partial \theta _{i}}\right) ^{T}\boldsymbol{\Sigma }%
\left( \boldsymbol{\theta }\right) ^{-1}\left( \boldsymbol{y}_{i}-%
\boldsymbol{\mu }\left( \boldsymbol{\theta }\right) \right) +\left( 
\boldsymbol{y}_{i}-\boldsymbol{\mu }\left( \boldsymbol{\theta }\right)
\right) ^{T}\left( \boldsymbol{\Sigma }\left( \boldsymbol{\theta }\right)
^{-1}\dfrac{\partial \boldsymbol{\Sigma }\left( \boldsymbol{\theta }\right) 
}{\partial \theta _{j}}\boldsymbol{\Sigma }\left( \boldsymbol{\theta }%
\right) ^{-1}\right) \left( \boldsymbol{y}_{i}-\boldsymbol{\mu }\left( 
\boldsymbol{\theta }\right) \right) \right] \right. \\
&&\left. -b\text{ }trace\left( \boldsymbol{\Sigma }\left( \boldsymbol{\theta 
}\right) ^{-1}\dfrac{\partial \boldsymbol{\Sigma }\left( \boldsymbol{\theta }%
\right) }{\partial \theta _{i}}\right) trace\left( \boldsymbol{\Sigma }%
\left( \boldsymbol{\theta }\right) ^{-1}\dfrac{\partial \boldsymbol{\Sigma }%
\left( \boldsymbol{\theta }\right) }{\partial \theta _{j}}\right) \exp \left
\{ -\frac{\tau }{2}\left( \boldsymbol{y}-\boldsymbol{\mu }\left( \boldsymbol{%
\theta }\right) \right) ^{T}\boldsymbol{\Sigma }\left( \boldsymbol{\theta }%
\right) ^{-1}\left( \boldsymbol{y}-\boldsymbol{\mu }\left( \boldsymbol{%
\theta }\right) \right) \right \} \right. \\
&&\left. +b^{2}trace\left( \boldsymbol{\Sigma }\left( \boldsymbol{\theta }%
\right) ^{-1}\dfrac{\partial \boldsymbol{\Sigma }\left( \boldsymbol{\theta }%
\right) }{\partial \theta _{i}}\right) trace\left( \boldsymbol{\Sigma }%
\left( \boldsymbol{\theta }\right) ^{-1}\dfrac{\partial \boldsymbol{\Sigma }%
\left( \boldsymbol{\theta }\right) }{\partial \theta _{j}}\right) \right. \\
&&\left. +b\text{ }trace\left( \boldsymbol{\Sigma }\left( \boldsymbol{\theta 
}\right) ^{-1}\dfrac{\partial \boldsymbol{\Sigma }\left( \boldsymbol{\theta }%
\right) }{\partial \theta _{i}}\right) \exp \left \{ -\frac{\tau }{2}\left( 
\boldsymbol{y}-\boldsymbol{\mu }\left( \boldsymbol{\theta }\right) \right)
^{T}\boldsymbol{\Sigma }\left( \boldsymbol{\theta }\right) ^{-1}\left( 
\boldsymbol{y}-\boldsymbol{\mu }\left( \boldsymbol{\theta }\right) \right)
\right \} \right. \\
&&\left. \left[ 2\left( \frac{\partial \boldsymbol{\mu }\left( \boldsymbol{%
\theta }\right) }{\partial \theta _{j}}\right) ^{T}\boldsymbol{\Sigma }%
\left( \boldsymbol{\theta }\right) ^{-1}\left( \boldsymbol{y}-\boldsymbol{%
\mu }\left( \boldsymbol{\theta }\right) \right) +\left( \boldsymbol{y}-%
\boldsymbol{\mu }\left( \boldsymbol{\theta }\right) \right) ^{T}\left( 
\boldsymbol{\Sigma }\left( \boldsymbol{\theta }\right) ^{-1}\dfrac{\partial 
\boldsymbol{\Sigma }\left( \boldsymbol{\theta }\right) }{\partial \theta _{j}%
}\boldsymbol{\Sigma }\left( \boldsymbol{\theta }\right) ^{-1}\right) \left( 
\boldsymbol{y}-\boldsymbol{\mu }\left( \boldsymbol{\theta }\right) \right) %
\right] \right. \\
&&\left. -trace\left( \boldsymbol{\Sigma }\left( \boldsymbol{\theta }\right)
^{-1}\dfrac{\partial \boldsymbol{\Sigma }\left( \boldsymbol{\theta }\right) 
}{\partial \theta _{i}}\right) \exp \left \{ -\frac{2\tau }{2}\left( 
\boldsymbol{y}_{i}-\boldsymbol{\mu }\left( \boldsymbol{\theta }\right)
\right) ^{T}\boldsymbol{\Sigma }\left( \boldsymbol{\theta }\right)
^{-1}\left( \boldsymbol{y}_{i}-\boldsymbol{\mu }\left( \boldsymbol{\theta }%
\right) \right) \right \} \right. \\
&&\left. \left[ 2\left( \frac{\partial \boldsymbol{\mu }\left( \boldsymbol{%
\theta }\right) }{\partial \theta _{j}}\right) ^{T}\boldsymbol{\Sigma }%
\left( \boldsymbol{\theta }\right) ^{-1}\left( \boldsymbol{y}-\boldsymbol{%
\mu }\left( \boldsymbol{\theta }\right) \right) +\left( \boldsymbol{y}-%
\boldsymbol{\mu }\left( \boldsymbol{\theta }\right) \right) ^{T}\left( 
\boldsymbol{\Sigma }\left( \boldsymbol{\theta }\right) ^{-1}\dfrac{\partial 
\boldsymbol{\Sigma }\left( \boldsymbol{\theta }\right) }{\partial \theta _{j}%
}\boldsymbol{\Sigma }\left( \boldsymbol{\theta }\right) ^{-1}\right) \left( 
\boldsymbol{y}-\boldsymbol{\mu }\left( \boldsymbol{\theta }\right) \right) %
\right] \right. \\
&&\left. +b\text{ }trace\left( \boldsymbol{\Sigma }\left( \boldsymbol{\theta 
}\right) ^{-1}\dfrac{\partial \boldsymbol{\Sigma }\left( \boldsymbol{\theta }%
\right) }{\partial \theta _{i}}\right) \exp \left \{ -\frac{2\tau }{2}\left( 
\boldsymbol{y}_{i}-\boldsymbol{\mu }\left( \boldsymbol{\theta }\right)
\right) ^{T}\boldsymbol{\Sigma }\left( \boldsymbol{\theta }\right)
^{-1}\left( \boldsymbol{y}_{i}-\boldsymbol{\mu }\left( \boldsymbol{\theta }%
\right) \right) \right \} \right. \\
&&\left. \left[ 2\left( \frac{\partial \boldsymbol{\mu }\left( \boldsymbol{%
\theta }\right) }{\partial \theta _{j}}\right) ^{T}\boldsymbol{\Sigma }%
\left( \boldsymbol{\theta }\right) ^{-1}\left( \boldsymbol{y}-\boldsymbol{%
\mu }\left( \boldsymbol{\theta }\right) \right) +\left( \boldsymbol{y}-%
\boldsymbol{\mu }\left( \boldsymbol{\theta }\right) \right) ^{T}\left( 
\boldsymbol{\Sigma }\left( \boldsymbol{\theta }\right) ^{-1}\dfrac{\partial 
\boldsymbol{\Sigma }\left( \boldsymbol{\theta }\right) }{\partial \theta _{j}%
}\boldsymbol{\Sigma }\left( \boldsymbol{\theta }\right) ^{-1}\right) \left( 
\boldsymbol{y}-\boldsymbol{\mu }\left( \boldsymbol{\theta }\right) \right) %
\right] \right. \\
&&\left. +\exp \left \{ -\frac{2\tau }{2}\left( \boldsymbol{y}_{i}-%
\boldsymbol{\mu }\left( \boldsymbol{\theta }\right) \right) ^{T}\boldsymbol{%
\Sigma }\left( \boldsymbol{\theta }\right) ^{-1}\left( \boldsymbol{y}_{i}-%
\boldsymbol{\mu }\left( \boldsymbol{\theta }\right) \right) \right \} \right.
\\
&&\left. \left[ 2\left( \frac{\partial \boldsymbol{\mu }\left( \boldsymbol{%
\theta }\right) }{\partial \theta _{i}}\right) ^{T}\boldsymbol{\Sigma }%
\left( \boldsymbol{\theta }\right) ^{-1}\left( \boldsymbol{y}_{i}-%
\boldsymbol{\mu }\left( \boldsymbol{\theta }\right) \right) +\left( 
\boldsymbol{y}_{i}-\boldsymbol{\mu }\left( \boldsymbol{\theta }\right)
\right) ^{T}\left( \boldsymbol{\Sigma }\left( \boldsymbol{\theta }\right)
^{-1}\dfrac{\partial \boldsymbol{\Sigma }\left( \boldsymbol{\theta }\right) 
}{\partial \theta _{i}}\boldsymbol{\Sigma }\left( \boldsymbol{\theta }%
\right) ^{-1}\right) \left( \boldsymbol{y}_{i}-\boldsymbol{\mu }\left( 
\boldsymbol{\theta }\right) \right) \right] \right. \\
&&\left. \left[ 2\left( \frac{\partial \boldsymbol{\mu }\left( \boldsymbol{%
\theta }\right) }{\partial \theta _{j}}\right) ^{T}\boldsymbol{\Sigma }%
\left( \boldsymbol{\theta }\right) ^{-1}\left( \boldsymbol{y}-\boldsymbol{%
\mu }\left( \boldsymbol{\theta }\right) \right) +\left( \boldsymbol{y}-%
\boldsymbol{\mu }\left( \boldsymbol{\theta }\right) \right) ^{T}\left( 
\boldsymbol{\Sigma }\left( \boldsymbol{\theta }\right) ^{-1}\dfrac{\partial 
\boldsymbol{\Sigma }\left( \boldsymbol{\theta }\right) }{\partial \theta _{j}%
}\boldsymbol{\Sigma }\left( \boldsymbol{\theta }\right) ^{-1}\right) \left( 
\boldsymbol{y}-\boldsymbol{\mu }\left( \boldsymbol{\theta }\right) \right) %
\right] \right \} .
\end{eqnarray*}%
We can write $Y_{\tau }^{i}(\boldsymbol{\theta })Y_{\tau }^{j}(\boldsymbol{%
\theta })$ by 
\begin{equation*}
Y_{\tau }^{i}(\boldsymbol{\theta })Y_{\tau }^{j}(\boldsymbol{\theta })=a^{2}%
\text{ }\left( \frac{\tau }{2}\right) ^{2}\left \vert \boldsymbol{\Sigma }%
\left( \boldsymbol{\theta }\right) \right \vert ^{-\tau }\left \{
C_{1}+C_{2}+C_{3}+C_{4}+C_{5}+C_{6}+C_{7}+C_{8}+C_{9}\right \}
\end{equation*}%
and%
\begin{eqnarray}
E\left[ Y_{i}(\boldsymbol{\theta })Y_{j}(\boldsymbol{\theta })\right]
&=&a^{2}\text{ }\left( \frac{\tau }{2}\right) ^{2}\left \vert \boldsymbol{%
\Sigma }\left( \boldsymbol{\theta }\right) \right \vert ^{-\tau }
\label{A100} \\
&&\times E\left[ C_{1}+C_{2}+C_{3}+C_{4}+C_{5}+C_{6}+C_{7}+C_{8}+C_{9}\right]
.
\end{eqnarray}%
Now we are going to calculate the different expectations appearing in (\ref%
{A100}).

\QTP{Dialog Math}
We have, 
\begin{equation*}
C_{1}=trace\left( \boldsymbol{\Sigma }\left( \boldsymbol{\theta }\right)
^{-1}\dfrac{\partial \boldsymbol{\Sigma }\left( \boldsymbol{\theta }\right) 
}{\partial \theta _{i}}\right) trace\left( \boldsymbol{\Sigma }\left( 
\boldsymbol{\theta }\right) ^{-1}\dfrac{\partial \boldsymbol{\Sigma }\left( 
\boldsymbol{\theta }\right) }{\partial \theta _{j}}\right) \exp \left \{ -%
\frac{2\tau }{2}\left( \boldsymbol{y}_{i}-\boldsymbol{\mu }\left( 
\boldsymbol{\theta }\right) \right) ^{T}\boldsymbol{\Sigma }\left( 
\boldsymbol{\theta }\right) ^{-1}\left( \boldsymbol{y}_{i}-\boldsymbol{\mu }%
\left( \boldsymbol{\theta }\right) \right) \right \} .
\end{equation*}%
Therefore, 
\begin{eqnarray*}
E\left[ C_{1}\right] &=&trace\left( \boldsymbol{\Sigma }\left( \boldsymbol{%
\theta }\right) ^{-1}\dfrac{\partial \boldsymbol{\Sigma }\left( \boldsymbol{%
\theta }\right) }{\partial \theta _{i}}\right) trace\left( \boldsymbol{%
\Sigma }\left( \boldsymbol{\theta }\right) ^{-1}\dfrac{\partial \boldsymbol{%
\Sigma }\left( \boldsymbol{\theta }\right) }{\partial \theta _{j}}\right) \\
&&\dint \frac{1}{\left( 2\pi \right) ^{m/2}\left \vert \boldsymbol{\Sigma }%
\left( \boldsymbol{\theta }\right) \right \vert ^{1/2}}\exp \left \{ -\frac{1%
}{2}\left( \boldsymbol{y}-\boldsymbol{\mu }\left( \boldsymbol{\theta }%
\right) \right) ^{T}\boldsymbol{\Sigma }\left( \boldsymbol{\theta }\right)
^{-1}\left( \boldsymbol{y}-\boldsymbol{\mu }\left( \boldsymbol{\theta }%
\right) \right) \right \} \\
&&\exp \left \{ -\frac{1}{2}\left( \boldsymbol{y}-\boldsymbol{\mu }\left( 
\boldsymbol{\theta }\right) \right) ^{T}\left( \frac{\boldsymbol{\Sigma }%
\left( \boldsymbol{\theta }\right) }{2\tau }\right) ^{-1}\left( \boldsymbol{y%
}-\boldsymbol{\mu }\left( \boldsymbol{\theta }\right) \right) \right \} d%
\boldsymbol{y} \\
&=&trace\left( \boldsymbol{\Sigma }\left( \boldsymbol{\theta }\right) ^{-1}%
\dfrac{\partial \boldsymbol{\Sigma }\left( \boldsymbol{\theta }\right) }{%
\partial \theta _{i}}\right) trace\left( \boldsymbol{\Sigma }\left( 
\boldsymbol{\theta }\right) ^{-1}\dfrac{\partial \boldsymbol{\Sigma }\left( 
\boldsymbol{\theta }\right) }{\partial \theta _{j}}\right) \\
&&\left( \frac{1}{2\tau +1}\right) ^{\frac{m}{2}}\dint \frac{1}{\left( 2\pi
\right) ^{m/2}\left \vert \frac{\boldsymbol{\Sigma }\left( \boldsymbol{%
\theta }\right) }{2\tau +1}\right \vert ^{1/2}}\exp \left \{ -\frac{1}{2}%
\left( \boldsymbol{y}-\boldsymbol{\mu }\left( \boldsymbol{\theta }\right)
\right) ^{T}\left( \frac{\boldsymbol{\Sigma }\left( \boldsymbol{\theta }%
\right) }{2\tau +1}\right) ^{-1}\left( \boldsymbol{y}-\boldsymbol{\mu }%
\left( \boldsymbol{\theta }\right) \right) \right \} d\boldsymbol{y} \\
&=&trace\left( \boldsymbol{\Sigma }\left( \boldsymbol{\theta }\right) ^{-1}%
\dfrac{\partial \boldsymbol{\Sigma }\left( \boldsymbol{\theta }\right) }{%
\partial \theta _{i}}\right) trace\left( \boldsymbol{\Sigma }\left( 
\boldsymbol{\theta }\right) ^{-1}\dfrac{\partial \boldsymbol{\Sigma }\left( 
\boldsymbol{\theta }\right) }{\partial \theta _{j}}\right) \left( \frac{1}{%
2\tau +1}\right) ^{\frac{m}{2}}
\end{eqnarray*}%
The expression of $C_{2}$ is given by 
\begin{eqnarray*}
C_{2} &=&-trace\left( \boldsymbol{\Sigma }\left( \boldsymbol{\theta }\right)
^{-1}\dfrac{\partial \boldsymbol{\Sigma }\left( \boldsymbol{\theta }\right) 
}{\partial \theta _{i}}\right) trace\left( \boldsymbol{\Sigma }\left( 
\boldsymbol{\theta }\right) ^{-1}\dfrac{\partial \boldsymbol{\Sigma }\left( 
\boldsymbol{\theta }\right) }{\partial \theta _{j}}\right) \text{ }b \\
&&\exp \left \{ -\frac{2\tau }{2}\left( \boldsymbol{y}_{i}-\boldsymbol{\mu }%
\left( \boldsymbol{\theta }\right) \right) ^{T}\boldsymbol{\Sigma }\left( 
\boldsymbol{\theta }\right) ^{-1}\left( \boldsymbol{y}_{i}-\boldsymbol{\mu }%
\left( \boldsymbol{\theta }\right) \right) \right \}
\end{eqnarray*}%
and 
\begin{eqnarray*}
E\left[ C_{2}\right] &=&-\frac{\tau }{\left( 1+\tau \right) ^{\frac{m}{2}+1}}%
trace\left( \boldsymbol{\Sigma }\left( \boldsymbol{\theta }\right) ^{-1}%
\dfrac{\partial \boldsymbol{\Sigma }\left( \boldsymbol{\theta }\right) }{%
\partial \theta _{i}}\right) trace\left( \boldsymbol{\Sigma }\left( 
\boldsymbol{\theta }\right) ^{-1}\dfrac{\partial \boldsymbol{\Sigma }\left( 
\boldsymbol{\theta }\right) }{\partial \theta _{j}}\right) \\
&&\dint \frac{1}{\left( 2\pi \right) ^{m/2}\left \vert \boldsymbol{\Sigma }%
\left( \boldsymbol{\theta }\right) \right \vert ^{1/2}}\exp \left \{ -\frac{1%
}{2}\left( \boldsymbol{y}-\boldsymbol{\mu }\left( \boldsymbol{\theta }%
\right) \right) ^{T}\boldsymbol{\Sigma }\left( \boldsymbol{\theta }\right)
^{-1}\left( \boldsymbol{y}-\boldsymbol{\mu }\left( \boldsymbol{\theta }%
\right) \right) \right \} \\
&&\exp \left \{ -\frac{1}{2}\left( \boldsymbol{y}-\boldsymbol{\mu }\left( 
\boldsymbol{\theta }\right) \right) ^{T}\left( \frac{\boldsymbol{\Sigma }%
\left( \boldsymbol{\theta }\right) }{\tau }\right) ^{-1}\left( \boldsymbol{y}%
-\boldsymbol{\mu }\left( \boldsymbol{\theta }\right) \right) \right \} d%
\boldsymbol{y} \\
&=&-\frac{\tau }{\left( 1+\tau \right) ^{\frac{m}{2}+1}}trace\left( 
\boldsymbol{\Sigma }\left( \boldsymbol{\theta }\right) ^{-1}\dfrac{\partial 
\boldsymbol{\Sigma }\left( \boldsymbol{\theta }\right) }{\partial \theta _{i}%
}\right) trace\left( \boldsymbol{\Sigma }\left( \boldsymbol{\theta }\right)
^{-1}\dfrac{\partial \boldsymbol{\Sigma }\left( \boldsymbol{\theta }\right) 
}{\partial \theta _{j}}\right) \\
&&\frac{1}{\left( 1+\tau \right) ^{\frac{m}{2}}}\dint \frac{1}{\left( 2\pi
\right) ^{m/2}\left \vert \frac{\boldsymbol{\Sigma }\left( \boldsymbol{%
\theta }\right) }{\tau +1}\right \vert ^{1/2}}\exp \left \{ -\frac{1}{2}%
\left( \boldsymbol{y}-\boldsymbol{\mu }\left( \boldsymbol{\theta }\right)
\right) ^{T}\left( \frac{\boldsymbol{\Sigma }\left( \boldsymbol{\theta }%
\right) }{\tau +1}\right) ^{-1}\left( \boldsymbol{y}-\boldsymbol{\mu }\left( 
\boldsymbol{\theta }\right) \right) \right \} d\boldsymbol{y} \\
&=&-\frac{\tau }{\left( 1+\tau \right) ^{m+1}}trace\left( \boldsymbol{\Sigma 
}\left( \boldsymbol{\theta }\right) ^{-1}\dfrac{\partial \boldsymbol{\Sigma }%
\left( \boldsymbol{\theta }\right) }{\partial \theta _{i}}\right)
trace\left( \boldsymbol{\Sigma }\left( \boldsymbol{\theta }\right) ^{-1}%
\dfrac{\partial \boldsymbol{\Sigma }\left( \boldsymbol{\theta }\right) }{%
\partial \theta _{j}}\right) .
\end{eqnarray*}%
The expression of $C_{3}$ is given by 
\begin{eqnarray*}
C_{3} &=&-trace\left( \boldsymbol{\Sigma }\left( \boldsymbol{\theta }\right)
^{-1}\dfrac{\partial \boldsymbol{\Sigma }\left( \boldsymbol{\theta }\right) 
}{\partial \theta _{i}}\right) \exp \left \{ -\frac{2\tau }{2}\left( 
\boldsymbol{y}_{i}-\boldsymbol{\mu }\left( \boldsymbol{\theta }\right)
\right) ^{T}\boldsymbol{\Sigma }\left( \boldsymbol{\theta }\right)
^{-1}\left( \boldsymbol{y}_{i}-\boldsymbol{\mu }\left( \boldsymbol{\theta }%
\right) \right) \right \} \\
&&\left[ 2\left( \frac{\partial \boldsymbol{\mu }\left( \boldsymbol{\theta }%
\right) }{\partial \theta _{i}}\right) ^{T}\boldsymbol{\Sigma }\left( 
\boldsymbol{\theta }\right) ^{-1}\left( \boldsymbol{y}_{i}-\boldsymbol{\mu }%
\left( \boldsymbol{\theta }\right) \right) +\left( \boldsymbol{y}_{i}-%
\boldsymbol{\mu }\left( \boldsymbol{\theta }\right) \right) ^{T}\left( 
\boldsymbol{\Sigma }\left( \boldsymbol{\theta }\right) ^{-1}\dfrac{\partial 
\boldsymbol{\Sigma }\left( \boldsymbol{\theta }\right) }{\partial \theta _{j}%
}\boldsymbol{\Sigma }\left( \boldsymbol{\theta }\right) ^{-1}\right) \left( 
\boldsymbol{y}_{i}-\boldsymbol{\mu }\left( \boldsymbol{\theta }\right)
\right) \right] .
\end{eqnarray*}%
Then%
\begin{eqnarray*}
E\left[ C_{3}\right] &=&-trace\left( \boldsymbol{\Sigma }\left( \boldsymbol{%
\theta }\right) ^{-1}\dfrac{\partial \boldsymbol{\Sigma }\left( \boldsymbol{%
\theta }\right) }{\partial \theta _{i}}\right) \frac{1}{\left( 1+2\tau
\right) ^{m/2}} \\
&&\dint \frac{1}{\left( 2\pi \right) ^{m/2}\left \vert \frac{\boldsymbol{%
\Sigma }\left( \boldsymbol{\theta }\right) }{2\tau +1}\right \vert ^{1/2}}%
\exp \left \{ -\frac{1}{2}\left( \boldsymbol{y}-\boldsymbol{\mu }\left( 
\boldsymbol{\theta }\right) \right) ^{T}\boldsymbol{\Sigma }\left( 
\boldsymbol{\theta }\right) ^{-1}\left( \boldsymbol{y}-\boldsymbol{\mu }%
\left( \boldsymbol{\theta }\right) \right) \right \} \\
&&\left( \boldsymbol{y}-\boldsymbol{\mu }\left( \boldsymbol{\theta }\right)
\right) ^{T}\left( \boldsymbol{\Sigma }\left( \boldsymbol{\theta }\right)
^{-1}\dfrac{\partial \boldsymbol{\Sigma }\left( \boldsymbol{\theta }\right) 
}{\partial \theta _{j}}\boldsymbol{\Sigma }\left( \boldsymbol{\theta }%
\right) ^{-1}\right) \left( \boldsymbol{y}-\boldsymbol{\mu }\left( 
\boldsymbol{\theta }\right) \right) d\boldsymbol{y} \\
&=&-trace\left( \boldsymbol{\Sigma }\left( \boldsymbol{\theta }\right) ^{-1}%
\dfrac{\partial \boldsymbol{\Sigma }\left( \boldsymbol{\theta }\right) }{%
\partial \theta _{i}}\right) trace\left( \boldsymbol{\Sigma }\left( 
\boldsymbol{\theta }\right) ^{-1}\dfrac{\partial \boldsymbol{\Sigma }\left( 
\boldsymbol{\theta }\right) }{\partial \theta _{j}}\right) \frac{1}{\left(
1+2\tau \right) ^{\frac{m}{2}+1}}.
\end{eqnarray*}%
In relation to $C_{4}$ we have, 
\begin{equation*}
C_{4}=-b\text{ }trace\left( \boldsymbol{\Sigma }\left( \boldsymbol{\theta }%
\right) ^{-1}\dfrac{\partial \boldsymbol{\Sigma }\left( \boldsymbol{\theta }%
\right) }{\partial \theta _{i}}\right) trace\left( \boldsymbol{\Sigma }%
\left( \boldsymbol{\theta }\right) ^{-1}\dfrac{\partial \boldsymbol{\Sigma }%
\left( \boldsymbol{\theta }\right) }{\partial \theta _{j}}\right) \exp \left
\{ -\frac{\tau }{2}\left( \boldsymbol{y}-\boldsymbol{\mu }\left( \boldsymbol{%
\theta }\right) \right) ^{T}\boldsymbol{\Sigma }\left( \boldsymbol{\theta }%
\right) ^{-1}\left( \boldsymbol{y}-\boldsymbol{\mu }\left( \boldsymbol{%
\theta }\right) \right) \right \}
\end{equation*}%
and 
\begin{eqnarray*}
E\left[ C_{4}\right] &=&-\frac{\tau }{\left( 1+\tau \right) ^{\frac{m}{2}+1}}%
trace\left( \boldsymbol{\Sigma }\left( \boldsymbol{\theta }\right) ^{-1}%
\dfrac{\partial \boldsymbol{\Sigma }\left( \boldsymbol{\theta }\right) }{%
\partial \theta _{i}}\right) trace\left( \boldsymbol{\Sigma }\left( 
\boldsymbol{\theta }\right) ^{-1}\dfrac{\partial \boldsymbol{\Sigma }\left( 
\boldsymbol{\theta }\right) }{\partial \theta _{j}}\right) \\
&&\dint \frac{1}{\left( 2\pi \right) ^{m/2}\left \vert \boldsymbol{\Sigma }%
\left( \boldsymbol{\theta }\right) \right \vert ^{1/2}}\exp \left \{ -\frac{1%
}{2}\left( \boldsymbol{y}-\boldsymbol{\mu }\left( \boldsymbol{\theta }%
\right) \right) ^{T}\boldsymbol{\Sigma }\left( \boldsymbol{\theta }\right)
^{-1}\left( \boldsymbol{y}-\boldsymbol{\mu }\left( \boldsymbol{\theta }%
\right) \right) \right \} \\
&&\exp \left \{ -\frac{\tau }{2}\left( \boldsymbol{y}-\boldsymbol{\mu }%
\left( \boldsymbol{\theta }\right) \right) ^{T}\left( \boldsymbol{\Sigma }%
\left( \boldsymbol{\theta }\right) \right) ^{-1}\left( \boldsymbol{y}-%
\boldsymbol{\mu }\left( \boldsymbol{\theta }\right) \right) \right \} d%
\boldsymbol{y} \\
&=&-\frac{\tau }{\left( 1+\tau \right) ^{\frac{m}{2}+1}}trace\left( 
\boldsymbol{\Sigma }\left( \boldsymbol{\theta }\right) ^{-1}\dfrac{\partial 
\boldsymbol{\Sigma }\left( \boldsymbol{\theta }\right) }{\partial \theta _{i}%
}\right) trace\left( \boldsymbol{\Sigma }\left( \boldsymbol{\theta }\right)
^{-1}\dfrac{\partial \boldsymbol{\Sigma }\left( \boldsymbol{\theta }\right) 
}{\partial \theta _{j}}\right) \\
&&\frac{1}{\left( 1+\tau \right) ^{\frac{m}{2}}}\dint \frac{1}{\left( 2\pi
\right) ^{m/2}\left \vert \frac{\boldsymbol{\Sigma }\left( \boldsymbol{%
\theta }\right) }{\tau +1}\right \vert ^{1/2}}\exp \left \{ -\frac{1}{2}%
\left( \boldsymbol{y}-\boldsymbol{\mu }\left( \boldsymbol{\theta }\right)
\right) ^{T}\left( \frac{\boldsymbol{\Sigma }\left( \boldsymbol{\theta }%
\right) }{\tau +1}\right) ^{-1}\left( \boldsymbol{y}-\boldsymbol{\mu }\left( 
\boldsymbol{\theta }\right) \right) \right \} d\boldsymbol{y} \\
&=&-\frac{\tau }{\left( 1+\tau \right) ^{m+1}}trace\left( \boldsymbol{\Sigma 
}\left( \boldsymbol{\theta }\right) ^{-1}\dfrac{\partial \boldsymbol{\Sigma }%
\left( \boldsymbol{\theta }\right) }{\partial \theta _{i}}\right)
trace\left( \boldsymbol{\Sigma }\left( \boldsymbol{\theta }\right) ^{-1}%
\dfrac{\partial \boldsymbol{\Sigma }\left( \boldsymbol{\theta }\right) }{%
\partial \theta _{j}}\right) .
\end{eqnarray*}%
In relation with $C_{5}$ we have,%
\begin{equation*}
C_{5}=b^{2}\text{ }trace\left( \boldsymbol{\Sigma }\left( \boldsymbol{\theta 
}\right) ^{-1}\dfrac{\partial \boldsymbol{\Sigma }\left( \boldsymbol{\theta }%
\right) }{\partial \theta _{i}}\right) trace\left( \boldsymbol{\Sigma }%
\left( \boldsymbol{\theta }\right) ^{-1}\dfrac{\partial \boldsymbol{\Sigma }%
\left( \boldsymbol{\theta }\right) }{\partial \theta _{j}}\right)
\end{equation*}%
and 
\begin{equation*}
E\left[ C_{5}\right] =\frac{\tau ^{2}}{\left( 1+\tau \right) ^{m+2}}%
trace\left( \boldsymbol{\Sigma }\left( \boldsymbol{\theta }\right) ^{-1}%
\dfrac{\partial \boldsymbol{\Sigma }\left( \boldsymbol{\theta }\right) }{%
\partial \theta _{i}}\right) trace\left( \boldsymbol{\Sigma }\left( 
\boldsymbol{\theta }\right) ^{-1}\dfrac{\partial \boldsymbol{\Sigma }\left( 
\boldsymbol{\theta }\right) }{\partial \theta _{j}}\right) .
\end{equation*}%
The expression of $C_{6}$ is given by, 
\begin{eqnarray*}
C_{6} &=&b\text{ }trace\left( \boldsymbol{\Sigma }\left( \boldsymbol{\theta }%
\right) ^{-1}\dfrac{\partial \boldsymbol{\Sigma }\left( \boldsymbol{\theta }%
\right) }{\partial \theta _{i}}\right) \exp \left \{ -\frac{\tau }{2}\left( 
\boldsymbol{y}-\boldsymbol{\mu }\left( \boldsymbol{\theta }\right) \right)
^{T}\boldsymbol{\Sigma }\left( \boldsymbol{\theta }\right) ^{-1}\left( 
\boldsymbol{y}-\boldsymbol{\mu }\left( \boldsymbol{\theta }\right) \right)
\right \} \\
&&\left[ 2\left( \frac{\partial \boldsymbol{\mu }\left( \boldsymbol{\theta }%
\right) }{\partial \theta _{j}}\right) ^{T}\boldsymbol{\Sigma }\left( 
\boldsymbol{\theta }\right) ^{-1}\left( \boldsymbol{y}-\boldsymbol{\mu }%
\left( \boldsymbol{\theta }\right) \right) +\left( \boldsymbol{y}-%
\boldsymbol{\mu }\left( \boldsymbol{\theta }\right) \right) ^{T}\left( 
\boldsymbol{\Sigma }\left( \boldsymbol{\theta }\right) ^{-1}\dfrac{\partial 
\boldsymbol{\Sigma }\left( \boldsymbol{\theta }\right) }{\partial \theta _{j}%
}\boldsymbol{\Sigma }\left( \boldsymbol{\theta }\right) ^{-1}\right) \left( 
\boldsymbol{y}-\boldsymbol{\mu }\left( \boldsymbol{\theta }\right) \right) %
\right]
\end{eqnarray*}%
and 
\begin{eqnarray*}
E\left[ C_{6}\right] &=&\frac{\tau }{\left( 1+\tau \right) ^{\frac{m}{2}+1}}%
trace\left( \boldsymbol{\Sigma }\left( \boldsymbol{\theta }\right) ^{-1}%
\dfrac{\partial \boldsymbol{\Sigma }\left( \boldsymbol{\theta }\right) }{%
\partial \theta _{i}}\right) trace\left( \boldsymbol{\Sigma }\left( 
\boldsymbol{\theta }\right) ^{-1}\dfrac{\partial \boldsymbol{\Sigma }\left( 
\boldsymbol{\theta }\right) }{\partial \theta _{j}}\right) \frac{1}{\left(
1+\tau \right) ^{\frac{m}{2}+1}} \\
&=&\frac{\tau }{\left( 1+\tau \right) ^{m+2}}trace\left( \boldsymbol{\Sigma }%
\left( \boldsymbol{\theta }\right) ^{-1}\dfrac{\partial \boldsymbol{\Sigma }%
\left( \boldsymbol{\theta }\right) }{\partial \theta _{i}}\right)
trace\left( \boldsymbol{\Sigma }\left( \boldsymbol{\theta }\right) ^{-1}%
\dfrac{\partial \boldsymbol{\Sigma }\left( \boldsymbol{\theta }\right) }{%
\partial \theta _{j}}\right) .
\end{eqnarray*}%
The expression of $C_{7}$ is 
\begin{eqnarray*}
C_{7} &=&-trace\left( \boldsymbol{\Sigma }\left( \boldsymbol{\theta }\right)
^{-1}\dfrac{\partial \boldsymbol{\Sigma }\left( \boldsymbol{\theta }\right) 
}{\partial \theta _{i}}\right) \exp \left \{ -\frac{2\tau }{2}\left( 
\boldsymbol{y}-\boldsymbol{\mu }\left( \boldsymbol{\theta }\right) \right)
^{T}\boldsymbol{\Sigma }\left( \boldsymbol{\theta }\right) ^{-1}\left( 
\boldsymbol{y}-\boldsymbol{\mu }\left( \boldsymbol{\theta }\right) \right)
\right \} \\
&&\left[ 2\left( \frac{\partial \boldsymbol{\mu }\left( \boldsymbol{\theta }%
\right) }{\partial \theta _{j}}\right) ^{T}\boldsymbol{\Sigma }\left( 
\boldsymbol{\theta }\right) ^{-1}\left( \boldsymbol{y}-\boldsymbol{\mu }%
\left( \boldsymbol{\theta }\right) \right) +\left( \boldsymbol{y}-%
\boldsymbol{\mu }\left( \boldsymbol{\theta }\right) \right) ^{T}\left( 
\boldsymbol{\Sigma }\left( \boldsymbol{\theta }\right) ^{-1}\dfrac{\partial 
\boldsymbol{\Sigma }\left( \boldsymbol{\theta }\right) }{\partial \theta _{j}%
}\boldsymbol{\Sigma }\left( \boldsymbol{\theta }\right) ^{-1}\right) \left( 
\boldsymbol{y}-\boldsymbol{\mu }\left( \boldsymbol{\theta }\right) \right) %
\right]
\end{eqnarray*}%
and 
\begin{equation*}
E\left[ C_{7}\right] =-\frac{\tau }{\left( 1+2\tau \right) ^{\frac{m}{2}+1}}%
trace\left( \boldsymbol{\Sigma }\left( \boldsymbol{\theta }\right) ^{-1}%
\dfrac{\partial \boldsymbol{\Sigma }\left( \boldsymbol{\theta }\right) }{%
\partial \theta _{i}}\right) trace\left( \boldsymbol{\Sigma }\left( 
\boldsymbol{\theta }\right) ^{-1}\dfrac{\partial \boldsymbol{\Sigma }\left( 
\boldsymbol{\theta }\right) }{\partial \theta _{j}}\right) .
\end{equation*}%
In relation to $C_{8},$%
\begin{eqnarray*}
C_{8} &=&b\text{ }trace\left( \boldsymbol{\Sigma }\left( \boldsymbol{\theta }%
\right) ^{-1}\dfrac{\partial \boldsymbol{\Sigma }\left( \boldsymbol{\theta }%
\right) }{\partial \theta _{i}}\right) \exp \left \{ -\frac{\tau }{2}\left( 
\boldsymbol{y}-\boldsymbol{\mu }\left( \boldsymbol{\theta }\right) \right)
^{T}\boldsymbol{\Sigma }\left( \boldsymbol{\theta }\right) ^{-1}\left( 
\boldsymbol{y}-\boldsymbol{\mu }\left( \boldsymbol{\theta }\right) \right)
\right \} \\
&&\left[ 2\left( \frac{\partial \boldsymbol{\mu }\left( \boldsymbol{\theta }%
\right) }{\partial \theta _{j}}\right) ^{T}\boldsymbol{\Sigma }\left( 
\boldsymbol{\theta }\right) ^{-1}\left( \boldsymbol{y}-\boldsymbol{\mu }%
\left( \boldsymbol{\theta }\right) \right) +\left( \boldsymbol{y}-%
\boldsymbol{\mu }\left( \boldsymbol{\theta }\right) \right) ^{T}\left( 
\boldsymbol{\Sigma }\left( \boldsymbol{\theta }\right) ^{-1}\dfrac{\partial 
\boldsymbol{\Sigma }\left( \boldsymbol{\theta }\right) }{\partial \theta _{j}%
}\boldsymbol{\Sigma }\left( \boldsymbol{\theta }\right) ^{-1}\right) \left( 
\boldsymbol{y}-\boldsymbol{\mu }\left( \boldsymbol{\theta }\right) \right) %
\right]
\end{eqnarray*}%
and 
\begin{eqnarray*}
E\left[ C_{8}\right] &=&\frac{\tau }{\left( 1+\tau \right) ^{\frac{m}{2}+1}}%
trace\left( \boldsymbol{\Sigma }\left( \boldsymbol{\theta }\right) ^{-1}%
\dfrac{\partial \boldsymbol{\Sigma }\left( \boldsymbol{\theta }\right) }{%
\partial \theta _{i}}\right) trace\left( \boldsymbol{\Sigma }\left( 
\boldsymbol{\theta }\right) ^{-1}\dfrac{\partial \boldsymbol{\Sigma }\left( 
\boldsymbol{\theta }\right) }{\partial \theta _{j}}\right) \frac{1}{\left(
1+\tau \right) ^{\frac{m}{2}+1}} \\
&=&\frac{\tau }{\left( 1+\tau \right) ^{m+2}}trace\left( \boldsymbol{\Sigma }%
\left( \boldsymbol{\theta }\right) ^{-1}\dfrac{\partial \boldsymbol{\Sigma }%
\left( \boldsymbol{\theta }\right) }{\partial \theta _{i}}\right)
trace\left( \boldsymbol{\Sigma }\left( \boldsymbol{\theta }\right) ^{-1}%
\dfrac{\partial \boldsymbol{\Sigma }\left( \boldsymbol{\theta }\right) }{%
\partial \theta _{j}}\right) .
\end{eqnarray*}%
Finally, 
\begin{eqnarray*}
C_{9} &=&\exp \left \{ -\frac{2\tau }{2}\left( \boldsymbol{y}_{i}-%
\boldsymbol{\mu }\left( \boldsymbol{\theta }\right) \right) ^{T}\boldsymbol{%
\Sigma }\left( \boldsymbol{\theta }\right) ^{-1}\left( \boldsymbol{y}_{i}-%
\boldsymbol{\mu }\left( \boldsymbol{\theta }\right) \right) \right \} \\
&&\left[ 2\left( \frac{\partial \boldsymbol{\mu }\left( \boldsymbol{\theta }%
\right) }{\partial \theta _{i}}\right) ^{T}\boldsymbol{\Sigma }\left( 
\boldsymbol{\theta }\right) ^{-1}\left( \boldsymbol{y}_{i}-\boldsymbol{\mu }%
\left( \boldsymbol{\theta }\right) \right) +\left( \boldsymbol{y}_{i}-%
\boldsymbol{\mu }\left( \boldsymbol{\theta }\right) \right) ^{T}\left( 
\boldsymbol{\Sigma }\left( \boldsymbol{\theta }\right) ^{-1}\dfrac{\partial 
\boldsymbol{\Sigma }\left( \boldsymbol{\theta }\right) }{\partial \theta _{i}%
}\boldsymbol{\Sigma }\left( \boldsymbol{\theta }\right) ^{-1}\right) \left( 
\boldsymbol{y}_{i}-\boldsymbol{\mu }\left( \boldsymbol{\theta }\right)
\right) \right] \\
&&\left[ 2\left( \frac{\partial \boldsymbol{\mu }\left( \boldsymbol{\theta }%
\right) }{\partial \theta _{j}}\right) ^{T}\boldsymbol{\Sigma }\left( 
\boldsymbol{\theta }\right) ^{-1}\left( \boldsymbol{y}_{i}-\boldsymbol{\mu }%
\left( \boldsymbol{\theta }\right) \right) +\left( \boldsymbol{y}_{i}-%
\boldsymbol{\mu }\left( \boldsymbol{\theta }\right) \right) ^{T}\left( 
\boldsymbol{\Sigma }\left( \boldsymbol{\theta }\right) ^{-1}\dfrac{\partial 
\boldsymbol{\Sigma }\left( \boldsymbol{\theta }\right) }{\partial \theta _{j}%
}\boldsymbol{\Sigma }\left( \boldsymbol{\theta }\right) ^{-1}\right) \left( 
\boldsymbol{y}_{i}-\boldsymbol{\mu }\left( \boldsymbol{\theta }\right)
\right) \right] .
\end{eqnarray*}%
Therefore,%
\begin{eqnarray*}
C_{9} &=&\exp \left \{ -\frac{2\tau }{2}\left( \boldsymbol{y}_{i}-%
\boldsymbol{\mu }\left( \boldsymbol{\theta }\right) \right) ^{T}\boldsymbol{%
\Sigma }\left( \boldsymbol{\theta }\right) ^{-1}\left( \boldsymbol{y}_{i}-%
\boldsymbol{\mu }\left( \boldsymbol{\theta }\right) \right) \right \} \left
\{ 4\left( \boldsymbol{y}_{i}-\boldsymbol{\mu }\left( \boldsymbol{\theta }%
\right) \right) ^{T}\boldsymbol{\Sigma }\left( \boldsymbol{\theta }\right)
^{-1}\left( \frac{\partial \boldsymbol{\mu }\left( \boldsymbol{\theta }%
\right) }{\partial \theta _{i}}\right) ^{T}\right. \\
&&\left. \frac{\partial \boldsymbol{\mu }\left( \boldsymbol{\theta }\right) 
}{\partial \theta _{j}}\boldsymbol{\Sigma }\left( \boldsymbol{\theta }%
\right) ^{-1}\left( \boldsymbol{y}_{i}-\boldsymbol{\mu }\left( \boldsymbol{%
\theta }\right) \right) \right. \\
&&\left. +2\left( \frac{\partial \boldsymbol{\mu }\left( \boldsymbol{\theta }%
\right) }{\partial \theta _{i}}\right) ^{T}\boldsymbol{\Sigma }\left( 
\boldsymbol{\theta }\right) ^{-1}\left( \boldsymbol{y}_{i}-\boldsymbol{\mu }%
\left( \boldsymbol{\theta }\right) \right) \left( \boldsymbol{y}-\boldsymbol{%
\mu }\left( \boldsymbol{\theta }\right) \right) ^{T}\left( \boldsymbol{%
\Sigma }\left( \boldsymbol{\theta }\right) ^{-1}\dfrac{\partial \boldsymbol{%
\Sigma }\left( \boldsymbol{\theta }\right) }{\partial \theta _{j}}%
\boldsymbol{\Sigma }\left( \boldsymbol{\theta }\right) ^{-1}\right) \left( 
\boldsymbol{y}_{i}-\boldsymbol{\mu }\left( \boldsymbol{\theta }\right)
\right) \right. \\
&&\left. +2\left( \boldsymbol{y}_{i}-\boldsymbol{\mu }\left( \boldsymbol{%
\theta }\right) \right) ^{T}\left( \boldsymbol{\Sigma }\left( \boldsymbol{%
\theta }\right) ^{-1}\dfrac{\partial \boldsymbol{\Sigma }\left( \boldsymbol{%
\theta }\right) }{\partial \theta _{i}}\boldsymbol{\Sigma }\left( 
\boldsymbol{\theta }\right) ^{-1}\right) \left( \boldsymbol{y}_{i}-%
\boldsymbol{\mu }\left( \boldsymbol{\theta }\right) \right) \left( \frac{%
\partial \boldsymbol{\mu }\left( \boldsymbol{\theta }\right) }{\partial
\theta _{j}}\right) ^{T}\boldsymbol{\Sigma }\left( \boldsymbol{\theta }%
\right) ^{-1}\left( \boldsymbol{y}_{i}-\boldsymbol{\mu }\left( \boldsymbol{%
\theta }\right) \right) \right. \\
&&\left. +\left( \boldsymbol{y}_{i}-\boldsymbol{\mu }\left( \boldsymbol{%
\theta }\right) \right) ^{T}\left( \boldsymbol{\Sigma }\left( \boldsymbol{%
\theta }\right) ^{-1}\dfrac{\partial \boldsymbol{\Sigma }\left( \boldsymbol{%
\theta }\right) }{\partial \theta _{i}}\boldsymbol{\Sigma }\left( 
\boldsymbol{\theta }\right) ^{-1}\right) \left( \boldsymbol{y}_{i}-%
\boldsymbol{\mu }\left( \boldsymbol{\theta }\right) \right) \right \} \\
&=&A_{1}+A_{2}+A_{3}+A_{4}
\end{eqnarray*}%
and 
\begin{equation*}
E\left[ C_{9}\right] =E\left[ A_{1}\right] +E\left[ A_{4}\right]
\end{equation*}%
because $E\left[ A_{2}\right] =E\left[ A_{3}\right] =0.$ We have, 
\begin{eqnarray*}
E\left[ A_{1}\right] &=&E\left[ \exp -\frac{2\tau }{2}\left( \boldsymbol{Y}-%
\boldsymbol{\mu }\left( \boldsymbol{\theta }\right) \right) ^{T}\boldsymbol{%
\Sigma }\left( \boldsymbol{\theta }\right) ^{-1}\left( \boldsymbol{Y}-%
\boldsymbol{\mu }\left( \boldsymbol{\theta }\right) \right) 4\left( 
\boldsymbol{Y}-\boldsymbol{\mu }\left( \boldsymbol{\theta }\right) \right)
^{T}\boldsymbol{\Sigma }\left( \boldsymbol{\theta }\right) ^{-1}\right. \\
&&\left. \left( \frac{\partial \boldsymbol{\mu }\left( \boldsymbol{\theta }%
\right) }{\partial \theta _{i}}\right) ^{T}\frac{\partial \boldsymbol{\mu }%
\left( \boldsymbol{\theta }\right) }{\partial \theta _{j}}\boldsymbol{\Sigma 
}\left( \boldsymbol{\theta }\right) ^{-1}\left( \boldsymbol{Y}-\boldsymbol{%
\mu }\left( \boldsymbol{\theta }\right) \right) \right] \\
&=&4\dint \frac{1}{\left( 2\pi \right) ^{m/2}\left \vert \boldsymbol{\Sigma }%
\left( \boldsymbol{\theta }\right) \right \vert ^{1/2}}\exp \left \{ -\frac{1%
}{2}\left( \boldsymbol{y}-\boldsymbol{\mu }\left( \boldsymbol{\theta }%
\right) \right) ^{T}\left( \frac{\boldsymbol{\Sigma }\left( \boldsymbol{%
\theta }\right) }{2\tau +1}\right) ^{-1}\left( \boldsymbol{y}-\boldsymbol{%
\mu }\left( \boldsymbol{\theta }\right) \right) \right \} \left( \boldsymbol{%
y}-\boldsymbol{\mu }\left( \boldsymbol{\theta }\right) \right) ^{T} \\
&&\boldsymbol{\Sigma }\left( \boldsymbol{\theta }\right) ^{-1}\left( \frac{%
\partial \boldsymbol{\mu }\left( \boldsymbol{\theta }\right) }{\partial
\theta _{i}}\right) ^{T}\frac{\partial \boldsymbol{\mu }\left( \boldsymbol{%
\theta }\right) }{\partial \theta _{j}}\boldsymbol{\Sigma }\left( 
\boldsymbol{\theta }\right) ^{-1}\left( \boldsymbol{y}-\boldsymbol{\mu }%
\left( \boldsymbol{\theta }\right) \right) d\boldsymbol{y} \\
&=&4\frac{1}{\left( 2\tau +1\right) ^{m/2}}\dint \frac{1}{\left( 2\pi
\right) ^{m/2}\left \vert \frac{\boldsymbol{\Sigma }\left( \boldsymbol{%
\theta }\right) }{2\tau +1}\right \vert ^{1/2}}\exp \left \{ -\frac{1}{2}%
\left( \boldsymbol{y}-\boldsymbol{\mu }\left( \boldsymbol{\theta }\right)
\right) ^{T}\left( \frac{\boldsymbol{\Sigma }\left( \boldsymbol{\theta }%
\right) }{2\tau +1}\right) ^{-1}\left( \boldsymbol{y}-\boldsymbol{\mu }%
\left( \boldsymbol{\theta }\right) \right) \right \} \\
&&\left( \boldsymbol{y}-\boldsymbol{\mu }\left( \boldsymbol{\theta }\right)
\right) ^{T}\boldsymbol{\Sigma }\left( \boldsymbol{\theta }\right)
^{-1}\left( \frac{\partial \boldsymbol{\mu }\left( \boldsymbol{\theta }%
\right) }{\partial \theta _{i}}\right) ^{T}\frac{\partial \boldsymbol{\mu }%
\left( \boldsymbol{\theta }\right) }{\partial \theta _{j}}\boldsymbol{\Sigma 
}\left( \boldsymbol{\theta }\right) ^{-1}\left( \boldsymbol{y}-\boldsymbol{%
\mu }\left( \boldsymbol{\theta }\right) \right) d\boldsymbol{y} \\
&=&4\frac{1}{\left( 2\tau +1\right) ^{m/2}}trace\left( \boldsymbol{\Sigma }%
\left( \boldsymbol{\theta }\right) ^{-1}\left( \frac{\partial \boldsymbol{%
\mu }\left( \boldsymbol{\theta }\right) }{\partial \theta _{i}}\right) ^{T}%
\frac{\partial \boldsymbol{\mu }\left( \boldsymbol{\theta }\right) }{%
\partial \theta _{j}}\boldsymbol{\Sigma }\left( \boldsymbol{\theta }\right)
^{-1}\frac{\boldsymbol{\Sigma }\left( \boldsymbol{\theta }\right) }{2\tau +1}%
\right) \\
&=&\frac{4}{\left( 2\tau +1\right) ^{m/2+1}}trace\left( \boldsymbol{\Sigma }%
\left( \boldsymbol{\theta }\right) ^{-1}\left( \frac{\partial \boldsymbol{%
\mu }\left( \boldsymbol{\theta }\right) }{\partial \theta _{i}}\right) ^{T}%
\frac{\partial \boldsymbol{\mu }\left( \boldsymbol{\theta }\right) }{%
\partial \theta _{j}}\right) .
\end{eqnarray*}%
Now we are going to get $E\left[ A_{4}\right] .$%
\begin{eqnarray*}
E\left[ A_{4}\right] &=&E\left[ \exp \left \{ -\frac{2\tau }{2}\left( 
\boldsymbol{Y}-\boldsymbol{\mu }\left( \boldsymbol{\theta }\right) \right)
^{T}\boldsymbol{\Sigma }\left( \boldsymbol{\theta }\right) ^{-1}\left( 
\boldsymbol{Y}-\boldsymbol{\mu }\left( \boldsymbol{\theta }\right) \right)
\right \} \left( \boldsymbol{Y}-\boldsymbol{\mu }\left( \boldsymbol{\theta }%
\right) \right) ^{T}\left( \boldsymbol{\Sigma }\left( \boldsymbol{\theta }%
\right) ^{-1}\dfrac{\partial \boldsymbol{\Sigma }\left( \boldsymbol{\theta }%
\right) }{\partial \theta _{i}}\boldsymbol{\Sigma }\left( \boldsymbol{\theta 
}\right) ^{-1}\right) \right. \\
&&\left. \left( \boldsymbol{Y}-\boldsymbol{\mu }\left( \boldsymbol{\theta }%
\right) \right) \left( \boldsymbol{Y}-\boldsymbol{\mu }\left( \boldsymbol{%
\theta }\right) \right) ^{T}\left( \boldsymbol{\Sigma }\left( \boldsymbol{%
\theta }\right) ^{-1}\dfrac{\partial \boldsymbol{\Sigma }\left( \boldsymbol{%
\theta }\right) }{\partial \theta _{j}}\boldsymbol{\Sigma }\left( 
\boldsymbol{\theta }\right) ^{-1}\right) \left( \boldsymbol{Y}-\boldsymbol{%
\mu }\left( \boldsymbol{\theta }\right) \right) \right] \\
&=&\dint \frac{1}{\left( 2\pi \right) ^{m/2}\left \vert \boldsymbol{\Sigma }%
\left( \boldsymbol{\theta }\right) \right \vert ^{1/2}}\exp \left \{ -\frac{1%
}{2}\left( \boldsymbol{y}-\boldsymbol{\mu }\left( \boldsymbol{\theta }%
\right) \right) ^{T}\left( \frac{\boldsymbol{\Sigma }\left( \boldsymbol{%
\theta }\right) }{2\tau +1}\right) ^{-1}\left( \boldsymbol{y}-\boldsymbol{%
\mu }\left( \boldsymbol{\theta }\right) \right) \right \} \left( \boldsymbol{%
y}-\boldsymbol{\mu }\left( \boldsymbol{\theta }\right) \right) ^{T} \\
&&\left( \boldsymbol{\Sigma }\left( \boldsymbol{\theta }\right) ^{-1}\dfrac{%
\partial \boldsymbol{\Sigma }\left( \boldsymbol{\theta }\right) }{\partial
\theta _{i}}\boldsymbol{\Sigma }\left( \boldsymbol{\theta }\right)
^{-1}\right) \left( \boldsymbol{y}-\boldsymbol{\mu }\left( \boldsymbol{%
\theta }\right) \right) \left( \boldsymbol{y}-\boldsymbol{\mu }\left( 
\boldsymbol{\theta }\right) \right) ^{T} \\
&&\left( \boldsymbol{\Sigma }\left( \boldsymbol{\theta }\right) ^{-1}\dfrac{%
\partial \boldsymbol{\Sigma }\left( \boldsymbol{\theta }\right) }{\partial
\theta _{j}}\boldsymbol{\Sigma }\left( \boldsymbol{\theta }\right)
^{-1}\right) \left( \boldsymbol{y}-\boldsymbol{\mu }\left( \boldsymbol{%
\theta }\right) \right) d\boldsymbol{y} \\
&=&\frac{1}{\left( 2\tau +1\right) ^{m/2}}\left( trace\left( \boldsymbol{%
\Sigma }\left( \boldsymbol{\theta }\right) ^{-1}\dfrac{\partial \boldsymbol{%
\Sigma }\left( \boldsymbol{\theta }\right) }{\partial \theta _{i}}%
\boldsymbol{\Sigma }\left( \boldsymbol{\theta }\right) ^{-1}\frac{%
\boldsymbol{\Sigma }\left( \boldsymbol{\theta }\right) }{2\tau +1}\right)
\right. \\
&&\left. \left[ \boldsymbol{\Sigma }\left( \boldsymbol{\theta }\right) ^{-1}%
\dfrac{\partial \boldsymbol{\Sigma }\left( \boldsymbol{\theta }\right) }{%
\partial \theta _{i}}\boldsymbol{\Sigma }\left( \boldsymbol{\theta }\right)
^{-1}+\boldsymbol{\Sigma }\left( \boldsymbol{\theta }\right) ^{-1}\dfrac{%
\partial \boldsymbol{\Sigma }\left( \boldsymbol{\theta }\right) }{\partial
\theta _{j}}\boldsymbol{\Sigma }\left( \boldsymbol{\theta }\right) ^{-1}%
\right] \frac{\boldsymbol{\Sigma }\left( \boldsymbol{\theta }\right) }{2\tau
+1}\right) \\
&&+\frac{1}{\left( 2\tau +1\right) ^{m/2}}trace\left( \boldsymbol{\Sigma }%
\left( \boldsymbol{\theta }\right) ^{-1}\dfrac{\partial \boldsymbol{\Sigma }%
\left( \boldsymbol{\theta }\right) }{\partial \theta _{i}}\boldsymbol{\Sigma 
}\left( \boldsymbol{\theta }\right) ^{-1}\frac{\boldsymbol{\Sigma }\left( 
\boldsymbol{\theta }\right) }{2\tau +1}\right) \\
&&\times trace\left( \boldsymbol{\Sigma }\left( \boldsymbol{\theta }\right)
^{-1}\dfrac{\partial \boldsymbol{\Sigma }\left( \boldsymbol{\theta }\right) 
}{\partial \theta _{j}}\boldsymbol{\Sigma }\left( \boldsymbol{\theta }%
\right) ^{-1}\frac{\boldsymbol{\Sigma }\left( \boldsymbol{\theta }\right) }{%
2\tau +1}\right) \\
&=&\frac{1}{\left( 2\tau +1\right) ^{\frac{m}{2}+2}}trace\left( \boldsymbol{%
\Sigma }\left( \boldsymbol{\theta }\right) ^{-1}\dfrac{\partial \boldsymbol{%
\Sigma }\left( \boldsymbol{\theta }\right) }{\partial \theta _{i}}%
\boldsymbol{\Sigma }\left( \boldsymbol{\theta }\right) ^{-1}\dfrac{\partial 
\boldsymbol{\Sigma }\left( \boldsymbol{\theta }\right) }{\partial \theta _{j}%
}\right) \\
&&+\frac{1}{\left( 2\tau +1\right) ^{\frac{m}{2}+2}}trace\left( \boldsymbol{%
\Sigma }\left( \boldsymbol{\theta }\right) ^{-1}\dfrac{\partial \boldsymbol{%
\Sigma }\left( \boldsymbol{\theta }\right) }{\partial \theta _{i}}%
\boldsymbol{\Sigma }\left( \boldsymbol{\theta }\right) ^{-1}\left( \dfrac{%
\partial \boldsymbol{\Sigma }\left( \boldsymbol{\theta }\right) }{\partial
\theta _{j}}\right) ^{T}\right) \\
&&+\frac{1}{\left( 2\tau +1\right) ^{\frac{m}{2}+2}}trace\left( \boldsymbol{%
\Sigma }\left( \boldsymbol{\theta }\right) ^{-1}\dfrac{\partial \boldsymbol{%
\Sigma }\left( \boldsymbol{\theta }\right) }{\partial \theta _{i}}\right)
trace\left( \boldsymbol{\Sigma }\left( \boldsymbol{\theta }\right) ^{-1}%
\dfrac{\partial \boldsymbol{\Sigma }\left( \boldsymbol{\theta }\right) }{%
\partial \theta _{j}}\right) \\
&=&2\frac{1}{\left( 2\tau +1\right) ^{\frac{m}{2}+2}}trace\left( \boldsymbol{%
\Sigma }\left( \boldsymbol{\theta }\right) ^{-1}\dfrac{\partial \boldsymbol{%
\Sigma }\left( \boldsymbol{\theta }\right) }{\partial \theta _{i}}%
\boldsymbol{\Sigma }\left( \boldsymbol{\theta }\right) ^{-1}\dfrac{\partial 
\boldsymbol{\Sigma }\left( \boldsymbol{\theta }\right) }{\partial \theta _{j}%
}\right) \\
&&+\frac{1}{\left( 2\tau +1\right) ^{\frac{m}{2}+2}}trace\left( \boldsymbol{%
\Sigma }\left( \boldsymbol{\theta }\right) ^{-1}\dfrac{\partial \boldsymbol{%
\Sigma }\left( \boldsymbol{\theta }\right) }{\partial \theta _{i}}\right)
trace\left( \boldsymbol{\Sigma }\left( \boldsymbol{\theta }\right) ^{-1}%
\dfrac{\partial \boldsymbol{\Sigma }\left( \boldsymbol{\theta }\right) }{%
\partial \theta _{j}}\right) .
\end{eqnarray*}%
Based on the previous results we have,%
\begin{eqnarray*}
E\left[ Y_{\tau }^{i}(\boldsymbol{\theta })Y_{\tau }^{j}(\boldsymbol{\theta }%
)\right] &=&a^{2}\left( \frac{\tau }{2}\right) ^{2}\left \vert \boldsymbol{%
\Sigma }\left( \boldsymbol{\theta }\right) \right \vert ^{-\tau } \\
&&\left \{ \frac{1}{\left( 2\tau +1\right) ^{\frac{m}{2}}}trace\left( 
\boldsymbol{\Sigma }\left( \boldsymbol{\theta }\right) ^{-1}\dfrac{\partial 
\boldsymbol{\Sigma }\left( \boldsymbol{\theta }\right) }{\partial \theta _{i}%
}\right) trace\left( \boldsymbol{\Sigma }\left( \boldsymbol{\theta }\right)
^{-1}\dfrac{\partial \boldsymbol{\Sigma }\left( \boldsymbol{\theta }\right) 
}{\partial \theta _{j}}\right) \right. \\
&&\left. -\frac{\tau }{\left( \tau +1\right) ^{m+1}}trace\left( \boldsymbol{%
\Sigma }\left( \boldsymbol{\theta }\right) ^{-1}\dfrac{\partial \boldsymbol{%
\Sigma }\left( \boldsymbol{\theta }\right) }{\partial \theta _{i}}\right)
trace\left( \boldsymbol{\Sigma }\left( \boldsymbol{\theta }\right) ^{-1}%
\dfrac{\partial \boldsymbol{\Sigma }\left( \boldsymbol{\theta }\right) }{%
\partial \theta _{j}}\right) \right. \\
&&\left. -\frac{1}{\left( 2\tau +1\right) ^{\frac{m}{2}+1}}trace\left( 
\boldsymbol{\Sigma }\left( \boldsymbol{\theta }\right) ^{-1}\dfrac{\partial 
\boldsymbol{\Sigma }\left( \boldsymbol{\theta }\right) }{\partial \theta _{i}%
}\right) trace\left( \boldsymbol{\Sigma }\left( \boldsymbol{\theta }\right)
^{-1}\dfrac{\partial \boldsymbol{\Sigma }\left( \boldsymbol{\theta }\right) 
}{\partial \theta _{j}}\right) \right. \\
&&\left. -\frac{\tau }{\left( \tau +1\right) ^{m+1}}trace\left( \boldsymbol{%
\Sigma }\left( \boldsymbol{\theta }\right) ^{-1}\dfrac{\partial \boldsymbol{%
\Sigma }\left( \boldsymbol{\theta }\right) }{\partial \theta _{i}}\right)
trace\left( \boldsymbol{\Sigma }\left( \boldsymbol{\theta }\right) ^{-1}%
\dfrac{\partial \boldsymbol{\Sigma }\left( \boldsymbol{\theta }\right) }{%
\partial \theta _{j}}\right) \right. \\
&&\left. +\frac{\tau ^{2}}{\left( \tau +1\right) ^{m+2}}trace\left( 
\boldsymbol{\Sigma }\left( \boldsymbol{\theta }\right) ^{-1}\dfrac{\partial 
\boldsymbol{\Sigma }\left( \boldsymbol{\theta }\right) }{\partial \theta _{i}%
}\right) trace\left( \boldsymbol{\Sigma }\left( \boldsymbol{\theta }\right)
^{-1}\dfrac{\partial \boldsymbol{\Sigma }\left( \boldsymbol{\theta }\right) 
}{\partial \theta _{j}}\right) \right. \\
&&\left. +\frac{\tau }{\left( \tau +1\right) ^{m+2}}trace\left( \boldsymbol{%
\Sigma }\left( \boldsymbol{\theta }\right) ^{-1}\dfrac{\partial \boldsymbol{%
\Sigma }\left( \boldsymbol{\theta }\right) }{\partial \theta _{i}}\right)
trace\left( \boldsymbol{\Sigma }\left( \boldsymbol{\theta }\right) ^{-1}%
\dfrac{\partial \boldsymbol{\Sigma }\left( \boldsymbol{\theta }\right) }{%
\partial \theta _{j}}\right) \right. \\
&&\left. -\frac{1}{\left( 2\tau +1\right) ^{\frac{m}{2}+1}}trace\left( 
\boldsymbol{\Sigma }\left( \boldsymbol{\theta }\right) ^{-1}\dfrac{\partial 
\boldsymbol{\Sigma }\left( \boldsymbol{\theta }\right) }{\partial \theta _{i}%
}\right) trace\left( \boldsymbol{\Sigma }\left( \boldsymbol{\theta }\right)
^{-1}\dfrac{\partial \boldsymbol{\Sigma }\left( \boldsymbol{\theta }\right) 
}{\partial \theta _{j}}\right) \right. \\
&&\left. +\frac{\tau }{\left( \tau +1\right) ^{m+2}}trace\left( \boldsymbol{%
\Sigma }\left( \boldsymbol{\theta }\right) ^{-1}\dfrac{\partial \boldsymbol{%
\Sigma }\left( \boldsymbol{\theta }\right) }{\partial \theta _{i}}\right)
trace\left( \boldsymbol{\Sigma }\left( \boldsymbol{\theta }\right) ^{-1}%
\dfrac{\partial \boldsymbol{\Sigma }\left( \boldsymbol{\theta }\right) }{%
\partial \theta _{j}}\right) \right. \\
&&\left. +\frac{4}{\left( 2\tau +1\right) ^{\frac{m}{2}+1}}trace\left( 
\boldsymbol{\Sigma }\left( \boldsymbol{\theta }\right) ^{-1}\left( \frac{%
\partial \boldsymbol{\mu }\left( \boldsymbol{\theta }\right) }{\partial
\theta _{i}}\right) ^{T}\frac{\partial \boldsymbol{\mu }\left( \boldsymbol{%
\theta }\right) }{\partial \theta _{j}}\right) \right. \\
&&\left. +\frac{2}{\left( 2\tau +1\right) ^{\frac{m}{2}+2}}trace\left( 
\boldsymbol{\Sigma }\left( \boldsymbol{\theta }\right) ^{-1}\dfrac{\partial 
\boldsymbol{\Sigma }\left( \boldsymbol{\theta }\right) }{\partial \theta _{i}%
}\boldsymbol{\Sigma }\left( \boldsymbol{\theta }\right) ^{-1}\dfrac{\partial 
\boldsymbol{\Sigma }\left( \boldsymbol{\theta }\right) }{\partial \theta _{j}%
}\right) \right. \\
&&\left. +\frac{1}{\left( 2\tau +1\right) ^{\frac{m}{2}+2}}trace\left( 
\boldsymbol{\Sigma }\left( \boldsymbol{\theta }\right) ^{-1}\dfrac{\partial 
\boldsymbol{\Sigma }\left( \boldsymbol{\theta }\right) }{\partial \theta _{i}%
}\right) trace\left( \boldsymbol{\Sigma }\left( \boldsymbol{\theta }\right)
^{-1}\dfrac{\partial \boldsymbol{\Sigma }\left( \boldsymbol{\theta }\right) 
}{\partial \theta _{j}}\right) \right \} .
\end{eqnarray*}%
The previous expression can be written as, 
\begin{eqnarray*}
E\left[ Y_{\tau }^{i}(\boldsymbol{\theta })Y_{\tau }^{j}(\boldsymbol{\theta }%
)\right] &=&a^{2}\left( \frac{\tau }{2}\right) ^{2}\boldsymbol{\Sigma }%
\left( \boldsymbol{\theta }\right) ^{-\tau }\left \{ trace\left( \boldsymbol{%
\Sigma }\left( \boldsymbol{\theta }\right) ^{-1}\dfrac{\partial \boldsymbol{%
\Sigma }\left( \boldsymbol{\theta }\right) }{\partial \theta _{i}}\right)
trace\left( \boldsymbol{\Sigma }\left( \boldsymbol{\theta }\right) ^{-1}%
\dfrac{\partial \boldsymbol{\Sigma }\left( \boldsymbol{\theta }\right) }{%
\partial \theta _{j}}\right) \right. \\
&&\left. \left[ \frac{1}{\left( 2\tau +1\right) ^{\frac{m}{2}}}-\frac{1}{%
\left( 2\tau +1\right) ^{\frac{m}{2}+1}}-\frac{1}{\left( 2\tau +1\right) ^{%
\frac{m}{2}+1}}+\frac{1}{\left( 2\tau +1\right) ^{\frac{m}{2}+2}}\right]
\right. \\
&&\left. +trace\left( \boldsymbol{\Sigma }\left( \boldsymbol{\theta }\right)
^{-1}\dfrac{\partial \boldsymbol{\Sigma }\left( \boldsymbol{\theta }\right) 
}{\partial \theta _{i}}\right) trace\left( \boldsymbol{\Sigma }\left( 
\boldsymbol{\theta }\right) ^{-1}\dfrac{\partial \boldsymbol{\Sigma }\left( 
\boldsymbol{\theta }\right) }{\partial \theta _{j}}\right) \right. \\
&&\left. \left[ -\frac{\tau }{\left( \tau +1\right) ^{m+1}}-\frac{\tau }{%
\left( \tau +1\right) ^{m+1}}+\frac{\tau ^{2}}{\left( \tau +1\right) ^{m+2}}+%
\frac{\tau }{\left( \tau +1\right) ^{m+2}}+\frac{\tau }{\left( \tau
+1\right) ^{m+2}}\right] \right. \\
&&\left. +\frac{4}{\left( 2\tau +1\right) ^{\frac{m}{2}+1}}trace\left( 
\boldsymbol{\Sigma }\left( \boldsymbol{\theta }\right) ^{-1}\left( \frac{%
\partial \boldsymbol{\mu }\left( \boldsymbol{\theta }\right) }{\partial
\theta _{i}}\right) ^{T}\frac{\partial \boldsymbol{\mu }\left( \boldsymbol{%
\theta }\right) }{\partial \theta _{j}}\right) \right. \\
&&\left. +\frac{2}{\left( 2\tau +1\right) ^{\frac{m}{2}+2}}trace\left( 
\boldsymbol{\Sigma }\left( \boldsymbol{\theta }\right) ^{-1}\dfrac{\partial 
\boldsymbol{\Sigma }\left( \boldsymbol{\theta }\right) }{\partial \theta _{i}%
}\boldsymbol{\Sigma }\left( \boldsymbol{\theta }\right) ^{-1}\dfrac{\partial 
\boldsymbol{\Sigma }\left( \boldsymbol{\theta }\right) }{\partial \theta _{j}%
}\right) \right \} .
\end{eqnarray*}%
Therefore, 
\begin{eqnarray*}
E\left[ Y_{\tau }^{i}(\boldsymbol{\theta })Y_{\tau }^{j}(\boldsymbol{\theta }%
)\right] &=&\left( \frac{\tau +1}{\tau \left( 2\pi \right) ^{m\tau /2}}%
\right) ^{2}\left( \frac{\tau }{2}\right) ^{2}\left \vert \boldsymbol{\Sigma 
}\left( \boldsymbol{\theta }\right) \right \vert ^{-\tau } \\
&&\left \{ \frac{4\tau ^{2}}{\left( 2\tau +1\right) ^{\frac{m}{2}+2}}%
trace\left( \boldsymbol{\Sigma }\left( \boldsymbol{\theta }\right) ^{-1}%
\dfrac{\partial \boldsymbol{\Sigma }\left( \boldsymbol{\theta }\right) }{%
\partial \theta _{i}}\right) trace\left( \boldsymbol{\Sigma }\left( 
\boldsymbol{\theta }\right) ^{-1}\dfrac{\partial \boldsymbol{\Sigma }\left( 
\boldsymbol{\theta }\right) }{\partial \theta _{j}}\right) \right. \\
&&\left. -\frac{\tau ^{2}}{\left( 1+\tau \right) ^{m+2}}trace\left( 
\boldsymbol{\Sigma }\left( \boldsymbol{\theta }\right) ^{-1}\dfrac{\partial 
\boldsymbol{\Sigma }\left( \boldsymbol{\theta }\right) }{\partial \theta _{i}%
}\right) trace\left( \boldsymbol{\Sigma }\left( \boldsymbol{\theta }\right)
^{-1}\dfrac{\partial \boldsymbol{\Sigma }\left( \boldsymbol{\theta }\right) 
}{\partial \theta _{j}}\right) \right. \\
&&+\left. 4\frac{4}{\left( 2\tau +1\right) ^{\frac{m}{2}+1}}trace\left( 
\boldsymbol{\Sigma }\left( \boldsymbol{\theta }\right) ^{-1}\left( \frac{%
\partial \boldsymbol{\mu }\left( \boldsymbol{\theta }\right) }{\partial
\theta _{i}}\right) ^{T}\frac{\partial \boldsymbol{\mu }\left( \boldsymbol{%
\theta }\right) }{\partial \theta _{j}}\right) \right. \\
&&\left. +\frac{2}{\left( 2\tau +1\right) ^{\frac{m}{2}+2}}trace\left( 
\boldsymbol{\Sigma }\left( \boldsymbol{\theta }\right) ^{-1}\dfrac{\partial 
\boldsymbol{\Sigma }\left( \boldsymbol{\theta }\right) }{\partial \theta _{i}%
}\boldsymbol{\Sigma }\left( \boldsymbol{\theta }\right) ^{-1}\dfrac{\partial 
\boldsymbol{\Sigma }\left( \boldsymbol{\theta }\right) }{\partial \theta _{j}%
}\right) \right \} .
\end{eqnarray*}%
Finally,%
\begin{eqnarray*}
E\left[ Y_{\tau }^{i}(\boldsymbol{\theta })Y_{\tau }^{j}(\boldsymbol{\theta }%
)\right] &=&\left( \tau +1\right) ^{2}\left \{ \left( \frac{1}{\left( 2\pi
\right) ^{m\tau /2}\left \vert \boldsymbol{\Sigma }\left( \boldsymbol{\theta 
}\right) \right \vert ^{1/2}}\right) ^{2\tau }\frac{1}{\left( 2\tau
+1\right) ^{\frac{m}{2}+2}}\right. \\
&&\left. \left( \Delta _{2\tau }^{i}\Delta _{2\tau }^{j}+\left( 2\tau
+1\right) trace\left( \boldsymbol{\Sigma }\left( \boldsymbol{\theta }\right)
^{-1}\left( \frac{\partial \boldsymbol{\mu }\left( \boldsymbol{\theta }%
\right) }{\partial \theta _{i}}\right) ^{T}\frac{\partial \boldsymbol{\mu }%
\left( \boldsymbol{\theta }\right) }{\partial \theta _{j}}\right) \right.
\right. \\
&&\left. \left. +\frac{1}{2}trace\left( \boldsymbol{\Sigma }\left( 
\boldsymbol{\theta }\right) ^{-1}\dfrac{\partial \boldsymbol{\Sigma }\left( 
\boldsymbol{\theta }\right) }{\partial \theta _{i}}\boldsymbol{\Sigma }%
\left( \boldsymbol{\theta }\right) ^{-1}\dfrac{\partial \boldsymbol{\Sigma }%
\left( \boldsymbol{\theta }\right) }{\partial \theta _{j}}\right) \right)
\right. \\
&&-\left. \left( \frac{1}{\left( 2\pi \right) ^{m\tau /2}\left \vert 
\boldsymbol{\Sigma }\left( \boldsymbol{\theta }\right) \right \vert ^{1/2}}%
\right) ^{2\tau }\frac{1}{\left( 1+\tau \right) ^{m+2}}\Delta _{\tau
}^{i}\Delta _{\tau }^{j}\right \} \\
&=&\left( \tau +1\right) ^{2}K_{\tau }^{ij}\left( \boldsymbol{\theta }\right)
\end{eqnarray*}%
where $K_{\tau }^{ij}\left( \boldsymbol{\theta }\right) $ was defined in (%
\ref{103aaa}).

Then%
\begin{equation*}
\sqrt{n}\frac{\partial }{\partial \boldsymbol{\theta }}H_{n}(\boldsymbol{%
\theta })\underset{n\longrightarrow \infty }{\overset{\mathcal{L}}{%
\longrightarrow }}\mathcal{N}\left( \boldsymbol{0},\left( \tau +1\right) ^{2}%
\boldsymbol{K}_{\boldsymbol{\tau }}\left( \boldsymbol{\theta }\right) \right)
\end{equation*}%
and 
\begin{equation*}
\sqrt{n}\left( \frac{1}{\tau +1}\frac{\partial }{\partial \boldsymbol{\theta 
}}H_{n}(\boldsymbol{\theta })\right) \underset{n\longrightarrow \infty }{%
\overset{\mathcal{L}}{\longrightarrow }}\mathcal{N}\left( \boldsymbol{0},%
\boldsymbol{K}_{\boldsymbol{\tau }}\left( \boldsymbol{\theta }\right)
\right) .
\end{equation*}

\subsection{Appendix B (Proof of Proposition \protect\ref{Proposition2})}

We have,%
\begin{eqnarray*}
\frac{\partial }{\partial \theta _{i}}H_{n}^{\tau }(\boldsymbol{\theta })
&=&-a\frac{\tau }{2}\left \vert \boldsymbol{\Sigma }\left( \boldsymbol{%
\theta }\right) \right \vert ^{-\tau /2}trace\left( \boldsymbol{\Sigma }%
\left( \boldsymbol{\theta }\right) ^{-1}\dfrac{\partial \boldsymbol{\Sigma }%
\left( \boldsymbol{\theta }\right) }{\partial \theta _{i}}\right) \\
&&\left[ \frac{1}{n}\tsum \limits_{i=1}^{n}\exp \left \{ -\frac{\tau }{2}%
\left( \boldsymbol{y}_{i}-\boldsymbol{\mu }\left( \boldsymbol{\theta }%
\right) \right) ^{T}\boldsymbol{\Sigma }\left( \boldsymbol{\theta }\right)
^{-1}\left( \boldsymbol{y}_{i}-\boldsymbol{\mu }\left( \boldsymbol{\theta }%
\right) \right) \right \} -b\right] \\
&&+a\text{ }\left \vert \boldsymbol{\Sigma }\left( \boldsymbol{\theta }%
\right) \right \vert ^{-\tau /2}\text{ }\left[ \frac{1}{n}\tsum
\limits_{i=1}^{n}\exp \left \{ -\frac{\tau }{2}\left( \boldsymbol{y}_{i}-%
\boldsymbol{\mu }\left( \boldsymbol{\theta }\right) \right) ^{T}\boldsymbol{%
\Sigma }\left( \boldsymbol{\theta }\right) ^{-1}\left( \boldsymbol{y}_{i}-%
\boldsymbol{\mu }\left( \boldsymbol{\theta }\right) \right) \right \} \left(
-\frac{\tau }{2}\right) \right. \\
&&\left( 2\left( \frac{\partial \boldsymbol{\mu }\left( \boldsymbol{\theta }%
\right) }{\partial \theta _{i}}\right) ^{T}\boldsymbol{\Sigma }\left( 
\boldsymbol{\theta }\right) ^{-1}\left( \boldsymbol{y}_{i}-\boldsymbol{\mu }%
\left( \boldsymbol{\theta }\right) \right) \right. \\
&&\left. \left. -\left( \boldsymbol{y}_{i}-\boldsymbol{\mu }\left( 
\boldsymbol{\theta }\right) \right) ^{T}\left( \boldsymbol{\Sigma }\left( 
\boldsymbol{\theta }\right) ^{-1}\dfrac{\partial \boldsymbol{\Sigma }\left( 
\boldsymbol{\theta }\right) }{\partial \theta _{j}}\boldsymbol{\Sigma }%
\left( \boldsymbol{\theta }\right) ^{-1}\right) \left( \boldsymbol{y}_{i}-%
\boldsymbol{\mu }\left( \boldsymbol{\theta }\right) \right) \right) \right]
\\
&=&-a\frac{\tau }{2}\left \vert \boldsymbol{\Sigma }\left( \boldsymbol{%
\theta }\right) \right \vert ^{-\tau /2}trace\left( \boldsymbol{\Sigma }%
\left( \boldsymbol{\theta }\right) ^{-1}\dfrac{\partial \boldsymbol{\Sigma }%
\left( \boldsymbol{\theta }\right) }{\partial \theta _{i}}\right) \\
&&\left[ \frac{1}{n}\tsum \limits_{i=1}^{n}\exp \left \{ -\frac{\tau }{2}%
\left( \boldsymbol{y}_{i}-\boldsymbol{\mu }\left( \boldsymbol{\theta }%
\right) \right) ^{T}\boldsymbol{\Sigma }\left( \boldsymbol{\theta }\right)
^{-1}\left( \boldsymbol{y}_{i}-\boldsymbol{\mu }\left( \boldsymbol{\theta }%
\right) \right) \right \} -b\right] \\
&&+a\frac{\tau }{2}\left \vert \boldsymbol{\Sigma }\left( \boldsymbol{\theta 
}\right) \right \vert ^{-\tau /2}\frac{\tau }{2}\left[ \frac{1}{n}\tsum
\limits_{i=1}^{n}\exp \left \{ -\frac{\tau }{2}\left( \boldsymbol{y}_{i}-%
\boldsymbol{\mu }\left( \boldsymbol{\theta }\right) \right) ^{T}\boldsymbol{%
\Sigma }\left( \boldsymbol{\theta }\right) ^{-1}\left( \boldsymbol{y}_{i}-%
\boldsymbol{\mu }\left( \boldsymbol{\theta }\right) \right) \right \} \right.
\\
&&\left. \left( 2\left( \frac{\partial \boldsymbol{\mu }\left( \boldsymbol{%
\theta }\right) }{\partial \theta _{i}}\right) ^{T}\boldsymbol{\Sigma }%
\left( \boldsymbol{\theta }\right) ^{-1}\left( \boldsymbol{y}_{i}-%
\boldsymbol{\mu }\left( \boldsymbol{\theta }\right) \right) \right. \right.
\\
&&\left. \left. \left( \boldsymbol{y}_{i}-\boldsymbol{\mu }\left( 
\boldsymbol{\theta }\right) \right) ^{T}\left( \boldsymbol{\Sigma }\left( 
\boldsymbol{\theta }\right) ^{-1}\dfrac{\partial \boldsymbol{\Sigma }\left( 
\boldsymbol{\theta }\right) }{\partial \theta _{j}}\boldsymbol{\Sigma }%
\left( \boldsymbol{\theta }\right) ^{-1}\right) \left( \boldsymbol{y}_{i}-%
\boldsymbol{\mu }\left( \boldsymbol{\theta }\right) \right) \right) \right] .
\end{eqnarray*}

Therefore, 
\begin{equation*}
\frac{\partial ^{2}}{\partial \theta _{i}\partial \theta _{j}}H_{n}^{\tau }(%
\boldsymbol{\theta })=\frac{\partial }{\partial \theta _{j}}L_{1}^{\tau }(%
\boldsymbol{\theta })+\frac{\partial }{\partial \theta _{j}}L_{2}^{\tau }(%
\boldsymbol{\theta })
\end{equation*}%
being,

\begin{eqnarray*}
L_{1}^{\tau }(\boldsymbol{\theta }) &=&-a\frac{\tau }{2}\left \vert 
\boldsymbol{\Sigma }\left( \boldsymbol{\theta }\right) \right \vert ^{-\tau
/2}trace\left( \boldsymbol{\Sigma }\left( \boldsymbol{\theta }\right) ^{-1}%
\dfrac{\partial \boldsymbol{\Sigma }\left( \boldsymbol{\theta }\right) }{%
\partial \theta _{i}}\right) \\
&&\left[ \frac{1}{n}\tsum \limits_{i=1}^{n}\exp \left \{ -\frac{\tau }{2}%
\left( \boldsymbol{y}_{i}-\boldsymbol{\mu }\left( \boldsymbol{\theta }%
\right) \right) ^{T}\boldsymbol{\Sigma }\left( \boldsymbol{\theta }\right)
^{-1}\left( \boldsymbol{y}_{i}-\boldsymbol{\mu }\left( \boldsymbol{\theta }%
\right) \right) \right \} -b\right]
\end{eqnarray*}%
and 
\begin{eqnarray*}
L_{2}^{\tau }(\boldsymbol{\theta }) &=&a\frac{\tau }{2}\left \vert 
\boldsymbol{\Sigma }\left( \boldsymbol{\theta }\right) \right \vert ^{-\tau
/2}\frac{\tau }{2}\left[ \frac{1}{n}\tsum \limits_{i=1}^{n}\exp \left \{ -%
\frac{\tau }{2}\left( \boldsymbol{y}_{i}-\boldsymbol{\mu }\left( \boldsymbol{%
\theta }\right) \right) ^{T}\boldsymbol{\Sigma }\left( \boldsymbol{\theta }%
\right) ^{-1}\left( \boldsymbol{y}_{i}-\boldsymbol{\mu }\left( \boldsymbol{%
\theta }\right) \right) \right \} \right. \\
&&\left. \left( 2\left( \frac{\partial \boldsymbol{\mu }\left( \boldsymbol{%
\theta }\right) }{\partial \theta _{i}}\right) ^{T}\boldsymbol{\Sigma }%
\left( \boldsymbol{\theta }\right) ^{-1}\left( \boldsymbol{y}_{i}-%
\boldsymbol{\mu }\left( \boldsymbol{\theta }\right) \right) \right. \right.
\\
&&\left. \left. \left( \boldsymbol{y}_{i}-\boldsymbol{\mu }\left( 
\boldsymbol{\theta }\right) \right) ^{T}\left( \boldsymbol{\Sigma }\left( 
\boldsymbol{\theta }\right) ^{-1}\dfrac{\partial \boldsymbol{\Sigma }\left( 
\boldsymbol{\theta }\right) }{\partial \theta _{i}}\boldsymbol{\Sigma }%
\left( \boldsymbol{\theta }\right) ^{-1}\right) \left( \boldsymbol{y}_{i}-%
\boldsymbol{\mu }\left( \boldsymbol{\theta }\right) \right) \right) \right] .
\end{eqnarray*}%
We are going to get $\frac{\partial }{\partial \theta _{j}}L_{1}^{\tau }(%
\boldsymbol{\theta }).$%
\begin{eqnarray*}
\frac{\partial }{\partial \theta _{j}}L_{1}^{\tau }(\boldsymbol{\theta })
&=&-a\frac{\tau }{2}\left( -\frac{\tau }{2}\right) \left \vert \boldsymbol{%
\Sigma }\left( \boldsymbol{\theta }\right) \right \vert ^{-\tau
/2}trace\left( \boldsymbol{\Sigma }\left( \boldsymbol{\theta }\right) ^{-1}%
\dfrac{\partial \boldsymbol{\Sigma }\left( \boldsymbol{\theta }\right) }{%
\partial \theta _{j}}\right) trace\left( \boldsymbol{\Sigma }\left( 
\boldsymbol{\theta }\right) ^{-1}\dfrac{\partial \boldsymbol{\Sigma }\left( 
\boldsymbol{\theta }\right) }{\partial \theta _{i}}\right) \\
&&\left[ \frac{1}{n}\tsum \limits_{i=1}^{n}\exp \left \{ -\frac{\tau }{2}%
\left( \boldsymbol{y}_{i}-\boldsymbol{\mu }\left( \boldsymbol{\theta }%
\right) \right) ^{T}\boldsymbol{\Sigma }\left( \boldsymbol{\theta }\right)
^{-1}\left( \boldsymbol{y}_{i}-\boldsymbol{\mu }\left( \boldsymbol{\theta }%
\right) \right) \right \} -b\right] \\
&&-a\frac{\tau }{2}\left \vert \boldsymbol{\Sigma }\left( \boldsymbol{\theta 
}\right) \right \vert ^{-\tau /2}trace\left( -\boldsymbol{\Sigma }\left( 
\boldsymbol{\theta }\right) ^{-1}\dfrac{\partial \boldsymbol{\Sigma }\left( 
\boldsymbol{\theta }\right) }{\partial \theta _{j}}\boldsymbol{\Sigma }%
\left( \boldsymbol{\theta }\right) ^{-1}\dfrac{\partial \boldsymbol{\Sigma }%
\left( \boldsymbol{\theta }\right) }{\partial \theta _{i}}+\boldsymbol{%
\Sigma }\left( \boldsymbol{\theta }\right) ^{-1}\dfrac{\partial ^{2}%
\boldsymbol{\Sigma }\left( \boldsymbol{\theta }\right) }{\partial \theta
_{i}\partial \theta _{j}}\right) \\
&&\left[ \frac{1}{n}\tsum \limits_{i=1}^{n}\exp \left \{ -\frac{\tau }{2}%
\left( \boldsymbol{y}_{i}-\boldsymbol{\mu }\left( \boldsymbol{\theta }%
\right) \right) ^{T}\boldsymbol{\Sigma }\left( \boldsymbol{\theta }\right)
^{-1}\left( \boldsymbol{y}_{i}-\boldsymbol{\mu }\left( \boldsymbol{\theta }%
\right) \right) \right \} -b\right] \\
&&-a\frac{\tau }{2}\left \vert \boldsymbol{\Sigma }\left( \boldsymbol{\theta 
}\right) \right \vert ^{-\tau /2}trace\left( \boldsymbol{\Sigma }\left( 
\boldsymbol{\theta }\right) ^{-1}\dfrac{\partial \boldsymbol{\Sigma }\left( 
\boldsymbol{\theta }\right) }{\partial \theta _{i}}\right) \\
&&\left[ \frac{1}{n}\tsum \limits_{i=1}^{n}\exp \left \{ -\frac{\tau }{2}%
\left( \boldsymbol{y}_{i}-\boldsymbol{\mu }\left( \boldsymbol{\theta }%
\right) \right) ^{T}\boldsymbol{\Sigma }\left( \boldsymbol{\theta }\right)
^{-1}\left( \boldsymbol{y}_{i}-\boldsymbol{\mu }\left( \boldsymbol{\theta }%
\right) \right) \right \} \left( -\frac{\tau }{2}\right) \right. \\
&&\left. \left( -2\left( \frac{\partial \boldsymbol{\mu }\left( \boldsymbol{%
\theta }\right) }{\partial \theta _{i}}\right) ^{T}\boldsymbol{\Sigma }%
\left( \boldsymbol{\theta }\right) ^{-1}\left( \boldsymbol{y}_{i}-%
\boldsymbol{\mu }\left( \boldsymbol{\theta }\right) \right) \right. \right.
\\
&&\left. \left. -\left( \boldsymbol{y}_{i}-\boldsymbol{\mu }\left( 
\boldsymbol{\theta }\right) \right) ^{T}\left( \boldsymbol{\Sigma }\left( 
\boldsymbol{\theta }\right) ^{-1}\dfrac{\partial \boldsymbol{\Sigma }\left( 
\boldsymbol{\theta }\right) }{\partial \theta _{j}}\boldsymbol{\Sigma }%
\left( \boldsymbol{\theta }\right) ^{-1}\right) \left( \boldsymbol{y}_{i}-%
\boldsymbol{\mu }\left( \boldsymbol{\theta }\right) \right) \right) \right]
\\
&=&D_{1}+D_{2}+D_{3},
\end{eqnarray*}%
being 
\begin{eqnarray*}
D_{1} &=&a\frac{\tau }{4}^{2}\left \vert \boldsymbol{\Sigma }\left( 
\boldsymbol{\theta }\right) \right \vert ^{-\tau /2}trace\left( \boldsymbol{%
\Sigma }\left( \boldsymbol{\theta }\right) ^{-1}\dfrac{\partial \boldsymbol{%
\Sigma }\left( \boldsymbol{\theta }\right) }{\partial \theta _{i}}\right)
trace\left( \boldsymbol{\Sigma }\left( \boldsymbol{\theta }\right) ^{-1}%
\dfrac{\partial \boldsymbol{\Sigma }\left( \boldsymbol{\theta }\right) }{%
\partial \theta _{j}}\right) \\
&&\left[ \frac{1}{n}\tsum \limits_{i=1}^{n}\exp \left \{ -\frac{\tau }{2}%
\left( \boldsymbol{y}_{i}-\boldsymbol{\mu }\left( \boldsymbol{\theta }%
\right) \right) ^{T}\boldsymbol{\Sigma }\left( \boldsymbol{\theta }\right)
^{-1}\left( \boldsymbol{y}_{i}-\boldsymbol{\mu }\left( \boldsymbol{\theta }%
\right) \right) \right \} -b\right]
\end{eqnarray*}%
\begin{eqnarray*}
D_{2} &=&-a\frac{\tau }{2}\left \vert \boldsymbol{\Sigma }\left( \boldsymbol{%
\theta }\right) \right \vert ^{-\tau /2}trace\left( -\boldsymbol{\Sigma }%
\left( \boldsymbol{\theta }\right) ^{-1}\dfrac{\partial \boldsymbol{\Sigma }%
\left( \boldsymbol{\theta }\right) }{\partial \theta _{j}}\boldsymbol{\Sigma 
}\left( \boldsymbol{\theta }\right) ^{-1}\dfrac{\partial \boldsymbol{\Sigma }%
\left( \boldsymbol{\theta }\right) }{\partial \theta _{i}}\right. \\
&&+\left. \boldsymbol{\Sigma }\left( \boldsymbol{\theta }\right) ^{-1}\dfrac{%
\partial ^{2}\boldsymbol{\Sigma }\left( \boldsymbol{\theta }\right) }{%
\partial \theta _{i}\partial \theta _{j}}\right) \\
&&\left[ \frac{1}{n}\tsum \limits_{i=1}^{n}\exp \left \{ -\frac{\tau }{2}%
\left( \boldsymbol{y}_{i}-\boldsymbol{\mu }\left( \boldsymbol{\theta }%
\right) \right) ^{T}\boldsymbol{\Sigma }\left( \boldsymbol{\theta }\right)
^{-1}\left( \boldsymbol{y}_{i}-\boldsymbol{\mu }\left( \boldsymbol{\theta }%
\right) \right) \right \} -b\right]
\end{eqnarray*}%
and 
\begin{eqnarray*}
D_{3} &=&a\frac{\tau ^{2}}{4}\left \vert \boldsymbol{\Sigma }\left( 
\boldsymbol{\theta }\right) \right \vert ^{-\tau /2}trace\left( \boldsymbol{%
\Sigma }\left( \boldsymbol{\theta }\right) ^{-1}\dfrac{\partial \boldsymbol{%
\Sigma }\left( \boldsymbol{\theta }\right) }{\partial \theta _{i}}\right) \\
&&\left[ \frac{1}{n}\tsum \limits_{i=1}^{n}\exp \left \{ -\frac{\tau }{2}%
\left( \boldsymbol{y}_{i}-\boldsymbol{\mu }\left( \boldsymbol{\theta }%
\right) \right) ^{T}\boldsymbol{\Sigma }\left( \boldsymbol{\theta }\right)
^{-1}\left( \boldsymbol{y}_{i}-\boldsymbol{\mu }\left( \boldsymbol{\theta }%
\right) \right) \right \} \right. \\
&&\left. \left( -2\left( \frac{\partial \boldsymbol{\mu }\left( \boldsymbol{%
\theta }\right) }{\partial \theta _{i}}\right) ^{T}\boldsymbol{\Sigma }%
\left( \boldsymbol{\theta }\right) ^{-1}\left( \boldsymbol{y}_{i}-%
\boldsymbol{\mu }\left( \boldsymbol{\theta }\right) \right) \right. \right.
\\
&&\left. \left. -\left( \boldsymbol{y}_{i}-\boldsymbol{\mu }\left( 
\boldsymbol{\theta }\right) \right) ^{T}\left( \boldsymbol{\Sigma }\left( 
\boldsymbol{\theta }\right) ^{-1}\dfrac{\partial \boldsymbol{\Sigma }\left( 
\boldsymbol{\theta }\right) }{\partial \theta _{j}}\boldsymbol{\Sigma }%
\left( \boldsymbol{\theta }\right) ^{-1}\right) \left( \boldsymbol{y}_{i}-%
\boldsymbol{\mu }\left( \boldsymbol{\theta }\right) \right) \right) \right]
\end{eqnarray*}%
Now we are going to see some result that will be important in order to get
the convergence in probability of $D_{1},D_{2}$ and $D_{3}.$

\begin{remark}
We have,%
\begin{equation*}
l=\frac{1}{n}\tsum \limits_{i=1}^{n}\exp \left \{ -\frac{\tau }{2}\left( 
\boldsymbol{Y}_{i}-\boldsymbol{\mu }\left( \boldsymbol{\theta }\right)
\right) ^{T}\boldsymbol{\Sigma }\left( \boldsymbol{\theta }\right)
^{-1}\left( \boldsymbol{Y}_{i}-\boldsymbol{\mu }\left( \boldsymbol{\theta }%
\right) \right) \right \} \underset{n\longrightarrow \infty }{\overset{%
\mathcal{P}}{\longrightarrow }}\frac{1}{\left( 1+\tau \right) ^{m/2}}
\end{equation*}
\end{remark}

\begin{proof}
It is clear that 
\begin{eqnarray*}
&&l\underset{n\longrightarrow \infty }{\overset{\mathcal{P}}{\longrightarrow 
}}E_{\mathcal{N}\left( \boldsymbol{\mu }\left( \boldsymbol{\theta }\right) ,%
\boldsymbol{\Sigma }\left( \boldsymbol{\theta }\right) \right) }\left[ \exp
\left \{ -\frac{\tau }{2}\left( \boldsymbol{Y}-\boldsymbol{\mu }\left( 
\boldsymbol{\theta }\right) \right) ^{T}\left( \frac{\boldsymbol{\Sigma }%
\left( \boldsymbol{\theta }\right) }{\tau }\right) \left( \boldsymbol{Y}-%
\boldsymbol{\mu }\left( \boldsymbol{\theta }\right) \right) \right \} \right]
\\
&=&\dint \exp \left \{ -\frac{\tau }{2}\left( \boldsymbol{Y}-\boldsymbol{\mu 
}\left( \boldsymbol{\theta }\right) \right) ^{T}\left( \frac{\boldsymbol{%
\Sigma }\left( \boldsymbol{\theta }\right) }{\tau }\right) \left( 
\boldsymbol{Y}-\boldsymbol{\mu }\left( \boldsymbol{\theta }\right) \right)
\right \} f_{\mathcal{N}\left( \boldsymbol{\mu }\left( \boldsymbol{\theta }%
\right) ,\boldsymbol{\Sigma }\left( \boldsymbol{\theta }\right) \right)
}\left( \boldsymbol{y}\right) d\boldsymbol{y} \\
&=&\frac{1}{\left( 1+\tau \right) ^{m/2}}\dint \frac{1}{\left( 2\pi \right)
^{m/2}}\frac{1}{\left \vert \frac{\boldsymbol{\Sigma }\left( \boldsymbol{%
\theta }\right) }{\tau +1}\right \vert ^{1/2}}\exp \left \{ -\frac{\tau +1}{2%
}\left( \boldsymbol{y}-\boldsymbol{\mu }\left( \boldsymbol{\theta }\right)
\right) ^{T}\left( \frac{\boldsymbol{\Sigma }\left( \boldsymbol{\theta }%
\right) }{\tau }\right) \left( \boldsymbol{y}-\boldsymbol{\mu }\left( 
\boldsymbol{\theta }\right) \right) \right \} d\boldsymbol{y} \\
&=&\frac{1}{\left( 1+\tau \right) ^{m/2}}.
\end{eqnarray*}
\end{proof}

\begin{remark}
We have,%
\begin{equation*}
m=\frac{1}{n}\tsum \limits_{i=1}^{n}\exp \left \{ -\frac{\tau }{2}\left( 
\boldsymbol{Y}_{i}-\boldsymbol{\mu }\left( \boldsymbol{\theta }\right)
\right) ^{T}\boldsymbol{\Sigma }\left( \boldsymbol{\theta }\right)
^{-1}\left( \boldsymbol{Y}_{i}-\boldsymbol{\mu }\left( \boldsymbol{\theta }%
\right) \right) \right \} -b\underset{n\longrightarrow \infty }{\overset{%
\mathcal{P}}{\longrightarrow }}\frac{1}{\left( 1+\tau \right) ^{\frac{m}{2}%
+1}}
\end{equation*}
\end{remark}

\begin{proof}
By result the previous remark 
\begin{equation*}
m\underset{n\longrightarrow \infty }{\overset{\mathcal{P}}{\longrightarrow }}%
\frac{1}{\left( 1+\tau \right) ^{m/2}}-\frac{\tau }{\left( 1+\tau \right) ^{%
\frac{m}{2}+1}}=\frac{1}{\left( 1+\tau \right) ^{\frac{m}{2}+1}}.
\end{equation*}
\end{proof}

\begin{remark}
We denote,%
\begin{equation*}
n=\frac{1}{n}\tsum \limits_{i=1}^{n}\exp \left \{ -\frac{\tau }{2}\left( 
\boldsymbol{Y}_{i}-\boldsymbol{\mu }\left( \boldsymbol{\theta }\right)
\right) ^{T}\boldsymbol{\Sigma }\left( \boldsymbol{\theta }\right)
^{-1}\left( \boldsymbol{Y}_{i}-\boldsymbol{\mu }\left( \boldsymbol{\theta }%
\right) \right) \right \} \left( \boldsymbol{Y}_{i}-\boldsymbol{\mu }\left( 
\boldsymbol{\theta }\right) \right) ^{T}\boldsymbol{A}\left( \boldsymbol{Y}%
_{i}-\boldsymbol{\mu }\left( \boldsymbol{\theta }\right) \right)
\end{equation*}%
and we have,%
\begin{equation*}
n\underset{n\longrightarrow \infty }{\overset{\mathcal{P}}{\longrightarrow }}%
\frac{trace\left( \boldsymbol{A\Sigma }\left( \boldsymbol{\theta }\right)
\right) }{\left( 1+\tau \right) ^{\frac{m}{2}+1}}.
\end{equation*}
\end{remark}

\begin{proof}
It is clear that,

\begin{eqnarray*}
&&n\underset{n\longrightarrow \infty }{\overset{\mathcal{P}}{\longrightarrow 
}}E_{\mathcal{N}\left( \boldsymbol{\mu }\left( \boldsymbol{\theta }\right) ,%
\boldsymbol{\Sigma }\left( \boldsymbol{\theta }\right) \right) }\left[ \exp
\left \{ -\frac{1}{2}\left( \boldsymbol{Y}-\boldsymbol{\mu }\left( 
\boldsymbol{\theta }\right) \right) ^{T}\left( \frac{\boldsymbol{\Sigma }%
\left( \boldsymbol{\theta }\right) }{\tau }\right) ^{-1}\left( \boldsymbol{Y}%
-\boldsymbol{\mu }\left( \boldsymbol{\theta }\right) \right) \right \}
\right. \\
&&\left. \left. \left( \boldsymbol{Y}-\boldsymbol{\mu }\left( \boldsymbol{%
\theta }\right) \right) ^{T}\boldsymbol{A}\left( \boldsymbol{Y}-\boldsymbol{%
\mu }\left( \boldsymbol{\theta }\right) \right) \right. \right] \\
&=&\frac{1}{\left( 1+\tau \right) ^{m/2}}\dint \frac{1}{\left( 2\pi \right)
^{m/2}}\frac{1}{\left \vert \frac{\boldsymbol{\Sigma }\left( \boldsymbol{%
\theta }\right) }{\tau +1}\right \vert ^{1/2}}\exp \left \{ -\frac{1}{2}%
\left( \boldsymbol{y}-\boldsymbol{\mu }\left( \boldsymbol{\theta }\right)
\right) ^{T}\left( \frac{\boldsymbol{\Sigma }\left( \boldsymbol{\theta }%
\right) }{\tau +1}\right) ^{-1}\left( \boldsymbol{y}-\boldsymbol{\mu }\left( 
\boldsymbol{\theta }\right) \right) \right \} \\
&&\left( \boldsymbol{y}-\boldsymbol{\mu }\left( \boldsymbol{\theta }\right)
\right) ^{T}\boldsymbol{A}\left( \boldsymbol{y}-\boldsymbol{\mu }\left( 
\boldsymbol{\theta }\right) \right) d\boldsymbol{y} \\
&=&\frac{1}{\left( 1+\tau \right) ^{m/2}}E_{\mathcal{N}\left( \boldsymbol{%
\mu }\left( \boldsymbol{\theta }\right) ,\frac{\boldsymbol{\Sigma }\left( 
\boldsymbol{\theta }\right) }{1+\tau }\right) }\left[ \left( \boldsymbol{Y}-%
\boldsymbol{\mu }\left( \boldsymbol{\theta }\right) \right) ^{T}\boldsymbol{A%
}\left( \boldsymbol{Y}-\boldsymbol{\mu }\left( \boldsymbol{\theta }\right)
\right) \right] \\
&=&\frac{1}{\left( 1+\tau \right) ^{\frac{m}{2}+1}}trace\left( \boldsymbol{%
A\Sigma }\left( \boldsymbol{\theta }\right) \right) .
\end{eqnarray*}
\end{proof}

Based on the previous results we have in relation to $D_{1,}$ 
\begin{equation*}
D_{1}\underset{n\longrightarrow \infty }{\overset{\mathcal{P}}{%
\longrightarrow }}\frac{\tau }{4}\frac{\left \vert \boldsymbol{\Sigma }%
\left( \boldsymbol{\theta }\right) \right \vert ^{-\frac{\tau }{2}}}{\left(
2\pi \right) ^{m\tau /2}\left( 1+\tau \right) ^{m/2}}trace\left( \boldsymbol{%
\Sigma }\left( \boldsymbol{\theta }\right) ^{-1}\dfrac{\partial \boldsymbol{%
\Sigma }\left( \boldsymbol{\theta }\right) }{\partial \theta _{i}}\right)
trace\left( \boldsymbol{\Sigma }\left( \boldsymbol{\theta }\right) ^{-1}%
\dfrac{\partial \boldsymbol{\Sigma }\left( \boldsymbol{\theta }\right) }{%
\partial \theta _{j}}\right) .
\end{equation*}%
With respect to $D_{2,}$ 
\begin{eqnarray*}
&&D_{2}\underset{n\longrightarrow \infty }{\overset{\mathcal{P}}{%
\longrightarrow }}\frac{1}{\left( 2\pi \right) ^{m/2}}\left \vert 
\boldsymbol{\Sigma }\left( \boldsymbol{\theta }\right) \right \vert ^{-\frac{%
\tau }{2}}\frac{1}{2}\frac{1}{\left( 1+\tau \right) ^{m/2}}trace\left( 
\boldsymbol{\Sigma }\left( \boldsymbol{\theta }\right) ^{-1}\dfrac{\partial 
\boldsymbol{\Sigma }\left( \boldsymbol{\theta }\right) }{\partial \theta _{j}%
}\boldsymbol{\Sigma }\left( \boldsymbol{\theta }\right) ^{-1}\dfrac{\partial 
\boldsymbol{\Sigma }\left( \boldsymbol{\theta }\right) }{\partial \theta _{i}%
}\right) \\
&&-\frac{\left \vert \boldsymbol{\Sigma }\left( \boldsymbol{\theta }\right)
\right \vert ^{-\frac{\tau }{2}}}{\left( 2\pi \right) ^{m/2}}\frac{1}{2}%
trace\left( \boldsymbol{\Sigma }\left( \boldsymbol{\theta }\right) ^{-1}%
\dfrac{\partial ^{2}\boldsymbol{\Sigma }\left( \boldsymbol{\theta }\right) }{%
\partial \theta _{j}\partial \theta _{i}}\right) .
\end{eqnarray*}%
In a similar way we get for $D_{3}$ that,%
\begin{equation*}
D_{3}\underset{n\longrightarrow \infty }{\overset{\mathcal{P}}{%
\longrightarrow }}-\frac{\tau }{4}\frac{\left \vert \boldsymbol{\Sigma }%
\left( \boldsymbol{\theta }\right) \right \vert ^{-\frac{\tau }{2}}}{\left(
2\pi \right) ^{m\tau /2}}trace\left( \boldsymbol{\Sigma }\left( \boldsymbol{%
\theta }\right) ^{-1}\dfrac{\partial \boldsymbol{\Sigma }\left( \boldsymbol{%
\theta }\right) }{\partial \theta _{i}}\right) \frac{1}{\left( 1+\tau
\right) ^{m/2}}trace\left( \boldsymbol{\Sigma }\left( \boldsymbol{\theta }%
\right) ^{-1}\dfrac{\partial \boldsymbol{\Sigma }\left( \boldsymbol{\theta }%
\right) }{\partial \theta _{j}}\right) .
\end{equation*}%
Therefore we have%
\begin{eqnarray*}
&&\frac{\partial }{\partial \theta _{j}}L_{1}^{\tau }(\boldsymbol{\theta })%
\underset{n\longrightarrow \infty }{\overset{\mathcal{P}}{\longrightarrow }}%
\frac{1}{2}\frac{\left \vert \boldsymbol{\Sigma }\left( \boldsymbol{\theta }%
\right) \right \vert ^{-\frac{\tau }{2}}}{\left( 2\pi \right) ^{m\tau /2}}%
\frac{1}{\left( 1+\tau \right) ^{m/2}}trace\left( \boldsymbol{\Sigma }\left( 
\boldsymbol{\theta }\right) ^{-1}\dfrac{\partial \boldsymbol{\Sigma }\left( 
\boldsymbol{\theta }\right) }{\partial \theta _{j}}\boldsymbol{\Sigma }%
\left( \boldsymbol{\theta }\right) ^{-1}\dfrac{\partial \boldsymbol{\Sigma }%
\left( \boldsymbol{\theta }\right) }{\partial \theta _{i}}\right) \\
&&-\frac{1}{2}\frac{\left \vert \boldsymbol{\Sigma }\left( \boldsymbol{%
\theta }\right) \right \vert ^{-\frac{\tau }{2}}}{\left( 2\pi \right)
^{m\tau /2}}\frac{1}{\left( 1+\tau \right) ^{m/2}}\left( \boldsymbol{\Sigma }%
\left( \boldsymbol{\theta }\right) ^{-1}\dfrac{\partial ^{2}\boldsymbol{%
\Sigma }\left( \boldsymbol{\theta }\right) }{\partial \theta _{j}\partial
\theta _{i}}\right) .
\end{eqnarray*}

Now we have, 
\begin{eqnarray*}
\frac{\partial }{\partial \theta _{j}}L_{2}^{\tau }(\boldsymbol{\theta })
&=&-a\frac{\tau ^{2}}{4}\left \vert \boldsymbol{\Sigma }\left( \boldsymbol{%
\theta }\right) \right \vert ^{-\frac{\tau }{2}}trace\left( \boldsymbol{%
\Sigma }\left( \boldsymbol{\theta }\right) ^{-1}\dfrac{\partial \boldsymbol{%
\Sigma }\left( \boldsymbol{\theta }\right) }{\partial \theta _{j}}\right) \\
&&\left[ \frac{1}{n}\tsum \limits_{i=1}^{n}\exp \left \{ -\frac{\tau }{2}%
\left( \boldsymbol{y}_{i}-\boldsymbol{\mu }\left( \boldsymbol{\theta }%
\right) \right) ^{T}\boldsymbol{\Sigma }\left( \boldsymbol{\theta }\right)
^{-1}\left( \boldsymbol{y}_{i}-\boldsymbol{\mu }\left( \boldsymbol{\theta }%
\right) \right) \right \} \left( 2\left( \frac{\partial \boldsymbol{\mu }%
\left( \boldsymbol{\theta }\right) }{\partial \theta _{i}}\right)
^{T}\right. \right. \\
&&\left. \left. \boldsymbol{\Sigma }\left( \boldsymbol{\theta }\right)
^{-1}\left( \boldsymbol{y}_{i}-\boldsymbol{\mu }\left( \boldsymbol{\theta }%
\right) \right) +\left( \boldsymbol{y}_{i}-\boldsymbol{\mu }\left( 
\boldsymbol{\theta }\right) \right) ^{T}\left( \boldsymbol{\Sigma }\left( 
\boldsymbol{\theta }\right) ^{-1}\dfrac{\partial \boldsymbol{\Sigma }\left( 
\boldsymbol{\theta }\right) }{\partial \theta _{j}}\boldsymbol{\Sigma }%
\left( \boldsymbol{\theta }\right) ^{-1}\right) \left( \boldsymbol{y}_{i}-%
\boldsymbol{\mu }\left( \boldsymbol{\theta }\right) \right) \right) \right]
\\
&&+a\frac{\tau }{2}\left \vert \boldsymbol{\Sigma }\left( \boldsymbol{\theta 
}\right) \right \vert ^{-\frac{\tau }{2}}\left[ \frac{1}{n}\tsum
\limits_{i=1}^{n}\exp \left \{ -\frac{\tau }{2}\left( \boldsymbol{y}_{i}-%
\boldsymbol{\mu }\left( \boldsymbol{\theta }\right) \right) ^{T}\boldsymbol{%
\Sigma }\left( \boldsymbol{\theta }\right) ^{-1}\left( \boldsymbol{y}_{i}-%
\boldsymbol{\mu }\left( \boldsymbol{\theta }\right) \right) \right \} \right.
\\
&&\left( -\frac{\tau }{2}\right) \left( \frac{\partial }{\partial \theta _{j}%
}\left( \boldsymbol{y}_{i}-\boldsymbol{\mu }\left( \boldsymbol{\theta }%
\right) \right) ^{T}\boldsymbol{\Sigma }\left( \boldsymbol{\theta }\right)
^{-1}\left( \boldsymbol{y}_{i}-\boldsymbol{\mu }\left( \boldsymbol{\theta }%
\right) \right) \right) \left( 2\left( \frac{\partial \boldsymbol{\mu }%
\left( \boldsymbol{\theta }\right) }{\partial \theta _{i}}\right) ^{T}%
\boldsymbol{\Sigma }\left( \boldsymbol{\theta }\right) ^{-1}\left( 
\boldsymbol{y}_{i}-\boldsymbol{\mu }\left( \boldsymbol{\theta }\right)
\right) \right. \\
&&\left. \left. \left( \boldsymbol{y}_{i}-\boldsymbol{\mu }\left( 
\boldsymbol{\theta }\right) \right) ^{T}\left( \boldsymbol{\Sigma }\left( 
\boldsymbol{\theta }\right) ^{-1}\dfrac{\partial \boldsymbol{\Sigma }\left( 
\boldsymbol{\theta }\right) }{\partial \theta _{j}}\boldsymbol{\Sigma }%
\left( \boldsymbol{\theta }\right) ^{-1}\right) \left( \boldsymbol{y}_{i}-%
\boldsymbol{\mu }\left( \boldsymbol{\theta }\right) \right) \right) \right]
\\
&&+a\frac{\tau }{2}\left \vert \boldsymbol{\Sigma }\left( \boldsymbol{\theta 
}\right) \right \vert ^{-\frac{\tau }{2}}\left[ \frac{1}{n}\tsum
\limits_{i=1}^{n}\exp \left \{ -\frac{\tau }{2}\left( \boldsymbol{y}_{i}-%
\boldsymbol{\mu }\left( \boldsymbol{\theta }\right) \right) ^{T}\boldsymbol{%
\Sigma }\left( \boldsymbol{\theta }\right) ^{-1}\left( \boldsymbol{y}_{i}-%
\boldsymbol{\mu }\left( \boldsymbol{\theta }\right) \right) \right \} \right.
\\
&&\left. \frac{\partial }{\partial \theta _{j}}\left( 2\left( \frac{\partial 
\boldsymbol{\mu }\left( \boldsymbol{\theta }\right) }{\partial \theta _{i}}%
\right) ^{T}\boldsymbol{\Sigma }\left( \boldsymbol{\theta }\right)
^{-1}\left( \boldsymbol{y}_{i}-\boldsymbol{\mu }\left( \boldsymbol{\theta }%
\right) \right) \right) \right. \\
&&\left. \left. +\left( \boldsymbol{y}_{i}-\boldsymbol{\mu }\left( 
\boldsymbol{\theta }\right) \right) ^{T}\left( \boldsymbol{\Sigma }\left( 
\boldsymbol{\theta }\right) ^{-1}\dfrac{\partial \boldsymbol{\Sigma }\left( 
\boldsymbol{\theta }\right) }{\partial \theta _{j}}\boldsymbol{\Sigma }%
\left( \boldsymbol{\theta }\right) ^{-1}\right) \left( \boldsymbol{y}_{i}-%
\boldsymbol{\mu }\left( \boldsymbol{\theta }\right) \right) \right) \right]
\\
&=&C_{1}+C_{2}+C_{3}.
\end{eqnarray*}%
Being,

\begin{eqnarray*}
C_{1} &=&-a\frac{\tau ^{2}}{4}\left \vert \boldsymbol{\Sigma }\left( 
\boldsymbol{\theta }\right) \right \vert ^{-\frac{\tau }{2}}trace\left( 
\boldsymbol{\Sigma }\left( \boldsymbol{\theta }\right) ^{-1}\dfrac{\partial 
\boldsymbol{\Sigma }\left( \boldsymbol{\theta }\right) }{\partial \theta _{j}%
}\right) \left[ \frac{1}{n}\tsum \limits_{i=1}^{n}\exp \left \{ -\frac{\tau 
}{2}\left( \boldsymbol{y}_{i}-\boldsymbol{\mu }\left( \boldsymbol{\theta }%
\right) \right) ^{T}\boldsymbol{\Sigma }\left( \boldsymbol{\theta }\right)
^{-1}\left( \boldsymbol{y}_{i}-\boldsymbol{\mu }\left( \boldsymbol{\theta }%
\right) \right) \right \} \right. \\
&&\left. \left( 2\left( \frac{\partial \boldsymbol{\mu }\left( \boldsymbol{%
\theta }\right) }{\partial \theta _{i}}\right) ^{T}\boldsymbol{\Sigma }%
\left( \boldsymbol{\theta }\right) ^{-1}\left( \boldsymbol{y}_{i}-%
\boldsymbol{\mu }\left( \boldsymbol{\theta }\right) \right) \right. \right.
\\
&&\left. \left. \left( \boldsymbol{y}_{i}-\boldsymbol{\mu }\left( 
\boldsymbol{\theta }\right) \right) ^{T}\left( \boldsymbol{\Sigma }\left( 
\boldsymbol{\theta }\right) ^{-1}\dfrac{\partial \boldsymbol{\Sigma }\left( 
\boldsymbol{\theta }\right) }{\partial \theta _{j}}\boldsymbol{\Sigma }%
\left( \boldsymbol{\theta }\right) ^{-1}\right) \left( \boldsymbol{y}_{i}-%
\boldsymbol{\mu }\left( \boldsymbol{\theta }\right) \right) \right) \right] .
\end{eqnarray*}

It is clear that,%
\begin{equation*}
C_{1}\underset{n\longrightarrow \infty }{\overset{\mathcal{P}}{%
\longrightarrow }}-\frac{\tau }{4}\frac{1}{\left( 1+\tau \right) ^{m/2}}%
\frac{\left \vert \boldsymbol{\Sigma }\left( \boldsymbol{\theta }\right)
\right \vert ^{-\frac{\tau }{2}}}{\left( 2\pi \right) ^{m\tau /2}}%
trace\left( \boldsymbol{\Sigma }\left( \boldsymbol{\theta }\right) ^{-1}%
\dfrac{\partial \boldsymbol{\Sigma }\left( \boldsymbol{\theta }\right) }{%
\partial \theta _{i}}\right) trace\left( \boldsymbol{\Sigma }\left( 
\boldsymbol{\theta }\right) ^{-1}\dfrac{\partial \boldsymbol{\Sigma }\left( 
\boldsymbol{\theta }\right) }{\partial \theta _{j}}\right) .
\end{equation*}%
It is immediate to see that 
\begin{eqnarray*}
C_{2} &=&-a\frac{\tau ^{2}}{4}\left \vert \boldsymbol{\Sigma }\left( 
\boldsymbol{\theta }\right) \right \vert ^{-\frac{\tau }{2}}\frac{1}{n}\tsum
\limits_{i=1}^{n}\exp \left \{ -\frac{\tau }{2}\left( \boldsymbol{y}_{i}-%
\boldsymbol{\mu }\left( \boldsymbol{\theta }\right) \right) ^{T}\boldsymbol{%
\Sigma }\left( \boldsymbol{\theta }\right) ^{-1}\left( \boldsymbol{y}_{i}-%
\boldsymbol{\mu }\left( \boldsymbol{\theta }\right) \right) \right \} \left(
S_{1}+S_{2}+S_{3}+S_{4}\right) \\
&=&L_{1}^{\ast }+L_{2}^{\ast }+L_{3}^{\ast }+L_{4}^{\ast }
\end{eqnarray*}%
where

\begin{eqnarray*}
L_{1}^{\ast } &=&-a\frac{\tau ^{2}}{4}\left \vert \boldsymbol{\Sigma }\left( 
\boldsymbol{\theta }\right) \right \vert ^{-\frac{\tau }{2}}\frac{1}{n}\tsum
\limits_{i=1}^{n}\exp \left \{ -\frac{\tau }{2}\left( \boldsymbol{y}_{i}-%
\boldsymbol{\mu }\left( \boldsymbol{\theta }\right) \right) ^{T}\boldsymbol{%
\Sigma }\left( \boldsymbol{\theta }\right) ^{-1}\left( \boldsymbol{y}_{i}-%
\boldsymbol{\mu }\left( \boldsymbol{\theta }\right) \right) \right \} \\
&&\left( -4\left( \boldsymbol{y}_{i}-\boldsymbol{\mu }\left( \boldsymbol{%
\theta }\right) \right) ^{T}\boldsymbol{\Sigma }\left( \boldsymbol{\theta }%
\right) ^{-1}\frac{\partial \boldsymbol{\mu }\left( \boldsymbol{\theta }%
\right) }{\partial \theta _{j}}\left( \frac{\partial \boldsymbol{\mu }\left( 
\boldsymbol{\theta }\right) }{\partial \theta _{i}}\right) ^{T}\boldsymbol{%
\Sigma }\left( \boldsymbol{\theta }\right) ^{-1}\left( \boldsymbol{y}_{i}-%
\boldsymbol{\mu }\left( \boldsymbol{\theta }\right) \right) \right)
\end{eqnarray*}%
and 
\begin{equation*}
L_{1}^{\ast }\underset{n\longrightarrow \infty }{\overset{\mathcal{P}}{%
\longrightarrow }}\frac{\left \vert \boldsymbol{\Sigma }\left( \boldsymbol{%
\theta }\right) \right \vert ^{-\frac{\tau }{2}}}{\left( 2\pi \right)
^{m\tau /2}}\frac{\tau }{\left( 1+\tau \right) ^{m/2}}trace\left( 
\boldsymbol{\Sigma }\left( \boldsymbol{\theta }\right) ^{-1}\frac{\partial 
\boldsymbol{\mu }\left( \boldsymbol{\theta }\right) }{\partial \theta _{j}}%
\left( \frac{\partial \boldsymbol{\mu }\left( \boldsymbol{\theta }\right) }{%
\partial \theta _{i}}\right) ^{T}\right) .
\end{equation*}%
It is clear that%
\begin{equation*}
L_{2}^{\ast }\underset{n\longrightarrow \infty }{\overset{\mathcal{P}}{%
\longrightarrow }}\boldsymbol{0}\text{ and }L_{3}^{\ast }\underset{%
n\longrightarrow \infty }{\overset{\mathcal{P}}{\longrightarrow }}%
\boldsymbol{0.}\text{ }
\end{equation*}%
On the other hand,%
\begin{eqnarray*}
&&L_{4}^{\ast }\underset{n\longrightarrow \infty }{\overset{\mathcal{P}}{%
\longrightarrow }}\frac{\tau +1}{\left( 2\pi \right) ^{m\tau /2}}\frac{\tau 
}{4}\left \vert \boldsymbol{\Sigma }\left( \boldsymbol{\theta }\right)
\right \vert ^{-\frac{\tau }{2}}\frac{1}{\left( 1+\tau \right) ^{m/2}} \\
&&\left \{ trace\left \{ \left( \boldsymbol{\Sigma }\left( \boldsymbol{%
\theta }\right) ^{-1}\dfrac{\partial \boldsymbol{\Sigma }\left( \boldsymbol{%
\theta }\right) }{\partial \theta _{j}}\boldsymbol{\Sigma }\left( 
\boldsymbol{\theta }\right) ^{-1}\dfrac{\partial \boldsymbol{\Sigma }\left( 
\boldsymbol{\theta }\right) }{\partial \theta _{i}}\right) \left[ 
\boldsymbol{\Sigma }\left( \boldsymbol{\theta }\right) ^{-1}\dfrac{\partial 
\boldsymbol{\Sigma }\left( \boldsymbol{\theta }\right) }{\partial \theta _{i}%
}\boldsymbol{\Sigma }\left( \boldsymbol{\theta }\right) ^{-1}+\boldsymbol{%
\Sigma }\left( \boldsymbol{\theta }\right) ^{-1}\dfrac{\partial \boldsymbol{%
\Sigma }\left( \boldsymbol{\theta }\right) }{\partial \theta _{i}}%
\boldsymbol{\Sigma }\left( \boldsymbol{\theta }\right) ^{-1}\right] \right.
\right. \\
&&\left. \left. \frac{\boldsymbol{\Sigma }\left( \boldsymbol{\theta }\right) 
}{1+\tau }\right \} +trace\left( \boldsymbol{\Sigma }\left( \boldsymbol{%
\theta }\right) ^{-1}\dfrac{\partial \boldsymbol{\Sigma }\left( \boldsymbol{%
\theta }\right) }{\partial \theta _{i}}\boldsymbol{\Sigma }\left( 
\boldsymbol{\theta }\right) ^{-1}\frac{\boldsymbol{\Sigma }\left( 
\boldsymbol{\theta }\right) }{1+\tau }\right) trace\left( \boldsymbol{\Sigma 
}\left( \boldsymbol{\theta }\right) ^{-1}\dfrac{\partial \boldsymbol{\Sigma }%
\left( \boldsymbol{\theta }\right) }{\partial \theta _{i}}\boldsymbol{\Sigma 
}\left( \boldsymbol{\theta }\right) ^{-1}\frac{\boldsymbol{\Sigma }\left( 
\boldsymbol{\theta }\right) }{1+\tau }\right) \right \} \\
&&=2\frac{\tau }{4}\frac{\left \vert \boldsymbol{\Sigma }\left( \boldsymbol{%
\theta }\right) \right \vert ^{-\frac{\tau }{2}}}{\left( 2\pi \right)
^{m\tau /2}}\frac{1}{\left( 1+\tau \right) ^{m/2+1}}trace\left( \boldsymbol{%
\Sigma }\left( \boldsymbol{\theta }\right) ^{-1}\dfrac{\partial \boldsymbol{%
\Sigma }\left( \boldsymbol{\theta }\right) }{\partial \theta _{i}}%
\boldsymbol{\Sigma }\left( \boldsymbol{\theta }\right) ^{-1}\boldsymbol{%
\Sigma }\left( \boldsymbol{\theta }\right) \dfrac{\partial \boldsymbol{%
\Sigma }\left( \boldsymbol{\theta }\right) }{\partial \theta _{i}}\right) \\
&&+\frac{\tau }{4}\frac{\left \vert \boldsymbol{\Sigma }\left( \boldsymbol{%
\theta }\right) \right \vert ^{-\frac{\tau }{2}}}{\left( 2\pi \right)
^{m\tau /2}}\frac{1}{\left( 1+\tau \right) ^{m/2+1}}trace\left( \boldsymbol{%
\Sigma }\left( \boldsymbol{\theta }\right) ^{-1}\dfrac{\partial \boldsymbol{%
\Sigma }\left( \boldsymbol{\theta }\right) }{\partial \theta _{j}}\right)
trace\left( \boldsymbol{\Sigma }\left( \boldsymbol{\theta }\right) ^{-1}%
\dfrac{\partial \boldsymbol{\Sigma }\left( \boldsymbol{\theta }\right) }{%
\partial \theta _{i}}\right) .
\end{eqnarray*}%
Therefore,

\begin{equation*}
C_{2}=L_{1}^{\ast }+L_{2}^{\ast }+L_{3}^{\ast }+L_{4}^{\ast }\underset{%
n\longrightarrow \infty }{\overset{\mathcal{P}}{\longrightarrow }}R
\end{equation*}%
being 
\begin{eqnarray*}
R &=&\frac{\tau \left \vert \boldsymbol{\Sigma }\left( \boldsymbol{\theta }%
\right) \right \vert ^{-\frac{\tau }{2}}}{\left( 2\pi \right) ^{m\tau
/2}\left( 1+\tau \right) ^{m/2}}trace\left( \boldsymbol{\Sigma }\left( 
\boldsymbol{\theta }\right) ^{-1}\dfrac{\partial \boldsymbol{\mu }\left( 
\boldsymbol{\theta }\right) }{\partial \theta _{j}}\left( \dfrac{\partial 
\boldsymbol{\mu }\left( \boldsymbol{\theta }\right) }{\partial \theta _{j}}%
\right) ^{T}\right) \\
&&+\frac{\tau }{2}\frac{\left \vert \boldsymbol{\Sigma }\left( \boldsymbol{%
\theta }\right) \right \vert ^{-\frac{\tau }{2}}}{\left( 2\pi \right)
^{m\tau /2}\left( 1+\tau \right) ^{m/2+1}}trace\left( \boldsymbol{\Sigma }%
\left( \boldsymbol{\theta }\right) ^{-1}\dfrac{\partial \boldsymbol{\Sigma }%
\left( \boldsymbol{\theta }\right) }{\partial \theta _{j}}\boldsymbol{\Sigma 
}\left( \boldsymbol{\theta }\right) ^{-1}\dfrac{\partial \boldsymbol{\Sigma }%
\left( \boldsymbol{\theta }\right) }{\partial \theta _{i}}\right) \\
&&+\frac{\tau }{4}\frac{\left \vert \boldsymbol{\Sigma }\left( \boldsymbol{%
\theta }\right) \right \vert ^{-\frac{\tau }{2}}}{\left( 2\pi \right)
^{m\tau /2}\left( 1+\tau \right) ^{m/2+1}}trace\left( \boldsymbol{\Sigma }%
\left( \boldsymbol{\theta }\right) ^{-1}\dfrac{\partial \boldsymbol{\Sigma }%
\left( \boldsymbol{\theta }\right) }{\partial \theta _{j}}\right)
trace\left( \boldsymbol{\Sigma }\left( \boldsymbol{\theta }\right) ^{-1}%
\dfrac{\partial \boldsymbol{\Sigma }\left( \boldsymbol{\theta }\right) }{%
\partial \theta _{i}}\right) .
\end{eqnarray*}%
Finally,%
\begin{eqnarray*}
-C_{3} &=&a\frac{\tau }{2}\left \vert \boldsymbol{\Sigma }\left( \boldsymbol{%
\theta }\right) \right \vert ^{-\frac{\tau }{2}}\frac{1}{n}\tsum
\limits_{i=1}^{n}\exp \left \{ -\frac{\tau }{2}\left( \boldsymbol{y}_{i}-%
\boldsymbol{\mu }\left( \boldsymbol{\theta }\right) \right) ^{T}\boldsymbol{%
\Sigma }\left( \boldsymbol{\theta }\right) ^{-1}\left( \boldsymbol{y}_{i}-%
\boldsymbol{\mu }\left( \boldsymbol{\theta }\right) \right) \right \} \\
&&\left \{ \left( 2\frac{\partial }{\partial \theta _{j}}\left( \frac{%
\partial \boldsymbol{\mu }\left( \boldsymbol{\theta }\right) }{\partial
\theta _{i}}\right) ^{T}\boldsymbol{\Sigma }\left( \boldsymbol{\theta }%
\right) ^{-1}\left( \boldsymbol{y}_{i}-\boldsymbol{\mu }\left( \boldsymbol{%
\theta }\right) \right) \right) \right. \\
&&\left. +2\left( \frac{\partial \boldsymbol{\mu }\left( \boldsymbol{\theta }%
\right) }{\partial \theta _{i}}\right) ^{T}\dfrac{\partial \boldsymbol{%
\Sigma }\left( \boldsymbol{\theta }\right) ^{-1}}{\partial \theta _{i}}%
\left( \boldsymbol{y}_{i}-\boldsymbol{\mu }\left( \boldsymbol{\theta }%
\right) \right) \right. \\
&&\left. -2\left( \frac{\partial \boldsymbol{\mu }\left( \boldsymbol{\theta }%
\right) }{\partial \theta _{i}}\right) ^{T}\boldsymbol{\Sigma }\left( 
\boldsymbol{\theta }\right) ^{-1}\left( \frac{\partial \boldsymbol{\mu }%
\left( \boldsymbol{\theta }\right) }{\partial \theta _{j}}\right) ^{T}\right.
\\
&&\left. -2\left( \frac{\partial \boldsymbol{\mu }\left( \boldsymbol{\theta }%
\right) }{\partial \theta _{i}}\right) ^{T}\left( \boldsymbol{\Sigma }\left( 
\boldsymbol{\theta }\right) ^{-1}\dfrac{\partial \boldsymbol{\Sigma }\left( 
\boldsymbol{\theta }\right) }{\partial \theta _{i}}\boldsymbol{\Sigma }%
\left( \boldsymbol{\theta }\right) ^{-1}\right) \left( \boldsymbol{y}_{i}-%
\boldsymbol{\mu }\left( \boldsymbol{\theta }\right) \right) \right. \\
&&\left. +\left( \boldsymbol{y}_{i}-\boldsymbol{\mu }\left( \boldsymbol{%
\theta }\right) \right) ^{T}\left \{ -\right. \boldsymbol{\Sigma }\left( 
\boldsymbol{\theta }\right) ^{-1}\dfrac{\partial \boldsymbol{\Sigma }\left( 
\boldsymbol{\theta }\right) }{\partial \theta _{j}}\boldsymbol{\Sigma }%
\left( \boldsymbol{\theta }\right) ^{-1}\dfrac{\partial \boldsymbol{\Sigma }%
\left( \boldsymbol{\theta }\right) }{\partial \theta _{i}}\boldsymbol{\Sigma 
}\left( \boldsymbol{\theta }\right) ^{-1}\right. \\
&&\left. +\boldsymbol{\Sigma }\left( \boldsymbol{\theta }\right) ^{-1}\dfrac{%
\partial ^{2}\boldsymbol{\Sigma }\left( \boldsymbol{\theta }\right) }{%
\partial \theta _{i}\partial \theta _{j}}\boldsymbol{\Sigma }\left( 
\boldsymbol{\theta }\right) ^{-1}-\boldsymbol{\Sigma }\left( \boldsymbol{%
\theta }\right) ^{-1}\dfrac{\partial \boldsymbol{\Sigma }\left( \boldsymbol{%
\theta }\right) }{\partial \theta _{j}}\boldsymbol{\Sigma }\left( 
\boldsymbol{\theta }\right) ^{-1}\dfrac{\partial \boldsymbol{\Sigma }\left( 
\boldsymbol{\theta }\right) }{\partial \theta _{i}}\boldsymbol{\Sigma }%
\left( \boldsymbol{\theta }\right) ^{-1}\right \} \left( \boldsymbol{y}_{i}-%
\boldsymbol{\mu }\left( \boldsymbol{\theta }\right) \right) \\
&=&A_{1}^{\ast }+A_{2}^{\ast }+A_{3}^{\ast }+A_{4}^{\ast }+A_{5}^{\ast }.
\end{eqnarray*}%
It is clear that 
\begin{equation*}
A_{1}^{\ast }\underset{n\longrightarrow \infty }{\overset{\mathcal{P}}{%
\longrightarrow }}0\text{ }A_{2}^{\ast }\underset{n\longrightarrow \infty }{%
\overset{\mathcal{P}}{\longrightarrow }}0\text{ and }A_{4}^{\ast }\underset{%
n\longrightarrow \infty }{\overset{\mathcal{P}}{\longrightarrow }}0.
\end{equation*}%
On the other hand,%
\begin{eqnarray*}
A_{3}^{\ast } &=&-a\frac{\tau }{2}\left \vert \boldsymbol{\Sigma }\left( 
\boldsymbol{\theta }\right) \right \vert ^{-\frac{\tau }{2}}\frac{1}{n}\tsum
\limits_{i=1}^{n}\exp \left \{ -\frac{\tau }{2}\left( \boldsymbol{y}_{i}-%
\boldsymbol{\mu }\left( \boldsymbol{\theta }\right) \right) ^{T}\boldsymbol{%
\Sigma }\left( \boldsymbol{\theta }\right) ^{-1}\left( \boldsymbol{y}_{i}-%
\boldsymbol{\mu }\left( \boldsymbol{\theta }\right) \right) \right \} \\
&&\left( 2\left( \frac{\partial \boldsymbol{\mu }\left( \boldsymbol{\theta }%
\right) }{\partial \theta _{i}}\right) ^{T}\boldsymbol{\Sigma }\left( 
\boldsymbol{\theta }\right) ^{-1}\frac{\partial \boldsymbol{\mu }\left( 
\boldsymbol{\theta }\right) }{\partial \theta _{i}}\right)
\end{eqnarray*}%
and 
\begin{equation*}
A_{3}^{\ast }\underset{n\longrightarrow \infty }{\overset{\mathcal{P}}{%
\longrightarrow }}-(\tau +1)\frac{\left \vert \boldsymbol{\Sigma }\left( 
\boldsymbol{\theta }\right) \right \vert ^{-\frac{\tau }{2}}}{\left( 2\pi
\right) ^{m\tau /2}\left( 1+\tau \right) ^{m/2}}\left( \frac{\partial 
\boldsymbol{\mu }\left( \boldsymbol{\theta }\right) }{\partial \theta _{i}}%
\right) ^{T}\boldsymbol{\Sigma }\left( \boldsymbol{\theta }\right) ^{-1}%
\frac{\partial \boldsymbol{\mu }\left( \boldsymbol{\theta }\right) }{%
\partial \theta _{j}}.
\end{equation*}%
In relation to $A_{5}^{\ast }$ we have, 
\begin{eqnarray*}
A_{5}^{\ast } &=&a\frac{\tau }{2}\left \vert \boldsymbol{\Sigma }\left( 
\boldsymbol{\theta }\right) \right \vert ^{-\frac{\tau }{2}}\frac{1}{n}\tsum
\limits_{i=1}^{n}\exp \left \{ -\frac{\tau }{2}\left( \boldsymbol{y}_{i}-%
\boldsymbol{\mu }\left( \boldsymbol{\theta }\right) \right) ^{T}\boldsymbol{%
\Sigma }\left( \boldsymbol{\theta }\right) ^{-1}\left( \boldsymbol{y}_{i}-%
\boldsymbol{\mu }\left( \boldsymbol{\theta }\right) \right) \right \} \\
&&\left \{ \left( \boldsymbol{y}_{i}-\boldsymbol{\mu }\left( \boldsymbol{%
\theta }\right) \right) ^{T}\left[ \boldsymbol{\Sigma }\left( \boldsymbol{%
\theta }\right) ^{-1}\dfrac{\partial \boldsymbol{\Sigma }\left( \boldsymbol{%
\theta }\right) }{\partial \theta _{j}}\boldsymbol{\Sigma }\left( 
\boldsymbol{\theta }\right) ^{-1}\dfrac{\partial \boldsymbol{\Sigma }\left( 
\boldsymbol{\theta }\right) }{\partial \theta _{i}}\boldsymbol{\Sigma }%
\left( \boldsymbol{\theta }\right) ^{-1}\right. \right. \\
&&\left. \left. +\boldsymbol{\Sigma }\left( \boldsymbol{\theta }\right) ^{-1}%
\dfrac{\partial ^{2}\boldsymbol{\Sigma }\left( \boldsymbol{\theta }\right) }{%
\partial \theta _{i}\partial \theta _{j}}\boldsymbol{\Sigma }\left( 
\boldsymbol{\theta }\right) ^{-1}-\boldsymbol{\Sigma }\left( \boldsymbol{%
\theta }\right) ^{-1}\dfrac{\partial \boldsymbol{\Sigma }\left( \boldsymbol{%
\theta }\right) }{\partial \theta _{j}}\boldsymbol{\Sigma }\left( 
\boldsymbol{\theta }\right) ^{-1}\dfrac{\partial \boldsymbol{\Sigma }\left( 
\boldsymbol{\theta }\right) }{\partial \theta _{j}}\boldsymbol{\Sigma }%
\left( \boldsymbol{\theta }\right) ^{-1}\right] \left( \boldsymbol{y}_{i}-%
\boldsymbol{\mu }\left( \boldsymbol{\theta }\right) \right) \right \} ,
\end{eqnarray*}%
and 
\begin{eqnarray*}
&&A_{5}^{\ast }\underset{n\longrightarrow \infty }{\overset{\mathcal{P}}{%
\longrightarrow }}-\frac{1}{\left( 2\pi \right) ^{m\tau /2}}\frac{1}{2}\left
\vert \boldsymbol{\Sigma }\left( \boldsymbol{\theta }\right) \right \vert ^{-%
\frac{\tau }{2}}\frac{1}{\left( 1+\tau \right) ^{m/2}}trace\left( 
\boldsymbol{\Sigma }\left( \boldsymbol{\theta }\right) ^{-1}\dfrac{\partial 
\boldsymbol{\Sigma }\left( \boldsymbol{\theta }\right) }{\partial \theta _{j}%
}\boldsymbol{\Sigma }\left( \boldsymbol{\theta }\right) ^{-1}\dfrac{\partial 
\boldsymbol{\Sigma }\left( \boldsymbol{\theta }\right) }{\partial \theta _{i}%
}\right) \\
&&+\frac{1}{\left( 2\pi \right) ^{m\tau /2}}\frac{1}{2}\left \vert 
\boldsymbol{\Sigma }\left( \boldsymbol{\theta }\right) \right \vert ^{-\frac{%
\tau }{2}}\frac{1}{\left( 1+\tau \right) ^{m/2}}trace\left( \boldsymbol{%
\Sigma }\left( \boldsymbol{\theta }\right) ^{-1}\dfrac{\partial ^{2}%
\boldsymbol{\Sigma }\left( \boldsymbol{\theta }\right) }{\partial \theta
_{i}\partial \theta }\right) \\
&&-\frac{1}{\left( 2\pi \right) ^{m\tau /2}}\frac{1}{2}\left \vert 
\boldsymbol{\Sigma }\left( \boldsymbol{\theta }\right) \right \vert ^{-\frac{%
\tau }{2}}\frac{1}{\left( 1+\tau \right) ^{m/2}}trace\left( \boldsymbol{%
\Sigma }\left( \boldsymbol{\theta }\right) ^{-1}\dfrac{\partial \boldsymbol{%
\Sigma }\left( \boldsymbol{\theta }\right) }{\partial \theta _{j}}%
\boldsymbol{\Sigma }\left( \boldsymbol{\theta }\right) ^{-1}\dfrac{\partial 
\boldsymbol{\Sigma }\left( \boldsymbol{\theta }\right) }{\partial \theta _{i}%
}\right) .
\end{eqnarray*}%
Therefore,%
\begin{eqnarray*}
&&C_{3}\underset{n\longrightarrow \infty }{\overset{\mathcal{P}}{%
\longrightarrow }}-\left( \tau +1\right) \frac{\left \vert \boldsymbol{%
\Sigma }\left( \boldsymbol{\theta }\right) \right \vert ^{-\frac{\tau }{2}}}{%
\left( 2\pi \right) ^{m\tau /2}}\frac{1}{\left( 1+\tau \right) ^{m/2}}\left( 
\frac{\partial \boldsymbol{\mu }\left( \boldsymbol{\theta }\right) }{%
\partial \theta _{i}}\right) ^{T}\boldsymbol{\Sigma }\left( \boldsymbol{%
\theta }\right) ^{-1}\frac{\partial \boldsymbol{\mu }\left( \boldsymbol{%
\theta }\right) }{\partial \theta _{j}} \\
&&-\frac{\left \vert \boldsymbol{\Sigma }\left( \boldsymbol{\theta }\right)
\right \vert ^{-\frac{\tau }{2}}}{\left( 2\pi \right) ^{m\tau /2}}\frac{1}{%
\left( 1+\tau \right) ^{m/2}}\frac{1}{2}trace\left( \boldsymbol{\Sigma }%
\left( \boldsymbol{\theta }\right) ^{-1}\dfrac{\partial \boldsymbol{\Sigma }%
\left( \boldsymbol{\theta }\right) }{\partial \theta _{j}}\boldsymbol{\Sigma 
}\left( \boldsymbol{\theta }\right) ^{-1}\dfrac{\partial \boldsymbol{\Sigma }%
\left( \boldsymbol{\theta }\right) }{\partial \theta _{i}}\right) \\
&&+\frac{\left \vert \boldsymbol{\Sigma }\left( \boldsymbol{\theta }\right)
\right \vert ^{-\frac{\tau }{2}}}{\left( 2\pi \right) ^{m\tau /2}}\frac{1}{%
\left( 1+\tau \right) ^{m/2}}\frac{1}{2}trace\left( \boldsymbol{\Sigma }%
\left( \boldsymbol{\theta }\right) ^{-1}\dfrac{\partial ^{2}\boldsymbol{%
\Sigma }\left( \boldsymbol{\theta }\right) }{\partial \theta _{i}\partial
\theta _{j}}\right) \\
&&-\frac{\left \vert \boldsymbol{\Sigma }\left( \boldsymbol{\theta }\right)
\right \vert ^{-\frac{\tau }{2}}}{\left( 2\pi \right) ^{m\tau /2}}\frac{1}{2}%
\frac{1}{\left( 1+\tau \right) ^{m/2}}trace\left( \boldsymbol{\Sigma }\left( 
\boldsymbol{\theta }\right) ^{-1}\dfrac{\partial \boldsymbol{\Sigma }\left( 
\boldsymbol{\theta }\right) }{\partial \theta _{j}}\boldsymbol{\Sigma }%
\left( \boldsymbol{\theta }\right) ^{-1}\dfrac{\partial \boldsymbol{\Sigma }%
\left( \boldsymbol{\theta }\right) }{\partial \theta _{i}}\right) .
\end{eqnarray*}

We are going to joint all the previous expressions in order to get $\frac{%
\partial }{\partial \theta _{j}}L_{2}^{\tau }(\boldsymbol{\theta }),$%
\begin{eqnarray*}
\frac{\partial }{\partial \theta _{j}}L_{2}^{\tau }(\boldsymbol{\theta })
&=&C_{1}+C_{2}+C_{3} \\
&=&C_{1}+L_{1}^{\ast }+L_{2}^{\ast }+L_{3}^{\ast }+L_{4}^{\ast }+C_{3} \\
&=&C_{1}+L_{1}^{\ast }+L_{2}^{\ast }+L_{3}^{\ast }+L_{4}^{\ast }+A_{1}^{\ast
}+A_{2}^{\ast }+A_{3}^{\ast }+A_{4}^{\ast }+A_{5}^{\ast }.
\end{eqnarray*}%
Then,%
\begin{eqnarray*}
&&\frac{\partial }{\partial \theta _{j}}L_{2}^{\tau }(\boldsymbol{\theta })%
\underset{n\longrightarrow \infty }{\overset{\mathcal{P}}{\longrightarrow }}-%
\frac{\tau }{4}\frac{1}{\left( 1+\tau \right) ^{m/2}}\frac{\left \vert 
\boldsymbol{\Sigma }\left( \boldsymbol{\theta }\right) \right \vert ^{-\frac{%
\tau }{2}}}{\left( 2\pi \right) ^{m\tau /2}}trace\left( \boldsymbol{\Sigma }%
\left( \boldsymbol{\theta }\right) ^{-1}\dfrac{\partial \boldsymbol{\Sigma }%
\left( \boldsymbol{\theta }\right) }{\partial \theta _{i}}\right)
trace\left( \boldsymbol{\Sigma }\left( \boldsymbol{\theta }\right) ^{-1}%
\dfrac{\partial \boldsymbol{\Sigma }\left( \boldsymbol{\theta }\right) }{%
\partial \theta _{j}}\right) \\
&&+\frac{\tau \left \vert \boldsymbol{\Sigma }\left( \boldsymbol{\theta }%
\right) \right \vert ^{-\frac{\tau }{2}}}{\left( 2\pi \right) ^{m\tau
/2}\left( 1+\tau \right) ^{m/2}}trace\left( \boldsymbol{\Sigma }\left( 
\boldsymbol{\theta }\right) ^{-1}\dfrac{\partial \boldsymbol{\Sigma }\left( 
\boldsymbol{\theta }\right) }{\partial \theta _{j}}\boldsymbol{\Sigma }%
\left( \boldsymbol{\theta }\right) ^{-1}\dfrac{\partial \boldsymbol{\Sigma }%
\left( \boldsymbol{\theta }\right) }{\partial \theta _{i}}\right) \\
&&+\frac{\tau }{2}\frac{\left \vert \boldsymbol{\Sigma }\left( \boldsymbol{%
\theta }\right) \right \vert ^{-\frac{\tau }{2}}}{\left( 2\pi \right)
^{m\tau /2}\left( 1+\tau \right) ^{\frac{m}{2}+1}}trace\left( \boldsymbol{%
\Sigma }\left( \boldsymbol{\theta }\right) ^{-1}\dfrac{\partial \boldsymbol{%
\Sigma }\left( \boldsymbol{\theta }\right) }{\partial \theta _{j}}%
\boldsymbol{\Sigma }\left( \boldsymbol{\theta }\right) ^{-1}\dfrac{\partial 
\boldsymbol{\Sigma }\left( \boldsymbol{\theta }\right) }{\partial \theta _{i}%
}\right) \\
&&+\frac{\tau }{4}\frac{\left \vert \boldsymbol{\Sigma }\left( \boldsymbol{%
\theta }\right) \right \vert ^{-\frac{\tau }{2}}}{\left( 2\pi \right)
^{m\tau /2}\left( 1+\tau \right) ^{m/2+1}}trace\left( \boldsymbol{\Sigma }%
\left( \boldsymbol{\theta }\right) ^{-1}\dfrac{\partial \boldsymbol{\Sigma }%
\left( \boldsymbol{\theta }\right) }{\partial \theta _{j}}\right)
trace\left( \boldsymbol{\Sigma }\left( \boldsymbol{\theta }\right) ^{-1}%
\dfrac{\partial \boldsymbol{\Sigma }\left( \boldsymbol{\theta }\right) }{%
\partial \theta _{i}}\right) \\
&&-\left( \tau +1\right) \frac{\left \vert \boldsymbol{\Sigma }\left( 
\boldsymbol{\theta }\right) \right \vert ^{-\frac{\tau }{2}}}{\left( 2\pi
\right) ^{m\tau /2}}\frac{1}{\left( 1+\tau \right) ^{m/2}}\left( \frac{%
\partial \boldsymbol{\mu }\left( \boldsymbol{\theta }\right) }{\partial
\theta _{i}}\right) ^{T}\boldsymbol{\Sigma }\left( \boldsymbol{\theta }%
\right) ^{-1}\frac{\partial \boldsymbol{\mu }\left( \boldsymbol{\theta }%
\right) }{\partial \theta _{j}} \\
&&-\frac{\left \vert \boldsymbol{\Sigma }\left( \boldsymbol{\theta }\right)
\right \vert ^{-\frac{\tau }{2}}}{\left( 2\pi \right) ^{m\tau /2}}\frac{1}{%
\left( 1+\tau \right) ^{m/2}}\frac{1}{2}trace\left( \boldsymbol{\Sigma }%
\left( \boldsymbol{\theta }\right) ^{-1}\dfrac{\partial \boldsymbol{\Sigma }%
\left( \boldsymbol{\theta }\right) }{\partial \theta _{j}}\boldsymbol{\Sigma 
}\left( \boldsymbol{\theta }\right) ^{-1}\dfrac{\partial \boldsymbol{\Sigma }%
\left( \boldsymbol{\theta }\right) }{\partial \theta _{i}}\right) \\
&&+\frac{\left \vert \boldsymbol{\Sigma }\left( \boldsymbol{\theta }\right)
\right \vert ^{-\frac{\tau }{2}}}{\left( 2\pi \right) ^{m\tau /2}}\frac{1}{%
\left( 1+\tau \right) ^{m/2}}\frac{1}{2}trace\left( \boldsymbol{\Sigma }%
\left( \boldsymbol{\theta }\right) ^{-1}\dfrac{\partial ^{2}\boldsymbol{%
\Sigma }\left( \boldsymbol{\theta }\right) }{\partial \theta _{i}\partial
\theta _{j}}\right) \\
&&-\frac{\left \vert \boldsymbol{\Sigma }\left( \boldsymbol{\theta }\right)
\right \vert ^{-\frac{\tau }{2}}}{\left( 2\pi \right) ^{m\tau /2}}\frac{1}{2}%
\frac{1}{\left( 1+\tau \right) ^{m/2}}trace\left( \boldsymbol{\Sigma }\left( 
\boldsymbol{\theta }\right) ^{-1}\dfrac{\partial \boldsymbol{\Sigma }\left( 
\boldsymbol{\theta }\right) }{\partial \theta _{j}}\boldsymbol{\Sigma }%
\left( \boldsymbol{\theta }\right) ^{-1}\dfrac{\partial \boldsymbol{\Sigma }%
\left( \boldsymbol{\theta }\right) }{\partial \theta _{i}}\right) .
\end{eqnarray*}%
Based on the previous results we have%
\begin{eqnarray*}
\frac{\partial ^{2}}{\partial \theta _{i}\partial \theta _{j}}H_{n}^{\tau }(%
\boldsymbol{\theta }) &=&\frac{\partial }{\partial \theta _{j}}L_{1}^{\tau }(%
\boldsymbol{\theta })+\frac{\partial }{\partial \theta _{j}}L_{2}^{\tau }(%
\boldsymbol{\theta }) \\
&=&D_{1}+D_{2}+D_{3}+C_{1}+C_{2}+C_{3} \\
&=&D_{1}+D_{2}+D_{3}+C_{1}+L_{1}^{\ast }+L_{2}^{\ast }+L_{3}^{\ast
}+L_{4}^{\ast } \\
&&+A_{1}^{\ast }+A_{2}^{\ast }+A_{3}^{\ast }+A_{4}^{\ast }+A_{5}^{\ast }
\end{eqnarray*}%
and 
\begin{eqnarray*}
&&\frac{\partial ^{2}}{\partial \theta _{i}\partial \theta _{j}}H_{n}^{\tau
}(\boldsymbol{\theta })\underset{n\longrightarrow \infty }{\overset{\mathcal{%
P}}{\longrightarrow }}\frac{1}{2}\frac{\left \vert \boldsymbol{\Sigma }%
\left( \boldsymbol{\theta }\right) \right \vert ^{-\frac{\tau }{2}}}{\left(
2\pi \right) ^{m\tau /2}}\frac{1}{\left( 1+\tau \right) ^{m/2}}trace\left( 
\boldsymbol{\Sigma }\left( \boldsymbol{\theta }\right) ^{-1}\dfrac{\partial 
\boldsymbol{\Sigma }\left( \boldsymbol{\theta }\right) }{\partial \theta _{j}%
}\boldsymbol{\Sigma }\left( \boldsymbol{\theta }\right) ^{-1}\dfrac{\partial 
\boldsymbol{\Sigma }\left( \boldsymbol{\theta }\right) }{\partial \theta _{i}%
}\right) \\
&&-\frac{1}{2}\frac{\left \vert \boldsymbol{\Sigma }\left( \boldsymbol{%
\theta }\right) \right \vert ^{-\frac{\tau }{2}}}{\left( 2\pi \right)
^{m\tau /2}}\frac{1}{\left( 1+\tau \right) ^{m/2}}\left( \boldsymbol{\Sigma }%
\left( \boldsymbol{\theta }\right) ^{-1}\dfrac{\partial ^{2}\boldsymbol{%
\Sigma }\left( \boldsymbol{\theta }\right) }{\partial \theta _{j}\partial
\theta _{i}}\right) \\
&&-\frac{\tau }{4}\frac{1}{\left( 1+\tau \right) ^{m/2}}\frac{\left \vert 
\boldsymbol{\Sigma }\left( \boldsymbol{\theta }\right) \right \vert ^{-\frac{%
\tau }{2}}}{\left( 2\pi \right) ^{m\tau /2}}trace\left( \boldsymbol{\Sigma }%
\left( \boldsymbol{\theta }\right) ^{-1}\dfrac{\partial \boldsymbol{\Sigma }%
\left( \boldsymbol{\theta }\right) }{\partial \theta _{i}}\right)
trace\left( \boldsymbol{\Sigma }\left( \boldsymbol{\theta }\right) ^{-1}%
\dfrac{\partial \boldsymbol{\Sigma }\left( \boldsymbol{\theta }\right) }{%
\partial \theta _{j}}\right) \\
&&+\frac{\tau \left \vert \boldsymbol{\Sigma }\left( \boldsymbol{\theta }%
\right) \right \vert ^{-\frac{\tau }{2}}}{\left( 2\pi \right) ^{m\tau
/2}\left( 1+\tau \right) ^{m/2}}trace\left( \boldsymbol{\Sigma }\left( 
\boldsymbol{\theta }\right) ^{-1}\dfrac{\partial \boldsymbol{\Sigma }\left( 
\boldsymbol{\theta }\right) }{\partial \theta _{j}}\boldsymbol{\Sigma }%
\left( \boldsymbol{\theta }\right) ^{-1}\dfrac{\partial \boldsymbol{\Sigma }%
\left( \boldsymbol{\theta }\right) }{\partial \theta _{i}}\right) \\
&&+\frac{\tau }{2}\frac{\left \vert \boldsymbol{\Sigma }\left( \boldsymbol{%
\theta }\right) \right \vert ^{-\frac{\tau }{2}}}{\left( 2\pi \right)
^{m\tau /2}\left( 1+\tau \right) ^{m/2+1}}trace\left( \boldsymbol{\Sigma }%
\left( \boldsymbol{\theta }\right) ^{-1}\dfrac{\partial \boldsymbol{\Sigma }%
\left( \boldsymbol{\theta }\right) }{\partial \theta _{j}}\boldsymbol{\Sigma 
}\left( \boldsymbol{\theta }\right) ^{-1}\dfrac{\partial \boldsymbol{\Sigma }%
\left( \boldsymbol{\theta }\right) }{\partial \theta _{i}}\right) \\
&&+\frac{\tau }{4}\frac{\left \vert \boldsymbol{\Sigma }\left( \boldsymbol{%
\theta }\right) \right \vert ^{-\frac{\tau }{2}}}{\left( 2\pi \right)
^{m\tau /2}\left( 1+\tau \right) ^{m/2+1}}trace\left( \boldsymbol{\Sigma }%
\left( \boldsymbol{\theta }\right) ^{-1}\dfrac{\partial \boldsymbol{\Sigma }%
\left( \boldsymbol{\theta }\right) }{\partial \theta _{j}}\right)
trace\left( \boldsymbol{\Sigma }\left( \boldsymbol{\theta }\right) ^{-1}%
\dfrac{\partial \boldsymbol{\Sigma }\left( \boldsymbol{\theta }\right) }{%
\partial \theta _{i}}\right) \\
&&-\left( \tau +1\right) \frac{\left \vert \boldsymbol{\Sigma }\left( 
\boldsymbol{\theta }\right) \right \vert ^{-\frac{\tau }{2}}}{\left( 2\pi
\right) ^{m\tau /2}}\frac{1}{\left( 1+\tau \right) ^{m/2}}\left( \frac{%
\partial \boldsymbol{\mu }\left( \boldsymbol{\theta }\right) }{\partial
\theta _{i}}\right) ^{T}\boldsymbol{\Sigma }\left( \boldsymbol{\theta }%
\right) ^{-1}\frac{\partial \boldsymbol{\mu }\left( \boldsymbol{\theta }%
\right) }{\partial \theta _{j}} \\
&&-\frac{\left \vert \boldsymbol{\Sigma }\left( \boldsymbol{\theta }\right)
\right \vert ^{-\frac{\tau }{2}}}{\left( 2\pi \right) ^{m\tau /2}}\frac{1}{%
\left( 1+\tau \right) ^{m/2}}\frac{1}{2}trace\left( \boldsymbol{\Sigma }%
\left( \boldsymbol{\theta }\right) ^{-1}\dfrac{\partial \boldsymbol{\Sigma }%
\left( \boldsymbol{\theta }\right) }{\partial \theta _{j}}\boldsymbol{\Sigma 
}\left( \boldsymbol{\theta }\right) ^{-1}\dfrac{\partial \boldsymbol{\Sigma }%
\left( \boldsymbol{\theta }\right) }{\partial \theta _{i}}\right) \\
&&+\frac{\left \vert \boldsymbol{\Sigma }\left( \boldsymbol{\theta }\right)
\right \vert ^{-\frac{\tau }{2}}}{\left( 2\pi \right) ^{m\tau /2}}\frac{1}{%
\left( 1+\tau \right) ^{m/2}}\frac{1}{2}trace\left( \boldsymbol{\Sigma }%
\left( \boldsymbol{\theta }\right) ^{-1}\dfrac{\partial ^{2}\boldsymbol{%
\Sigma }\left( \boldsymbol{\theta }\right) }{\partial \theta _{i}\partial
\theta _{j}}\right) \\
&&-\frac{\left \vert \boldsymbol{\Sigma }\left( \boldsymbol{\theta }\right)
\right \vert ^{-\frac{\tau }{2}}}{\left( 2\pi \right) ^{m\tau /2}}\frac{1}{%
\left( 1+\tau \right) ^{m/2}}\frac{1}{2}trace\left( \boldsymbol{\Sigma }%
\left( \boldsymbol{\theta }\right) ^{-1}\dfrac{\partial \boldsymbol{\Sigma }%
\left( \boldsymbol{\theta }\right) }{\partial \theta _{j}}\boldsymbol{\Sigma 
}\left( \boldsymbol{\theta }\right) ^{-1}\dfrac{\partial \boldsymbol{\Sigma }%
\left( \boldsymbol{\theta }\right) }{\partial \theta _{i}}\right) .
\end{eqnarray*}%
After some algebra we have, 
\begin{equation*}
\frac{\partial ^{2}}{\partial \theta _{i}\partial \theta _{j}}H_{n}^{\tau }(%
\boldsymbol{\theta })\underset{n\longrightarrow \infty }{\overset{\mathcal{P}%
}{\longrightarrow }}-\left( \tau +1\right) J_{\tau }^{ij}\left( \boldsymbol{%
\theta }\right) .
\end{equation*}

\end{document}